\documentclass[12pt,a4paper,reqno]{amsart}

\usepackage[lmargin=1.2in, rmargin= 1.2in, tmargin=1.5in]{geometry}

\title{Frobenius templicial modules and the dg-nerve}

\author{Wendy Lowen} 
\address[Wendy Lowen]{Universiteit Antwerpen, Departement Wiskunde, Middelheimcampus,
Middelheimlaan 1,
2020 Antwerp, Belgium}
\email{wendy.lowen@uantwerpen.be}

\author{Arne Mertens}
\address[Arne Mertens]{Universiteit Antwerpen, Departement Wiskunde, Middelheimcampus,
Middelheimlaan 1,
2020 Antwerp, Belgium}
\email{arne.mertens@uantwerpen.be}

\thanks{
This project has received funding from the European Research Council (ERC) under the European Union’s Horizon 2020 research and innovation programme (grant agreement No. 817762).\\
The second named author was a predoctoral fellow of the Research Foundation -  Flanders (FWO), file number 1137921N. The current paper is based on part of the resulting PhD thesis \cite{mertens2022templicial}.
}

\makeatletter
\@namedef{subjclassname@2020}{
  \textup{2022} Mathematics Subject Classification}
\makeatother



\usepackage{graphicx}
\usepackage[english]{babel}
\usepackage{amsmath}
\usepackage{amssymb}
\usepackage{amsthm}
\usepackage[all,cmtip]{xy}
\usepackage{cancel}
\usepackage[alpine]{ifsym}
\usepackage{enumerate}
\usepackage{stmaryrd}
\usepackage{tikz}
\usepackage{tikz-cd}
\usetikzlibrary{matrix}
\usepackage[toc,page]{appendix}
\usepackage{hyperref}
\hypersetup{breaklinks=true}
\usepackage[anythingbreaks]{breakurl}

\usepackage[style=alphabetic, maxnames = 50, maxalphanames=50, urldate=long]{biblatex}
\DeclareMathOperator{\Ob}{Ob}

\DeclareMathOperator{\id}{id}

\DeclareMathOperator{\Fun}{Fun}

\DeclareMathOperator*{\colim}{colim}

\DeclareMathOperator{\Set}{Set}

\DeclareMathOperator{\Ab}{Ab}
\DeclareMathOperator{\Mod}{Mod}

\DeclareMathOperator{\Ch}{Ch}

\DeclareMathOperator{\Lin}{Lin}
\DeclareMathOperator{\SSet}{SSet}

\DeclareMathOperator{\Quiv}{Quiv}
\DeclareMathOperator{\Cat}{Cat}

\newcommand{\fint}{\mathbf{\Delta}_{f}} 
\newcommand{\simp}{\mathbf{\Delta}} 
\newcommand{\asimp}{\mathbf{\Delta}_{+}} 
\newcommand{\nec}{\mathcal{N}ec} 
\newcommand{\ts}{S_{\otimes}} 
\newcommand{\Fs}{S^{Frob}_{\otimes}} 

\newtheorem{Thm}{Theorem}[section]
\newtheorem*{Thm*}{Theorem}
\newtheorem{Lem}[Thm]{Lemma}
\newtheorem{Prop}[Thm]{Proposition}
\newtheorem{Cor}[Thm]{Corollary}

\theoremstyle{definition}
\newtheorem{Def}[Thm]{Definition}
\newtheorem{Ex}[Thm]{Example}

\newtheorem{Con}[Thm]{Construction}
\newtheorem{Not}[Thm]{Notation}

\theoremstyle{remark}
\newtheorem{Rem}[Thm]{Remark}

\begin{document}

\begin{abstract}
Templicial objects were put forth in \cite{lowen2023enriched} to set up a suitable simplicial framework for enriched quasi-categories. Following Leinster \cite{leinster2000homotopy}, these objects feature certain comultiplications as a replacement for outer face maps in the non-cartesian case. In the present paper, we consider Frobenius templicial objects, thus re-introducing multiplications into the picture. When enriching over $k$-modules for a commutative ring $k$, we prove an equivalence of categories between (homologically) positively graded dg-categories on the one hand and Frobenius templicial modules on the other hand. This equivalence yields a natural enrichment of the dg-nerve  \cite{block2009Riemann} \cite{lurie2016higher}, turning dg-categories into quasi-categories \emph{in modules}. Assuming a projectivity condition, we further prove that a templicial module is a quasi-category in modules precisely when it can be equipped with a \emph{nonassociative} Frobenius structure.
\end{abstract}

\maketitle

\tableofcontents

\section{Introduction}\label{section: Introduction}

\subsection{Motivation}

Ever since the Grothendieck school revolutionised the subject of algebraic geometry, the importance of derived and triangulated categories in this field as well as in neighbouring fields like representation theory or the theory of D-modules, has continued to grow. This eventually gave rise to the development of various flavours of what one might call ``categorical algebraic geometry’’, characterised on the one hand by the idea of putting categorical models for ``spaces’’ centre stage, and on the other hand by the realisation that in general, these models should be of a more algebraic nature than ordinary categories in order to ``do geometry’’ with them. One instance of this principle originated from noncommutative ring theory in the work of Van den Bergh and others and became known as ``noncommutative algebraic geometry’’ (NCAG) \cite{artintatevandenbergh}, \cite{stafford2001nc} \cite{kuznetsov2007homological} \cite{spenko2017nc}. Another instance is ``derived algebraic geometry’’ (DAG) in the sense of To\"en and Vezzosi \cite{toen2008homotopicalII} \cite{toen2014derived} and Lurie \cite{lurie2016higher}. 

Both in NCAG and in DAG, inspired by appropriate (enhanced, derived) categories of quasi-coherent sheaves, suitable dg-categories eventually emerged as appropriate models for spaces \cite{bondalkapranov}, \cite{keller2006differential}. In particular, these models have allowed for the development of noncommutative Hodge theory and a theory of noncommutative motives by Keller, Kontsevich, Tabuada and others \cite{keller1999cyclic} \cite{katzarkov2008hodge} \cite{tabuada2008higher} \cite{cisinski2014lefschetz} \cite{robalo2015ktheory}. 
The notion that dg-categories are indeed geometric objects which at the same time belong to the realm of algebraic topology is exemplified by Kaledin’s proof of Hodge-De Rham degeneration for smooth proper dg-algebras \cite{kontsevich2009notes} \cite{KaledinHdR} \cite{kaledin2017spectral} \cite{Kal2}.
In the wake of these developments, the question of existence and uniqueness of dg-enhancements has become centrally important \cite{bondal2003generators} \cite{canonaco2017tour} \cite{rizzardo2019example}.

Dg-categories being categories enriched in the symmetric monoidal model category of chain complexes over a commutative ring $k$, their category $k\Cat_{dg}$ can be endowed both with a monoidal structure, and with a model structure due to Tabuada \cite{tabuada2005structure}.
This model structure is however not \emph{model monoidal}, which makes it in some respects less convenient to work with. For instance, negotiating this shortcoming, in \cite{toen2007homotopy} To\"en has shown by hand that the simplicial localisation of $k\Cat_{dg}$ gives rise to a closed monoidal structure on the homotopy category, a crucial ingredient in the development of homotopy Morita theory for dg-categories (see also \cite{canonaco2015internal}).

In higher category theory, one encounters an analogous situation, where the category of simplicial categories with the Bergner model structure (the ``strict model’’) is not model monoidal, and here this is remedied by turning to several Quillen equivalent weak models - notably Segal categories, complete Segal spaces and quasi-categories \cite{bergner2007model} \cite{hirschowitz1998descente} \cite{rezk2001model} \cite{joyal2002quasi}.
This has inspired several attempts to come up with alternative ``weak models’’ for dg-categories as well. Recently, a model of ``dg-Segal spaces'' was put forth by Dimitriadis Bermejo \cite{dimitriadisbermejo2023model}. Earlier on, building on work by Simpson and Leinster, ``dg-Segal categories'' have been investigated by Bacard \cite{bacard2010Segal} \cite{simpson2012homotopy} \cite{leinster2000homotopy}. 
The present paper is part of a project in which the main goal is to establish a model of ``quasi-categories in modules’’ realising weak enrichment in simplicial modules.
Approaches to linearity and more general enrichment of higher categories of a different flavour are due to Lurie \cite{lurie2016higher} and Gepner-Haugseng \cite{gepner2015enriched}.

Apart from the general motivation of completing the picture of the enriched higher categorical landscape we just sketched, our particular interest in quasi-categories stems among others from algebraic deformation theory in the sense of Gerstenhaber \cite{gerstenhaber1} 
\cite{gerstenhaber1988cohomology}. 
In NCAG, the choice of a particular algebraic model for a space comes with an associated deformation theory, shedding its own light upon the notion of a noncommutative space. For instance, graded algebras, and more generally $\mathbb{Z}$-algebras in projective geometry give rise to the famous noncommutative planes and quadrics under deformation \cite{artintatevandenbergh} \cite{bondalpolishchuk} \cite{VdBquadrics}, whereas deforming the structure sheaf of a scheme naturally leads to prestacks \cite{lowenprestacks} \cite{VdBdefquant} \cite{dinh2023box}. 
 
Following Leinster \cite{leinster2000homotopy}, templicial modules were put forth in \cite{lowen2023enriched} as an appropriate framework for considering \emph{quasi-categories in modules} as a counterpart of the highly succesful quasi-categories ``in sets'' of Boardman-Vogt \cite{boardmanvogt}, Joyal \cite{joyal2002quasi} and Lurie \cite{lurie2009higher}. Although more involved in terms of algebraic structure, these templicial modules are just as tangible as simplicial sets, and in particular they are amenable to algebraic deformation theory. In joint work with Violeta Borges Marques \cite{lowen2023deformations}, we show among other things that the property of being a quasi-category in modules is preserved under infinitesimal deformation, as desired. Further, deformations are governed by a Hochschild type complex currently under investigation.
As such, quasi-categories in modules constitute a new candidate model for noncommutative spaces, with a distinctively higher categorical flavour. Note that from the point of view of Hochschild cohomology, dg-categories deform into weak variants of a rather different type, namely $A_{\infty}$-categories \cite{lefevre} \cite{keller2001introduction}. The present paper will provide some of the tools to investigate the relation between $A_{\infty}$-categories on the one hand and quasi-categories in modules on the other hand from a deformation theoretic perspective.

\subsection{Results}\label{section: Introduction_oud}

Let us now turn to the contents of the present paper. The most direct way of turning a dg-category into a higher category is arguably the construction of the dg-nerve \cite{block2009Riemann} \cite{lurie2016higher} which, more precisely, produces a quasi-category. A priori, linearity is lost in this process, and one way of reintroducing it is by turning to stable linear $\infty$-categories in the sense of Lurie \cite{lurie2016higher}. It has become common knowledge that working with pre-triangulated dg-categories on the one hand and with stable linear $\infty$-categories on the other hand are equivalent - see e.g. \cite{cohn} for a precise statement. 
However, this approach raises a natural question: what if we remove pre-triangulated? For sure, an ordinary category can be linear without being abelian or even additive, as is the case for any non-trivial ring seen as a one object category. In contrast, linearity of a stable $\infty$-category is expressed through the action of a module category and uses at least some aspect of additivity.

The first goal of the present paper is to enrich the dg-nerve taking a different approach, by enhancing the codomain in a way that works in complete generality, regardless of stability.
More precisely, we construct a \emph{templicial dg-nerve} $N^{dg}_{k}$ from the category $k\Cat_{dg}$ of small dg-categories to the category $\ts\Mod(k)$ of templicial modules in the sense of \cite{lowen2023enriched}, which lands in quasi-categories in modules and makes the following diagram commute:
\[\begin{tikzcd}
	{k\Cat_{dg}} & {\ts\Mod(k)} \\
	 & \SSet
	\arrow["{N^{dg}}"', from=1-1, to=2-2]
	\arrow["{N^{dg}_{k}}", from=1-1, to=1-2]
	\arrow["{\tilde{U}}", from=1-2, to=2-2]
\end{tikzcd}\]
Here, $\tilde{U}$ sends a templicial module to its underlying simplicial set (see Proposition \ref{proposition: properties of temp. objs.}). Note that the structure of a templicial module $X$ is considerably richer than that of $\tilde{U}(X)$, and $X_n$ fails to have $\tilde{U}(X)_n$ as underlying set \cite[Example 2.10]{lowen2023enriched}.

The construction of $N^{dg}_{k}$ has two parts. Dg-categories are enriched in chain complexes, and in the first step, we replace chain complexes by \emph{augmented} simplicial modules. To this end we formulate a modified version of the classical Dold-Kan correspondence, which is compatible with the monoidal structures (Proposition \ref{proposition: augmented Dold-Kan correspondence}). 
In the second step, we apply a kind of bar construction to the resulting enriched category (Theorem \ref{theorem: augmented simp. cat. and Frob. temp. obj. equiv.}). This takes values in templicial modules, which are by definition colax monoidal functors on the finite interval category landing in $k$-quivers (see Definition \ref{definition: temp. obj.}). A templicial module $X$ features comultiplication morphisms
$$
\left(\mu_{p,q}: X_{p+q}\rightarrow X_{p}\otimes_{S} X_{q}\right)_{p,q\geq 0},
$$
which can be interpreted as ``pulling a $(p+q)$-simplex apart into outer faces''.\\

\begin{center}
\begin{tikzpicture}
\filldraw
(0,-0.289) circle (1pt)
(0.75,0.577) circle (1pt)
(1.5,-0.289) circle (1pt)
(0.75,-0.7) circle (1pt);
\draw[-latex]
(0,-0.289) -- (0.75,0.577);
\draw[-latex]
(0.75,0.577) -- (1.5,-0.289);
\draw[-latex]
(1.5,-0.289) -- (0.75,-0.7);
\draw[-latex]
(0.75,0.577) -- (0.75,-0.7);
\draw[-latex]
(0,-0.289) -- (0.75,-0.7);
\draw[-latex, dotted]
(0,-0.289) -- (1.5,-0.289);
\fill[fill=gray,opacity=0.4]
(1.5,-0.289) -- (0.75,0.577) -- (0.75,-0.7);
\fill[fill=gray,opacity=0.2]
(0,-0.289) -- (0.75,0.577) -- (0.75,-0.7);

\draw
(0,1) node{};

\draw
(3,0) node{$\underset{\mu_{1,2}}{\longmapsto}$};

\filldraw
(4.5,-0.289) circle (1pt)
(5.25,0.577) circle (1pt)
(6,-0.289) circle (1pt)
(5.25,-0.7) circle (1pt);
\draw[-latex]
(4.5,-0.289) -- (5.25,0.577);
\draw[-latex]
(5.25,0.577) -- (6,-0.289);
\draw[-latex]
(6,-0.289) -- (5.25,-0.7);
\draw[-latex]
(5.25,0.577) -- (5.25,-0.7);
\fill[fill=gray,opacity=0.4]
(6,-0.289) -- (5.25,0.577) -- (5.25,-0.7);
\end{tikzpicture}
\end{center}

Now, the bar type construction also yields multiplication morphisms
$$
Z^{p,q}: X_{p}\otimes_{S} X_{q}\rightarrow X_{p+q}
$$
in the opposite direction. This makes $X$ into a Frobenius monoidal functor in the sense of \cite{day2008note}. Denote by
$$
\Fs\Mod(k)
$$
the category of Frobenius templicial modules. We show (see Proposition \ref{proposition: Frob. temp. obj. and dg-cat. equivalence}):

\begin{Thm*}
The functor $N^{dg}_{k}$ factors through an equivalence of categories
$$
k\Cat_{dg,\geq 0}\simeq \Fs\Mod(k)
$$
followed by the forgetful functor $\Fs\Mod(k)\rightarrow \ts\Mod(k)$.
\end{Thm*}

The second part of the paper is devoted to elucidating the relation between quasi-categories in modules on the one hand and Frobenius templicial modules on the other hand.
To this end, we consider a relaxation of the latter, in which we drop associativity of the multiplications.
In fact, this can be done in the generality of a suitable symmetric monoidal $\mathcal{V}$ (see \S \ref{subsection: Notations and conventions} below), and we will derive some of our results in this more general context.
We call the resulting objects \emph{naF-templicial objects}, where `naF' is short for `nonassociative Frobenius'. Similarly to the comultiplications, the multiplications in a naF-structure can be interpreted as ``filling up simplices joint at a vertex to an entire simplex''.

\begin{center}
\begin{tikzpicture}
\filldraw
(0,-0.289) circle (1pt)
(0.75,0.577) circle (1pt)
(1.5,-0.289) circle (1pt)
(0.75,-0.7) circle (1pt);
\draw[-latex]
(0,-0.289) -- (0.75,0.577);
\draw[-latex]
(0.75,0.577) -- (1.5,-0.289);
\draw[-latex]
(1.5,-0.289) -- (0.75,-0.7);
\draw[-latex]
(0.75,0.577) -- (0.75,-0.7);
\fill[fill=gray,opacity=0.4]
(1.5,-0.289) -- (0.75,0.577) -- (0.75,-0.7);

\draw
(3,0) node{$\underset{Z^{1,2}}{\longmapsto}$};

\filldraw
(4.5,-0.289) circle (1pt)
(5.25,0.577) circle (1pt)
(6,-0.289) circle (1pt)
(5.25,-0.7) circle (1pt);
\draw[-latex]
(4.5,-0.289) -- (5.25,0.577);
\draw[-latex]
(5.25,0.577) -- (6,-0.289);
\draw[-latex]
(6,-0.289) -- (5.25,-0.7);
\draw[-latex]
(5.25,0.577) -- (5.25,-0.7);
\draw[-latex]
(4.5,-0.289) -- (5.25,-0.7);
\draw[-latex, dotted]
(4.5,-0.289) -- (6,-0.289);
\fill[fill=gray,opacity=0.4]
(6,-0.289) -- (5.25,0.577) -- (5.25,-0.7);
\fill[fill=gray,opacity=0.2]
(4.5,-0.289) -- (5.25,0.577) -- (5.25,-0.7);
\end{tikzpicture}
\end{center}

In that respect, a templicial object with a naF-structure is reminiscent of a quasi-category in $\mathcal{V}$. But there are two important differences. Firstly, naF-structures only allow to fill up simplices joint at a vertex - and more generally the necklaces of \cite{baues1980geometry}\cite{dugger2011rigidification} - but not inner horns. Secondly, naF-structures give a specified choice of fillers while quasi-categories only require existence. Nonetheless, both concepts are related, as we now explain.

We will call a templicial object \emph{deg-projective} provided that it has projective inclusions of degenerate simplices (that is, they have the left lifting property with respect to all regular epimorphisms). Deg-projective templicial modules have well-behaved submodules of degenerate and non-degenerate simplices, and according to \cite{lowen2023enriched} they satisfy a version of the Eilenberg-Zilber Lemma.

Our main results can be summarised as follows (appearing in the text below as Proposition \ref{proposition: cofibrant fibrant temp. obj. has naF-structure} and Theorem \ref{theorem: linear naF-structure implies fibrant}):

\begin{Thm*}
\begin{enumerate}[1.]
	\item Every deg-projective quasi-category in $\mathcal{V}$ has a naF-structure.
	\item Taking $\mathcal{V} = \Mod(k)$, every naF-templicial module is a quasi-category in modules.
\end{enumerate}	
\end{Thm*}

The second of these results makes essential use of the additivity of $\Mod(k)$, using alternating sums of the multiplication maps of the naF-structure to fill in horns.

Note that ordinary quasi-categories (in $\mathcal{V} = \Set$) are automatically deg-projective, whence these can always be equipped with a naF-structure.
As a consequence, ordinary quasi-categories are a natural source of quasi-categories in modules by applying the (levelwise) free functor, which preserves naF-structures.

Finally, making use of the quasi-category property, in Proposition \ref{proposition: comparison of linear nerve and dg-nerve} we show that the templicial dg-nerve coincides with the templicial nerve from \cite{lowen2023enriched} when restricted to linear categories (i.e. dg-categories concentrated in degree zero).

\subsection{Future prospects}
To end this introduction, we comment upon work in progress and future prospects.
Let $\mathcal{V}$ be a suitable symmetric monoidal category as in \S \ref{subsection: Notations and conventions}. In \cite{lowen2023enriched}, we constructed a templicial homotopy coherent
nerve
$$
N^{hc}_{\mathcal{V}}: \mathcal{V}\Cat_{\Delta}\rightarrow \ts\mathcal{V}
$$
where $\mathcal{V}\Cat_{\Delta}$ denotes the category of small $S\mathcal{V}$-enriched categories.

Given a dg-category $\mathcal{C}$, there is a classical comparison map $N^{hc}(\mathcal{C}^{\triangle})\rightarrow N^{dg}(\mathcal{C})$ between the homotopy coherent nerve of its associated simplicial category and its dg-nerve. This map is an equivalence of quasi-categories \cite{lurie2016higher} \cite{faonte2017simplicial} and even a trivial fibration \cite[Tag 00SV]{kerodon}. 
In upcoming work by the second named author \cite{mertens2023nerves}, this result will be lifted to the templicial level, making use of a general procedure for generating enriched nerves.

Further we will show the existence of a model structure making $N^{hc}_{\mathcal{V}}$ above into a Quillen equivalence, furnishing $\ts\mathcal{V}$ as a model for categories weakly enriched in $S\mathcal{V}$. This remains work in progress.

Note that as a consequence of truncation in the (templicial) dg-nerve, the higher approach naturally relates templicial modules to dg-categories concentrated in homologically positive degrees. This is in contrast to e.g. \cite{cohn}, which involves unbounded pre-triangulted dg-categories and uses the passage through spectral categories.
In future work, we will relate our approach to the setup of \cite{cohn}. This will involve an appropriate notion of stability for quasi-categories in modules, allowing for a comparison with the stable linear $\infty$-categories of \cite{lurie2016higher}.

\subsection{Notations and conventions}\label{subsection: Notations and conventions}

\begin{enumerate}[1.]
\item Throughout the text, we let $(\mathcal{V},\otimes,I)$ be a fixed bicomplete, symmetric monoidal closed category (i.e. a B\'enabou cosmos in the sense of \cite{street1974elementary}). Up to natural isomorphism, there is a unique colimit preserving functor $F: \Set\rightarrow \mathcal{V}$ such that $F(\{*\}) = I$. This functor is left-adjoint to the forgetful functor $U = \mathcal{V}(I,-): \mathcal{V}\rightarrow \Set$. Endowing $\Set$ with the cartesian monoidal structure, $F$ is strong monoidal and $U$ is lax monoidal. These notations will remain fixed as well.

\item An adjunction between monoidal categories is called \emph{monoidal} if its left-adjoint has a strong monoidal structure. Note that the right-adjoint inherits a lax monoidal structure in this case (see \cite[Proposition 3.84]{aguiar2010monoidal} for example). A monoidal adjunction that is also an equivalence is called a \emph{monoidal equivalence}.

\item Let $(\mathcal{W},\otimes,I)$ be an arbitrary monoidal category. Given a set $S$, we refer to a collection $Q = (Q(a,b))_{a,b\in S}$ with $Q(a,b)\in \mathcal{W}$ as a \emph{$\mathcal{W}$-enriched quiver} with $S$ its set of \emph{vertices}. A \emph{quiver morphism} $f: Q\rightarrow P$ is a collection $(f_{a,b})_{a,b\in S}$ of morphisms $f_{a,b}: Q(a,b)\rightarrow P(a,b)$ in $\mathcal{W}$. $\mathcal{W}$-enriched quivers with a fixed set of vertices $S$ and morphisms between them form a category which we denote by
$$
\mathcal{W}\Quiv_{S}
$$
This category is monoidal with product $\otimes_{S}$ and unit $I_{S}$ defined by
$$
(Q\otimes_{S} P)(a,b) = \coprod_{c\in S}Q(a,c)\otimes P(c,b)\quad \text{and}\quad I_{S}(a,b) =
\begin{cases}
I & \text{if }a = b\\
0 & \text{if }a\neq b
\end{cases}
$$
for all $Q,P\in \mathcal{W}\Quiv_{S}$ and $a,b\in S$.

\item Let $f: S\rightarrow T$ be a map of sets. We have an induced lax monoidal functor $f^{*}: \mathcal{W}\Quiv_{T}\rightarrow \mathcal{W}\Quiv_{S}$ given by $f^{*}(Q)(a,b) = Q(f(a),f(b))$ for all $\mathcal{W}$-enriched quivers $Q$ and $a,b\in S$. The functor $f^{*}$ has a left-adjoint which we denote by $f_{!}: \mathcal{W}\Quiv_{S}\rightarrow \mathcal{W}\Quiv_{T}$. As $f^{*}$ is canonically lax monoidal, $f_{!}$ comes equipped with an induced colax monoidal structure.

\item To relate dg-categories to templicial modules (see \S\ref{subsection: The templicial dg-nerve}), it will be more convenient for us to consider a $\mathcal{W}$-enriched category (or $\mathcal{W}$-category for short) as a pair $(\mathcal{C},\Ob(\mathcal{C}))$ with $\Ob(\mathcal{C})$ its set of objects and $\mathcal{C}$ a monoid in $\mathcal{W}\Quiv_{\Ob(\mathcal{C})}$. Note that this convention implies that the composition in $\mathcal{C}$ is given by a collection of morphisms in $\mathcal{W}$, for all $A,B,C\in \Ob(\mathcal{C})$:
$$
m_{\mathcal{C}}: \mathcal{C}(A,B)\otimes \mathcal{C}(B,C)\rightarrow \mathcal{C}(A,C)
$$
as opposed to the more conventional $\mathcal{C}(B,C)\otimes \mathcal{C}(A,B)\rightarrow \mathcal{C}(A,C)$. A $\mathcal{W}$-functor $\mathcal{C}\rightarrow \mathcal{D}$ is then a pair $(H,f)$ with $f: \Ob(\mathcal{C})\rightarrow \Ob(\mathcal{D})$ a map of sets and $H: \mathcal{C}\rightarrow f^{*}(\mathcal{D})$ a morphism of monoids in $\mathcal{W}\Quiv_{\Ob(\mathcal{C})}$, where we used the lax structure of $f^{*}$. We denote the category of small $\mathcal{W}$-categories and $\mathcal{W}$-functors between them by
$$
\mathcal{W}\Cat
$$

In the case of dg-categories where $\mathcal{W} = \Ch(k)$ is the category of chain complexes, we must be careful about the Koszul sign rule. To avoid confusion, we will never denote the composition $m_{\mathcal{C}}$ of a dg-category $\mathcal{C}$ by $\circ$, but we put
$$
g\circ f\doteqdot (-1)^{p\cdot q}m_{\mathcal{C}}(f\otimes g)
$$
for any $f\in \mathcal{C}_{p}(A,B)$ and $g\in \mathcal{C}_{q}(B,C)$ with $A,B,C\in \Ob(\mathcal{C})$ and $p,q\in \mathbb{Z}$.

\item We will make use of the simplex categories $\fint\subseteq \simp\subseteq \asimp$, where:
\begin{itemize} 
\item $\asimp$ is the \emph{augmented simplex category}. Its objects are the posets $[n] = \{0,...,n\}$ with $n\geq -1$ (where $[-1] = \emptyset$), and its morphisms are the order morphisms $[m]\rightarrow [n]$. 
\item $\simp$ is the ordinary \emph{simplex category}. It is the full subcategory of $\asimp$ spanned by the objects $[n]$ with $n\geq 0$.
\item $\fint$ is the category of \emph{finite intervals}, which is the subcategory of $\simp$ consisting of all morphisms $f: [m]\rightarrow [n]$ that preserve the endpoints, that is $f(0) = 0$ and $f(m) = n$.
\end{itemize}
We denote the usual coface maps by $\delta_{i}: [n-1]\rightarrow [n]$ and the codegeneracy maps by $\sigma_{i}: [n+1]\rightarrow [n]$ (for all $0\leq i\leq n$). Note that each of the categories $\asimp$, $\simp$ and $\fint$ is generated by the coface and codegeneracy maps that they contain.

Unlike $\simp$, $\fint$ and $\asimp$ carry monoidal structures $(+,[0])$ and $(\star,[-1])$ respectively which are given as follows. 
For all $m,n\geq 0$ and $p,q\geq -1$:
$$
[m] + [n] = [m+n]\quad \text{and}\quad [p]\star [q] = [p + q + 1]
$$
For $f: [m]\rightarrow [m']$, $g: [n]\rightarrow [n']$ in $\fint$, and $r: [p]\rightarrow [p']$, $s: [q]\rightarrow [q']$ in $\asimp$:
\begin{align*}
(f + g)(i) &=
\begin{cases}
f(i) & \text{if }i\leq m\\
m' + g(i - m) & \text{if }i\geq m
\end{cases}\\
(r \star s)(i) &=
\begin{cases}
r(i) & \text{if }i\leq p\\
p' + 1 + s(i - p - 1) & \text{if }i > p
\end{cases}
\end{align*}
There is a well-known monoidal equivalence between $\fint^{op}$ and the augmented simplex category $\asimp$, see \cite{joyal1997disks}.

Note that for any morphism $f: [m]\rightarrow [n]$ in $\fint$ and $k,l\geq 0$ such that $k + l = m$, there exist unique morphisms $f_{1}: [k]\rightarrow [p]$ and $f_{2}: [l]\rightarrow [q]$ in $\fint$ such that $f_{1} + f_{2} = f$. Dually, for any morphism $f: [m]\rightarrow [n]$ in $\asimp$ and $p,q\geq -1$ such that $p + q + 1 = n$, there exist unique morphisms $f_{1}: [k]\rightarrow [p]$ and $f_{2}: [l]\rightarrow [q]$ in $\asimp$ such that $f_{1}\star f_{2} = f$.
\end{enumerate}

\vspace{0,3cm}
\noindent \emph{Acknowledgement.}
The authors are grateful to Violeta Borges Marques, Elena Dimitriadis Bermejo, Rune Haugseng, Lander Hermans, Bernhard Keller, Dmitry Kaledin, Tom Leinster, Boris Shoikhet, Carlos Simpson, Bertrand To\"en, Michel Van den Bergh and Ittay Weiss for valuable interactions regarding this project and the enriched higher categorical landscape.

\section{Templicial objects and Frobenius structures}\label{section: Templicial objects and Frobenius structures}

In this section, we introduce our main objects of study.

In \cite{joyal2002quasi}, Joyal introduced quasi-categories as simplicial sets satisfying Boardman and Vogt's \emph{weak Kan condition} \cite{boardmanvogt}, as a way of modelling weak composition.
In \cite{lowen2023enriched}, we introduced enriched analogues, that is templicial objects and quasi-categories in $\mathcal{V}$.

First, in \S \ref{subsection: Templicial objects and quasi-categories in a monoidal category}, we recall these concepts. Following Leinster \cite{leinster2000homotopy}, templicial objects are certain colax monoidal functors on the finite interval category. In order to express the quasi-category property for these objects, we use necklaces \cite{baues1980geometry} \cite{dugger2011rigidification}. In particular, we make use of a combinatorial description of the category $\nec$ of necklaces from \cite{lowen2023enriched}. This description is further used in \cite{borgesmarques2023category} to show that $\nec$ is Reedy monoidal.

Next, in \S \ref{subsection: Frobenius structures}, we put forth a different approach to composition in templicial objects by endowing them with 
Frobenius structures. These structures appear naturally when comparing templicial objects to dg-categories. 

The two approaches differ in two fundamental ways. Firstly, quasi-categories in $\mathcal{V}$ are templicial objects $X$ satisfying a horn-lifting condition analogous to ordinary quasi-categories, while a Frobenius structure on $X$ is given by a composition law
$$
(X_{p}\otimes_{S} X_{q}\rightarrow X_{p+q})_{p,q\geq 0}
$$
which is compatible with the templicial structure of $X$. So the former is a condition on $X$ while the latter consists of additional data. Secondly, the lifting condition in a quasi-category in $\mathcal{V}$ allows us to fill horns in $X$, while a Frobenius structure only allows to fill necklaces in $X$. Despite these differences, both approaches are related and we will return to the precise nature of this relation in Section \ref{section: Quasi-categories versus Frobenius structures}.

\subsection{Templicial objects and quasi-categories in a monoidal category}\label{subsection: Templicial objects and quasi-categories in a monoidal category}
We briefly recall the framework of templicial objects and quasi-categories in $\mathcal{V}$ from \cite{lowen2023enriched}, formulating the definitions and basic properties.

We denote the category of simplicial sets by $\SSet$ and the category of bipointed simplicial sets by $\SSet_{*,*} = \SSet_{\partial\Delta^{1}/}$. Given integers $0\leq j\leq n$, we denote by $\Delta^{n}$ and $\Lambda^{n}_{j}$ the standard $n$-simplex and the $j$th $n$-horn in $\SSet$ respectively. We consider both as bipointed simplicial sets with distinguished points $0$ and $n$.

\begin{Def}\label{definition: temp. obj.}
A \emph{templicial object} in $\mathcal{V}$ is a pair $(X,S)$ with $S$ a set and
$$
X: \fint^{op}\rightarrow \mathcal{V}\Quiv_{S}
$$
a strongly unital, colax monoidal functor. A \emph{templicial morphism} $(X,S)\rightarrow (Y,T)$ is a pair $(\alpha,f)$ with $f: S\rightarrow T$ a map of sets and $\alpha: f_{!}X\rightarrow Y$ a monoidal natural transformation between colax monoidal functors. Templicial objects and morphisms form a category which we denote by
$$
\ts\mathcal{V}
$$
\end{Def}

Let $(X,S)$ be a templicial object. The structure of $X$ as a functor $\fint^{op}\rightarrow \mathcal{V}\Quiv_{S}$ is equivalent to a collection of $\mathcal{V}$-quivers $(X_{n})_{n\geq 0}$ as well as quiver morphisms
\begin{align*}
d_{j} = X(\delta_{j}): X_{n}\rightarrow X_{n-1}\quad \text{for all integers }0 < j < n\\
s_{i} = X(\sigma_{i}): X_{n}\rightarrow X_{n+1}\quad \text{for all integers }0\leq i\leq n
\end{align*}
called the \emph{inner face morphisms} and \emph{degeneracy morphisms} respectively, which satisfy the usual simplicial identities. The colax monoidal structure of $X$ provides it with \emph{comultiplications}, which are quiver morphisms
$$
\mu_{k,l}: X_{k+l}\rightarrow X_{k}\otimes_{S} X_{l}
$$
for all $k,l\geq 0$ satisfying coassociativity and naturality conditions. Further, $X$ has a \emph{counit} $\epsilon: X_{0}\xrightarrow{\sim} I_{S}$, which is assumed to be an isomorphism by the strong unitality.

Note that, contrary to a simplicial object, $X$ does not have any outer face morphisms $d_{0}, d_{n}: X_{n}\rightarrow X_{n-1}$. Instead, $X$ is equipped with the comultiplications $\mu_{k,l}$ which serve as a replacement for the outer faces in the (non-cartesian) monoidal context.

\begin{Prop}\label{proposition: properties of temp. objs.}
The following statements hold true:
\begin{enumerate}[1.]
\item\label{item: temp. set is simp. set} If $\mathcal{V} = \Set$, then $\ts\mathcal{V}\simeq \SSet$.
\item\label{item: free-forget adjunction} The category $\ts\mathcal{V}$ is cocomplete and there is an adjunction $\tilde{F}: \SSet\leftrightarrows \ts\mathcal{V}: \tilde{U}$, where $\tilde{F}$ is induced by post-composition with $F$.
\item\label{item: simplex of underlying simp. set} Let $(X,S)$ be a templicial object and $n\geq 0$ an integer. Then an $n$-simplex of $\tilde{U}(X)$ is given by a pair
$$
\left((a_{i})_{i=0}^{n}, (\alpha_{i,j})_{0\leq i < j\leq n}\right)
$$
with $a_{i}\in S$ and $\alpha_{i,j}\in U(X_{j-i}(a_{i},a_{j}))$ such that for all $0\leq i < k < j\leq n$:
$$
\mu_{k-i,j-k}(\alpha_{i,j}) = \alpha_{i,k}\otimes \alpha_{k,j}
$$
\end{enumerate}
\end{Prop}
\begin{proof}
See Propositions 2.7 and 2.8, and Remark 2.9 of \cite{lowen2023enriched}.
\end{proof}

The category $\nec$ of \emph{necklaces} is the full subcategory of $\SSet_{*,*}$ spanned by all sequences $\Delta^{n_{1}}\vee ...\vee \Delta^{n_{k}}$ of standard simplices which are glued at their distinguished points (where $\vee$ denotes the wedge sum) \cite{baues1980geometry}\cite{dugger2011rigidification}. In \cite[Proposition 3.4]{lowen2023enriched}, it was shown that $\nec$ admits the following combinatorial description:
\begin{itemize}
\item The objects of $\nec$ are all pairs $(T,p)$ with $p\geq 0$ an integer and $T\subseteq [p]$ a subset containing $\{0 < p\}$.
\item A morphism $(T,p)\rightarrow (U,q)$ in $\nec$ is a morphism $f: [p]\rightarrow [q]$ in $\fint$ such that $U\subseteq f(T)$.
\end{itemize}
Here, a subset $T = \{0 = t_{0} < t_{1} < ... < t_{k} = p\}$ of $[p]$ corresponds to the necklace $\Delta^{t_{1}}\vee \Delta^{t_{2}-t_{1}}\vee ...\vee \Delta^{p-t_{k-1}}$. For example, $T = \{0 < p\}$ is the single simplex $\Delta^{p}$ while $T = [p]$ is a sequence of edges. Then the wedge sum $\vee$ is given as follows:
$$
(T,p)\vee (U,q) = (T\cup (p + U), p+q)
$$
for any necklaces $(T,p)$ and $(U,q)$. This makes $\nec$ into a monoidal category with monoidal unit given by $(\{0\},0)$.

Finally, we call a necklace map $f: (T,p)\rightarrow (U,q)$ \emph{inert} if the underlying morphism $f: [p]\rightarrow [q]$ in $\fint$ is the identity, and we call $f$ \emph{active} if $f(T) = U$. Every necklace map can be uniquely decomposed as an active map followed by an inert map.\\

We will further make use of the following notation and terminology.

\begin{Def}
Given a necklace $T$, we denote its number of beads by $\ell(T)$ and call it the \emph{length} of $T$. More precisley, $\ell(T) = \vert T\vert -1$.
\end{Def}

It is clear that $\ell(T\vee U) = \ell(T) + \ell(U)$ for all necklaces $T$ and $U$, and $\ell(\{0\}) = 0$.

\begin{Def}
Let $(T,p),(U,p)\in \nec$ with $p\geq 0$. Then there exist unique necklaces $(T_{i},u_{i}-u_{i-1})$ for $i\in \{1,...,l\}$ such that
$$
T\cup U = T_{1}\vee ...\vee T_{l}
$$
where we've written $U = \{0 = u_{0} < ... < u_{l} = p\}$. More precisely,
$$
T_{i} = \{0\}\cup \{t-u_{i-1}\mid t\in T, u_{i-1}\leq t\leq u_{i}\}\cup \{u_{i}-u_{i-1}\}
$$
We call the sequence
$$
(T_{1},...,T_{l})
$$
the \emph{splitting} of $T$ over $U$.
\end{Def}

The following proposition immediately follows from the definitions.

\begin{Prop}\label{proposition: properties of splittings}
Let $p\geq 0$ and $(T,p),(U,p)\in \nec$. The following statements hold true.
\begin{enumerate}[1.]
\item\label{item: properties of splittings 1} For any necklace $(V,p)$ with $T\cup U = V\cup U$, the splitting of $T$ over $U$ is equal to the splitting of $V$ over $U$.
\item\label{item: properties of splittings 2} If $T\subseteq U$ with $(T_{1},...,T_{l})$ the splitting of $T$ over $U$ and $(U_{1},...,U_{k})$ the splitting of $U$ over $T$, then $\ell(T_{i}) = 1$ for all $i\in\{1,...,p\}$ and $U = U_{1}\vee ... \vee U_{k}$.
\end{enumerate}
\end{Prop}

Necklace maps are useful because they allow us to parametrize the inner face maps and degeneracy maps of a templicial object $(X,S)$, as well as its comultiplications. In this way we can treat both structures on the same footing. More concretely, we can construct a functor $X^{nec}_{\bullet}: \nec^{op}\rightarrow \mathcal{V}\Quiv_{S}$ as follows. For any necklace $T = \{0 = t_{0} < t_{1} < ... < t_{k} = p\}$, consider the quiver
$$
X_{T} = X_{t_{1}}\otimes_{S} X_{t_{2}-t_{1}}\otimes_{S} ...\otimes_{S} X_{p-t_{k-1}}\in \mathcal{V}\Quiv_{S}
$$
Let $f: (T,p)\rightarrow (U,q)$ be a necklace map. As $f$ can be uniquely decomposed as an active map followed by an inert map, it suffices to define $X(f): X_{U}\rightarrow X_{T}$ when $f$ is either active or inert:
\begin{itemize}
\item If $f$ is inert, then $p = q$ and $U\subseteq T$. Let $(T_{1},...,T_{l})$ denote the splitting of $T$ over $U$. Then by Proposition \ref{proposition: properties of splittings}.\ref{item: properties of splittings 2}, $T = T_{1}\vee ...\vee T_{l}$. Now set
$$
X(f): X_{U}\xrightarrow{\mu_{T_{1}}\otimes ...\otimes \mu_{T_{l}}} X_{T}
$$
where for any necklace $T' = \{0 = t_{0} < t_{1} < ... < t_{k} = p\}$ we write $\mu_{T'}$ for the comultiplication
\begin{equation}\label{equation: general comultiplication}
\mu_{T'} = \mu_{t_{1},t_{2}-t_{1},...,p-t_{k-1}}: X_{p}\rightarrow X_{T'}
\end{equation}
which is well-defined by the coassociativity of $\mu$ (setting $\mu_{\{0 < n\}} = \id_{X_{n}}$).
\item If $f$ is active, write $T = \{0 = t_{0} < t_{1} < ... < t_{k} = p\}$. There exist unique $f_{i}: [t_{i}-t_{i-1}]\rightarrow [f(t_{i}) - f(t_{i-1})]$ in $\fint$ for all $i\in \{1,...,k\}$ such that $f = f_{1} + ... + f_{k}$. Now set
$$
X(f): X_{U}\simeq X_{f(t_{1})}\otimes_{S} ...\otimes_{S} X_{q - f(t_{k-1})}\xrightarrow{X(f_{1})\otimes ...\otimes X(f_{k})} X_{T}
$$
where the isomorphism is induced by the strong unitality of $X$ and the fact that $U = f(T)$.
\end{itemize}
In particular, we can evaluate in any pair $(a,b)$ of with $a,b\in S$ to obtain a functor
\begin{equation}\label{diagram: hom-object of necklace cat. assoc. to temp. obj.}
X_{\bullet}(a,b): \nec^{op}\rightarrow \mathcal{V}
\end{equation}

\begin{Ex}\label{example: hom-object of necklace cat. assoc. to simp. set}
When $\mathcal{V} = \Set$, we have a natural isomorphism
$$
K_{\bullet}(a,b)\simeq \SSet_{*,*}(-,K_{a,b})
$$
for any bipointed simplicial set $K$ (considered as a templicial set by virtue of Proposition \ref{proposition: properties of temp. objs.}.\ref{item: temp. set is simp. set}) with distinguished points $a$ and $b$.
\end{Ex}

\begin{Def}\label{definition: necklace functor lifts inner horns}
Let $Y: \nec^{op}\rightarrow \mathcal{V}$ be a functor. We say $Y$ \emph{lifts inner horns} if for all $0 < j < n$ any lifting problem
\[\begin{tikzcd}
	{\tilde{F}(\Lambda^{n}_{j})_{\bullet}(0,n)} & Y \\
	{\tilde{F}(\Delta^{n})_{\bullet}(0,n)}
	\arrow[from=1-1, to=2-1]
	\arrow[from=1-1, to=1-2]
	\arrow[dashed, from=2-1, to=1-2]
\end{tikzcd}\]
has a solution in $\mathcal{V}^{\nec^{op}}$. We call a templicial object $(X,S)$ in $\mathcal{V}$ a \emph{quasi-category in $\mathcal{V}$} if the functor $X_{\bullet}(a,b)$ lifts inner horns for all $a,b\in S$. In this case, we will refer to the elements of $S$ as the \emph{objects} of $X$ and to elements of $U(X_{1}(a,b))$ as the \emph{morphisms} $a\rightarrow b$ in $X$.

In this paper, a particular case of interest is the case where $\mathcal{V}$ is the category $\Mod(k)$ of modules over a commutative ring $k$. In this case we will also refer to $(X,S)$ as a \emph{quasi-category in modules}.
\end{Def}

\begin{Prop}\label{proposition: properties of quasi-cats.}
The following statements hold true:
\begin{enumerate}[1.]
\item\label{item: quasi-cat. in Set is quasi-cat.} A simplicial set is a quasi-category if and only if it is a quasi-category in $\Set$ in the sense of Definition \ref{definition: necklace functor lifts inner horns}.
\item\label{item: underlying simp. set preserves quasi-cats.} If $X$ is a quasi-category in $\mathcal{V}$, then $\tilde{U}(X)$ is a quasi-category.
\item\label{item: elementwise charac. of quasi-cat.} A templicial object $(X,S)$ is a quasi-category in $\mathcal{V}$ if and only if for all integers $0 < j < n$ and $a,b\in S$, the following holds. For any collection of elements
$$
x_{k}\in U((X_{k}\otimes_{S} X_{n-k})(a,b))\quad \text{and}\quad y_{i}\in U(X_{n-1}(a,b))
$$
for $0 < k,i < n$ with $i\neq j$, which satisfy:
\begin{itemize}
\item for all $0 < i < i' < n$ with $i\neq j\neq i'$,
$$
d_{i'-1}(y_{i}) = d_{i}(y_{i'}),
$$
\item for all $0 < k < l < n$,
$$
(\id_{X_{k}}\otimes \mu_{l-k,n-l})(x_{k}) = (\mu_{k,l-k}\otimes \id_{X_{n-l}})(x_{l})
$$
\item for all $0 < k < n-1$ and $0 < i < n$ with $i\neq j$,
$$
\mu_{k,n-k-1}(y_{i}) =
\begin{cases}
(d_{i}\otimes \id_{X_{n-k-1}})(x_{k+1}) & \text{if }i\leq k\\
(\id_{X_{k}}\otimes d_{i-k})(x_{k}) & \text{if }i > k
\end{cases}
$$
\end{itemize}
there exists an element $z\in U(X_{n}(a,b))$ such that
$$
d_{i}(z) = y_{i}\quad \text{and}\quad \mu_{k,n-k}(z) = x_{k}
$$
for all $0 < k,i < n$ with $i\neq j$.
\end{enumerate}
\end{Prop}
\begin{proof}
See Propositions 5.6, 5.8 and Corollary 5.13 of \cite{lowen2023enriched}.
\end{proof}

\subsection{Frobenius structures}\label{subsection: Frobenius structures}

In this section we introduce templicial objects $(X,S)$ endowed with a Frobenius structure. This is an associative operation which multiplies necklaces in $X$ to an entire simplex. Nerves of enriched categories, such as the dg-nerve of a dg-category, often come naturally equipped with such a structure (on a simplicial set). We will discuss the dg-nerve in more detail in Section \ref{section: Enriching the differential graded nerve}, but for now let us give an example in low dimensions.

Given a small dg-category $\mathcal{C}$, the $2$-simplices of its dg-nerve $N^{dg}(\mathcal{C})$ consist of objects $A,B,C\in \Ob(\mathcal{C})$, $0$-cycles $f\in \mathcal{C}_{0}(A,B)$, $g\in \mathcal{C}_{0}(B,C)$, $h\in \mathcal{C}_{0}(A,C)$ and a homotopy $\sigma\in \mathcal{C}_{1}(A,C)$ from $g\circ f$ to $h$. It was shown in \cite[Proposition 1.3.1.10]{lurie2009higher} that $N^{dg}(\mathcal{C})$ is always a quasi-category. In fact, we even have a map
$$
N^{dg}(\mathcal{C})_{1}\times_{N^{dg}(\mathcal{C})_{0}} N^{dg}(\mathcal{C})_{1}\rightarrow N^{dg}(\mathcal{C})_{2}
$$
Indeed, defining such a map comes down to choosing, for every pair of $0$-cycles $f\in \mathcal{C}_{0}(A,B)$ and $g\in \mathcal{C}_{0}(B,C)$, an $h$ and $\sigma$ as above. But there is an obvious choice, simply set $h = g\circ f$ and $\sigma = 0$ in the chain complex $\mathcal{C}_{\bullet}(A,C)$. This procedure extends to arbitrary dimensions so that we can define maps
$$
N^{dg}(\mathcal{C})_{p}\times_{N^{dg}(\mathcal{C})_{0}} N^{dg}(\mathcal{C})_{q}\rightarrow N^{dg}(\mathcal{C})_{p+q}
$$
for all $p,q\geq 0$, which are associative and compatible with the simplicial structure.

When comparing Frobenius structures to quasi-categories in Section \ref{section: Quasi-categories versus Frobenius structures}, we will require a weakening where we drop the associativity condition.

\begin{Def}\label{definition: naF-structure}
Let $H: \mathcal{U}\rightarrow \mathcal{V}$ be a functor between monoidal categories with a colax monoidal structure $(\mu,\epsilon)$. A \emph{nonassociative Frobenius (naF) structure} on $H$ is a pair $(Z,\eta)$ with $\eta: I\rightarrow H(I)$ a morphism in $\mathcal{V}$, called the \emph{unit}, and 
$$
Z: H(-)\otimes H(-)\rightarrow H(-\otimes -)
$$
a natural transformation, called the \emph{multiplication}, such that the following diagrams commute for all $A,B,C\in\mathcal{U}$:
\begin{equation}\label{diagram: Frobenius 1}
\begin{tikzcd}
	{H(A\otimes B)\otimes H(C)} & {H(A)\otimes H(B)\otimes H(C)} \\
	{H(A\otimes B\otimes C)} & {H(A)\otimes H(B\otimes C)}
	\arrow["{\mu_{A,B\otimes C}}"', from=2-1, to=2-2]
	\arrow["{Z_{A\otimes B,C}}"', from=1-1, to=2-1]
	\arrow["{\mu_{A,B}\otimes \id}", from=1-1, to=1-2]
	\arrow["{\id\otimes Z_{B,C}}", from=1-2, to=2-2]
\end{tikzcd}
\end{equation}
\begin{equation}\label{diagram: Frobenius 2}
\begin{tikzcd}
	{H(A)\otimes H(B\otimes C)} & {H(A)\otimes H(B)\otimes H(C)} \\
	{H(A\otimes B\otimes C)} & {H(A\otimes B)\otimes H(C)}
	\arrow["{Z_{A,B\otimes C}}"', from=1-1, to=2-1]
	\arrow["{\mu_{A\otimes B,C}}"', from=2-1, to=2-2]
	\arrow["{Z_{A,B}\otimes \id}", from=1-2, to=2-2]
	\arrow["{\id\otimes \mu_{B,C}}", from=1-1, to=1-2]
\end{tikzcd}
\end{equation}
and
\begin{equation}\label{diagram: Frobenius unitality}
\begin{tikzcd}
	{H(A)\otimes H(I)} & {H(A\otimes I)} \\
	{H(A)\otimes I} & {H(A)}
	\arrow["{H(A)\otimes \eta}", from=2-1, to=1-1]
	\arrow["{\rho_{H(A)}}"', from=2-1, to=2-2]
	\arrow["\wr"', from=1-2, to=2-2]
	\arrow["{Z_{A,I}}", from=1-1, to=1-2]
	\arrow["{H(\rho_{A})}", from=1-2, to=2-2]
	\arrow["\sim", from=2-1, to=2-2]
\end{tikzcd}
\begin{tikzcd}
	{H(I)\otimes H(A)} & {H(I\otimes A)} \\
	{I\otimes H(A)} & {H(A)}
	\arrow["{\eta\otimes H(A)}", from=2-1, to=1-1]
	\arrow["{\lambda_{(A)}}"', from=2-1, to=2-2]
	\arrow["\wr"', from=1-2, to=2-2]
	\arrow["{Z_{I,A}}", from=1-1, to=1-2]
	\arrow["\sim", from=2-1, to=2-2]
	\arrow["{H(\lambda_{A})}", from=1-2, to=2-2]
\end{tikzcd}
\end{equation}
where $\lambda$ and $\rho$ denote the left and right unit isomorphisms respectively.

For the purposes of this paper, we will always assume that a naF-structure is \emph{strongly unital}. That is, $\epsilon$ is invertible and
$$
\eta = \epsilon^{-1}$$
Then the naF-structure $(Z,\eta)$ is completely determined by $Z$ and $H$.
\end{Def}

\begin{Def}\label{definition: Frobenius monoidal functor}
Let $H: \mathcal{U}\rightarrow \mathcal{V}$ be a colax monoidal functor with a naF-structure. In the special case where the multiplication $Z$ is associative, that is
\begin{equation}\label{equation: Frobenius structure assoc.}
Z_{A\otimes B,C}(Z_{A,B}\otimes \id_{C}) = Z_{A,B\otimes C}(\id_{A}\otimes Z_{B,C})
\end{equation}
for all $A,B,C\in\mathcal{U}$, we refer to the naF-structure $(Z,\eta)$ as a \emph{Frobenius structure}. Note that in particular, $H$ is then also a lax monoidal functor.

In this case, $H$ is precisely a \emph{Frobenius monoidal functor} in the sense of \cite{day2008note}, with the additional property that the unit and counit are each others inverses.
\end{Def}

\begin{Ex}\label{example: strong monoidal is Frobenius}
A strong monoidal functor is exactly a Frobenius monoidal functor whose multiplication and comultiplication are each others inverses. In particular, the Frobenius structure is uniquely determined.
\end{Ex}

Let $(X,S)$ be a templicial object. Then in particular we have a colax monoidal functor $X: \fint^{op}\rightarrow \mathcal{V}\Quiv_{S}$. So it makes sense to consider naF-structures on $X$. Suppose $X$ has a naF-structure whose multiplication we denote by $Z$. Then $Z$ consists of quiver morphisms
$$
\left(Z^{p,q}: X_{p}\otimes_{S} X_{q}\rightarrow X_{p+q}\right)_{p,q\geq 0}
$$
which are natural in $p$ and $q$. The diagrams \eqref{diagram: Frobenius 1} and \eqref{diagram: Frobenius 2} then come down to
\begin{equation}\label{equation: mu-Z compatibility}
\mu_{k,l}Z^{p,q} =
\begin{cases}
(Z^{p,k-p}\otimes \id_{X_{l}})(\id_{X_{p}}\otimes \mu_{k-p,l}) & \text{if }p\leq k\\
(\id_{X_{k}}\otimes Z^{p-k,q})(\mu_{k,p-k}\otimes \id_{X_{q}}) & \text{if }p\geq k
\end{cases}
\end{equation}
for all $k,l,p,q\geq 0$ such that $k + l = p + q$. Note that in particular $\mu_{k,l}Z^{k,l} = \id_{X_{k}\otimes X_{l}}$ for all $k,l\geq 0$ by the strong unitality.

Later in the present paper, we will pay special attention to the templicial objects with an (associative) Frobenius structure.

\begin{Def}\label{defnaFtemp}
A \emph{naF-templicial object} is a pair $(X,Z)$ with $(X,S)$ a templicial object and $Z$ a naF-structure on $X: \fint^{op}\rightarrow \mathcal{V}\Quiv_{S}$. If $Z$ is actually a Frobenius structure, then $X$ is called a \emph{Frobenius templicial object}. Recall that in this case $X$ is in particular a lax monoidal functor.

Let $(\alpha,f): (X,S)\rightarrow (Y,T)$ be a templicial map and assume that $X$ and $Y$ have Frobenius structures $Z_{X}$ and $Z_{Y}$ respectively. We call $(\alpha,f)$ a \emph{Frobenius templicial morphism} if the induced natural transformation $X\rightarrow f^{*}Y$ is monoidal with respect to the lax structures on $X$ and $f^{*}Y$. This is equivalent to requiring the following diagram to commute for all $k,l\geq 0$:
\begin{equation}\label{diagram: Frob. temp. map}
\begin{tikzcd}
	{f_{!}(X_{k+l})} & {Y_{k+l}} \\
	{f_{!}(X_{k}\otimes_{S} X_{l})} & {f_{!}(X_{k})\otimes_{T} f_{!}(X_{l})} & {Y_{k}\otimes_{T} Y_{l}}
	\arrow["{f_{!}(Z^{k,l}_{X})}", from=2-1, to=1-1]
	\arrow[from=2-1, to=2-2]
	\arrow["{\alpha_{k}\otimes_{T} \alpha_{l}}"', from=2-2, to=2-3]
	\arrow["{ \alpha_{k+l}}", from=1-1, to=1-2]
	\arrow["{Z^{k,l}_{Y}}"', from=2-3, to=1-2]
\end{tikzcd}
\end{equation}

We denote the category of Frobenius templicial objects and Frobenius templicial morphisms between them by
$$
\Fs\mathcal{V}
$$
Note that there is an obvious faithful forgetful functor $\Fs\mathcal{V}\rightarrow \ts\mathcal{V}$.
\end{Def}

\begin{Rem}
The forgetful functor $\Fs\mathcal{V}\rightarrow \ts\mathcal{V}$ has a left-adjoint, but we postpone its construction to a subsequent paper. 
\end{Rem}

\begin{Ex}\label{example: nerve has Frob structure}
let $\mathcal{C}$ be a small $\mathcal{V}$-enriched category. Then its templicial nerve $N_{\mathcal{V}}(\mathcal{C})$ (see \cite[Definition 2.11]{lowen2023enriched}) carries a Frobenius structure. Indeed, this is a direct consequence of Example \ref{example: strong monoidal is Frobenius} since the underlying colax monoidal functor $\fint^{op}\rightarrow \mathcal{V}\Quiv_{\Ob(\mathcal{V})}$ is strong monoidal.
\end{Ex}

\begin{Rem}
The templicial homotopy coherent nerve $N^{hc}_{\mathcal{V}}: \mathcal{V}\Cat_{\Delta}\rightarrow \ts\mathcal{V}$ introduced in \cite{lowen2023enriched} also carries a natural Frobenius structure. We will come back to this in a subsequent paper \cite{mertens2023nerves}.
\end{Rem}

We end the section with some formulas that describe the interaction between a nonassociative Frobenius structure and the comultiplication maps of a templicial object, which we will use later on.

\begin{Not}
Let $(X,S)$ be a templicial object with naF-structure $Z$. Given a necklace $T = \{0 = t_{0} < t_{1} < ... < t_{k} = p\}$, recall the comultiplication morphism $\mu_{T}: X_{p}\rightarrow X_{T}$ \eqref{equation: general comultiplication}. We'd like to similarly define a quiver morphism
$$
Z^{T}: X_{T}\rightarrow X_{p}
$$
However, since $Z$ is not assumed to be associative, this will depend on how we compose the two-variable morphisms $Z^{k,l}$. Nevertheless, making an arbitrary choice, we can define
$$
Z^{p_{1},...,p_{k}} = Z^{p_{1},p_{2}+...+p_{k}}(\id_{X_{p_{1}}}\otimes Z^{p_{2},...,p_{k}})
$$
inductively on $k\geq 2$, for all $p_{1},...,p_{k}\geq 0$, and subsequently set
$$
Z^{T} =
\begin{cases}
\epsilon^{-1} & \text{if }k = 0\\
\id_{X_{p}} & \text{if }k = 1\\
Z^{t_{1},t_{2}-t_{1},...,p-t_{k-1}} & \text{if }k\geq 2\\
\end{cases}
$$
\end{Not}

\begin{Rem}
Let $(X,S)$ be a templicial object with naF-structure $Z$. Consider a necklace $T = \{0 = t_{0} < ... < t_{k} = p\}$. It follows from the naturality of $Z$ that for all $i,j\in [p]\setminus T$:
\begin{align*}
d_{j}Z^{T} = Z^{\delta^{-1}_{j}(T)}(\id\otimes ...\otimes \id\otimes d_{j-i_{m-1}}\otimes \id\otimes ...\otimes \id)\\
s_{i}Z^{T} = Z^{\sigma^{-1}_{i}(T)}(\id\otimes ...\otimes \id\otimes s_{i-i_{m-1}}\otimes \id\otimes ...\otimes \id)
\end{align*}
where $m\in \{1,...,k\}$ is minimal such that $i < i_{m}$ or $j < i_{m}$ respectively. On the other hand, if $i\in T$, then
$$
s_{i}Z^{T} = Z^{\sigma^{-1}_{i}(T)}(\id\otimes ...\otimes \id\otimes s_{0}\epsilon^{-1}\otimes \id\otimes ...\otimes \id)
$$
However, if $0 < j < n$ and $j\in T$ the naturality of $Z$ doesn't supply us with a formula to pass the face map $d_{j}$ through $Z$.
\end{Rem}

\begin{Prop}\label{proposition: higher mu-Z compatibility}
Let $(X,S)$ be a templicial object with naF-structure $Z$. Let $p\geq 0$ and $(T,p),(U,p)\in \nec$. Then
\begin{equation}\label{equation: higher mu-Z compatibility}
\mu_{T}Z^{U} = (Z^{U_{1}}\otimes ...\otimes Z^{U_{k}})(\mu_{T_{1}}\otimes ...\otimes \mu_{T_{l}})
\end{equation}
where $(U_{1},...,U_{k})$ is the splitting of $U$ over $T$ and $(T_{1},...,T_{l})$ is the splitting of $T$ over $U$.
\end{Prop}
\begin{proof}
We use induction on $k = \ell(T)$ and $l = \ell(U)$. If either $k = 0$ or $l = 0$, then both are zero and (\ref{equation: higher mu-Z compatibility}) is trivially true. For $k = 1$, both sides of (\ref{equation: higher mu-Z compatibility}) reduce to $Z^{U}$. Similarly, if $l = 1$ both sides reduce to $\mu_{T}$.

Assume further that $k,l\geq 2$. Let $t\in T$ and $u\in U$ be minimal such that $0 < t$ and $0 < u$. We can write $T = \{0 < t\}\vee T'$ and $U = \{0 < u\}\vee U'$ for some unique necklaces $(T',p-t)$ and $(U',p-u)$. Then:
$$
\mu_{T}Z^{U} = (\id_{X_{t}}\otimes \mu_{T'})\mu_{t,p-t}Z^{u,p-u}(\id_{X_{u}}\otimes Z^{U'})
$$
If $t\leq u$, then $\mu_{t,p-t}Z^{u,p-u} = (\id_{X_{t}}\otimes Z^{u-t,p-u})(\mu_{t,u-t}\otimes \id_{X_{p-u}})$ by (\ref{equation: mu-Z compatibility}), and we can write $T_{1} = \{0 < t\}\vee T'_{1}$ for some unique $(T'_{1},u-t)$. So, by the induction hypothesis, we have
\begin{align*}
\mu_{T}Z^{U} &= (\id_{X_{t}}\otimes \mu_{T'}Z^{u-t,p-u})(\mu_{t,u-t}\otimes Z^{U'})\\
&= (\id_{X_{t}}\otimes \mu_{T'}Z^{\{0 < u-t\}\vee U'})(\mu_{t,u-t}\otimes \id_{X_{U'}})\\
&= (\id_{X_{t}}\otimes Z^{U_{2}}\otimes ...\otimes Z^{U_{k}})((\id_{X_{t}}\otimes \mu_{T'_{1}})\mu_{t,u-t}\otimes \mu_{T_{2}}\otimes ...\otimes \mu_{T_{l}})\\
&= (Z^{U_{1}}\otimes Z^{U_{2}}\otimes ...\otimes Z^{U_{k}})(\mu_{T_{1}}\otimes \mu_{T_{2}}\otimes ...\otimes \mu_{T_{l}})
\end{align*}
where we used that $U_{1} = \{0 < t\}$ since $t \leq u$. A similar argument shows the case for $t\geq u$.
\end{proof}

\begin{Cor}\label{corollary: higher mu-Z compatibility}
Let $(X,S)$ be a templicial object with naF-structure $Z$. Let $p\geq 0$ and $(T,p),(U,p)\in\nec$. The following statements are true.
\begin{enumerate}[1.]
\item If $T\subseteq U$, and $(U_{1},...,U_{k})$ is the splitting of $U$ over $T$, then
$$
\mu_{T}Z^{U} = Z^{U_{1}}\otimes ... \otimes Z^{U_{k}}
$$
\item If $U\subseteq T$, and $(T_{1},...,T_{l})$ is the splitting of $T$ over $U$, then
$$
\mu_{T}Z^{U} = \mu_{T_{1}}\otimes ... \otimes \mu_{T_{l}}
$$
\item We have $\mu_{T}Z^{U}\mu_{U} = \mu_{T}Z^{T\cup U}\mu_{T\cup U}$.
\item If $Z$ is associative, then we have $Z^{U}\mu_{U}Z^{T} = Z^{T\cup U}\mu_{T\cup U}Z^{T}$.
\end{enumerate}
\begin{proof}
Statements $1$ and $2$ follow from Propositions \ref{proposition: properties of splittings}.\ref{item: properties of splittings 2} and \ref{proposition: higher mu-Z compatibility}.

To prove $3$, consider the splittings $(U_{1},...,U_{k})$ and $(T_{1},...,T_{l})$ of $U$ over $T$ and $T$ over $U$ respectively. By Proposition \ref{proposition: properties of splittings}.\ref{item: properties of splittings 1}, $(U_{1},...,U_{k})$ is also the splitting of $T\cup U$ over $T$. Thus as $T\subseteq T\cup U$, it follows from $1$ that
\begin{align*}
\mu_{T}Z^{T\cup U}\mu_{T\cup U} &= (Z^{U_{1}}\otimes ... \otimes Z^{U_{k}})\mu_{T\cup U} = (Z^{U_{1}}\otimes ... \otimes Z^{U_{k}})\mu_{T_{1}\vee ... \vee T_{l}}\\
&= (Z^{U_{1}}\otimes ... \otimes Z^{U_{k}})(\mu_{T_{1}}\otimes ...\otimes \mu_{T_{l}})\mu_{U} = \mu_{T}Z^{U}\mu_{U}
\end{align*}
where we used the coassociativity of $\mu$. Finally, $4$ follows similarly by using $2$.
\end{proof}
\end{Cor}

\section{Enriching the differential graded nerve}\label{section: Enriching the differential graded nerve}

Let $k$ be a fixed commutative unital ring. In this section, we consider the case where $\mathcal{V}$ is the category $\Mod(k)$ of $k$-modules with the tensor product over $k$. It is worth noting that all the proceeding proofs are still valid when $\mathcal{V}$ is a general abelian category (in addition to the conditions imposed in \S\ref{subsection: Notations and conventions}). Nevertheless, we will restrict to $\mathcal{V} = \Mod(k)$ for simplicity.

We denote by $\Ch(k)$ the category of chain complexes over $k$, and by $k\Cat_{dg}$ the category of small (homologically graded) dg-categories over $k$. Our starting point is the dg-nerve functor
$$
N^{dg}: k\Cat_{dg}\rightarrow \SSet
$$
from \cite{block2009Riemann}, \cite[Construction 1.3.1.6]{lurie2016higher}. This section is devoted to the construction of an enriched version of the dg-nerve which we call the \emph{templicial dg-nerve}:
$$
N^{dg}_{k}: k\Cat_{dg}\rightarrow \ts\Mod(k)
$$
(see Definition \ref{definition: templicial dg-nerve}), which recovers $N^{dg}$ after composition with $\tilde{U}$ (Corollary \ref{corollary: underlying simp. set of templicial dg-nerve}).

This will be done in two main steps. First, in \S \ref{subsection: A Dold-Kan correspondence for augmented simplicial modules}, we prove a version of the Dold-Kan correspondence relating chain complexes to \emph{augmented} simplicial modules, which is compatible with the monoidal structures (Proposition \ref{proposition: augmented Dold-Kan correspondence}). 
Applying this locally to a given dg-category, we thus obtain a category $\mathcal{C}$ enriched in augmented simplicial modules.

Next, in \S\ref{subsection: Frobenius structures and augmented simplicial categories}, we apply a construction to the enriched category $\mathcal{C}$ reminiscent of the bar construction, which takes values in templicial modules. Our main result is a resulting equivalence of categories between positively graded dg-categories on the one hand and Frobenius templicial modules on the other hand (Theorem \ref{theorem: augmented simp. cat. and Frob. temp. obj. equiv.}). 

\subsection{A Dold-Kan correspondence for augmented simplicial modules}\label{subsection: A Dold-Kan correspondence for augmented simplicial modules}

Let us denote by $\Ch(k)_{\geq 0}$ the category of chain complexes concentrated in non-negative degree. Similarly, we denote the category of small non-negatively graded dg-cate\-gories by $k\Cat_{dg,\geq 0}$.

The classical Dold-Kan correspondence \cite{dold1958homology}\cite{kan1958functors} states that the normalized chain functor is an equivalence of categories:
$$
N_{\bullet}: S\Mod(k)\xrightarrow{\sim} \Ch(k)_{\geq 0}
$$
Consider the category of \emph{augmented simplicial ($k$-)modules}:
$$
S^{+}\Mod(k) = \Fun(\asimp^{op},\Mod(k))
$$
We can also define a normalized chain functor for augmented simplicial modules by making use of a shift. Consider the following isomorphism of categories:
$$
s: \Ch_{\geq 0}(k)\xrightarrow{\simeq} \Ch_{> 0}(k)
$$
with $sC_{n} = C_{n-1}$ and $\partial^{sC}_{n} = \partial^{C}_{n-1}$ for all $n > 0$.

\begin{Con}\label{construction: augmented normalized chain functor}
We construct a functor
$$
N^{+}_{\bullet}: S^{+}\Mod(k)\rightarrow \Ch_{\geq 0}(k)
$$
Given a augmented simplicial $k$-module $A$, let $A_{\geq 0}$ denote its restriction to $\simp^{op}$. Note that this forgets the module $A_{-1}$ and the face map $d_{0}: A_{0}\rightarrow A_{-1}$. Then define
$$
N^{+}_{\bullet}(A) = (sN_{\bullet}(A_{\geq 0})\xrightarrow{d_{0}} A_{-1})
$$
Clearly the assignment $A\mapsto N^{+}_{\bullet}(A)$ extends to a functor $S^{+}\Mod(k)\rightarrow \Ch_{\geq 0}(k)$.
\end{Con}

\begin{Rem}
Note that for all $n\geq 0$ we have
$$
N^{+}_{n}(A)\simeq \frac{A_{n-1}}{\sum_{i=0}^{n-2}s_{i}(A_{n-2})}
$$
So in low degrees, we have $N^{+}_{0}(A) = A_{-1}$, $N^{+}_{1}(A) = A_{0}$ and $N^{+}_{2}(A) = A_{1}/s_{0}(A_{0})$. The differential is given by, for all $n\geq 0$:
$$
\partial_{n+1} = \sum_{i=0}^{n}(-1)^{i}\overline{d}_{i}: N^{+}_{n+1}(A)\rightarrow N^{+}_{n}(A)
$$
where $\overline{d}_{i}$ is induced by the $i$th face map $d_{i}: A_{n}\rightarrow A_{n-1}$ of $A$.
\end{Rem}

\begin{Prop}\label{proposition: augmented Dold-Kan correspondence}
The functor $N^{+}: S^{+}\Mod(k)\rightarrow \Ch(k)$ of Construction \ref{construction: augmented normalized chain functor} has a right-adjoint $\Gamma^{+}: \Ch(k)\rightarrow S^{+}\Mod(k)$ which is given by
$$
\Gamma^{+}(C_{\bullet}) = \Ch(k)\left(N^{+}_{\bullet}(\Delta^{(-)};k), C_{\bullet}\right): \asimp^{op}\rightarrow \Mod(k)
$$
for all chain complexes $C_{\bullet}$. Moreover, the restriction
\[\begin{tikzcd}
	{S^{+}\Mod(k)} & {\Ch_{\geq 0}(k)}
	\arrow["{N^{+}_{\bullet}}", shift left=2, from=1-1, to=1-2]
	\arrow["{\Gamma^{+}}", shift left=2, from=1-2, to=1-1]
	\arrow["\sim"{description}, draw=none, from=1-1, to=1-2]
\end{tikzcd}\]
is an adjoint equivalence of categories.
\end{Prop}
\begin{proof}
Because $s$, $(-)_{\geq 0}$ and $N$ all clearly preserve colimits, it follows that $N^{+}$ preserves colimits as well. Thus the first statement follows from a general nerve construction applied to the the functor $N^{+}_{\bullet}(\Delta^{(-)};k): \asimp\rightarrow \Ch(k)$.

It remains to show that $N^{+}_{\bullet}: S^{+}\Mod(k)\rightarrow \Ch_{\geq 0}(k)$ is an equivalence, but this quickly follows from the definition of $N^{+}_{\bullet}$ and the fact that $N_{\bullet}$ is an equivalence.
\end{proof}

\begin{Rem}
Given a chain complex $C_{\bullet}$, let us make the augmented simplicial $k$-module $\Gamma^{+}(C_{\bullet})$ a little more explicit. For each $n\geq -1$, the $k$-module $\Gamma^{+}(C_{\bullet})_{n}$ consists of all collections
$$
(a_{I})_{I\subseteq [n]}\in \bigoplus_{I\subseteq [n]}C_{\vert I\vert}
$$
that satisfy, for all $I = \{i_{1} < ... < i_{m}\}\subseteq [n]$:
$$
\partial(a_{I}) = \sum_{j=1}^{m}(-1)^{j-1}a_{I\setminus\{i_{j}\}}
$$
For $h: [m]\rightarrow [n]$ in $\asimp$, the map
$$
\Gamma^{+}(C_{\bullet})(h): \Gamma^{+}(C_{\bullet})_{n}\rightarrow \Gamma^{+}(C_{\bullet})_{m}: (a_{I})_{I\subseteq [n]}\mapsto (b_{J})_{J\subseteq [m]}
$$
is given by $b_{J} = a_{h(J)}$ if $h\vert_{J}$ is injective and $b_{J} = 0$ otherwise.
\end{Rem}

As both $\Mod(k)$ and $\asimp$ are monoidal categories, we can endow $S^{+}\Mod(k)$ with the monoidal structure given by Day convolution (see \cite{day1970closed}). This is also known as the \emph{join} of augmented simplicial objects. We denote the resulting monoidal closed category by $(S^{+}\Mod(k),\star,F(\Delta^{-1}))$.

Explicitely, the join of two augmented simplicial modules $A$ and $B$ is given by
$$
(A\star B)_{n} = \bigoplus_{\substack{k,l\geq -1\\ k+l+1 = n}}A_{k}\otimes B_{l}
$$
for all $n\geq -1$. Given $f: [m]\rightarrow [n]$, and $k,l\geq -1$ such that $k + l + 1 = n$, there exist unique $f_{1}: [p]\rightarrow [k]$ and $f_{2}: [q]\rightarrow [l]$ with $p + q + 1 = m$ and $f_{1}\star f_{2} = f$. With these notations, we have for all $a\in A_{k}$ and $b\in B_{l}$:
$$
(A\star B)(f)(a\otimes b) = A(f_{1})(a)\otimes B(f_{2})(b)
$$
The monoidal unit is given by $F(\Delta^{-1})$, that is $F(\Delta^{-1})_{-1} = k$ and $F(\Delta^{-1})_{n} = 0$ for all $n\geq 0$.

\begin{Prop}\label{proposition: monoidal augmented Dold-Kan correspondence}
The adjunction $N^{+}_{\bullet}: S^{+}\Mod(k)\leftrightarrows \Ch(k): \Gamma^{+}$ is monoidal. Further, the restriction $N^{+}_{\bullet}: S^{+}\Mod(k)\leftrightarrows \Ch_{\geq 0}(k): \Gamma^{+}$ is a monoidal equivalence.
\end{Prop}
\begin{proof}
For both statements it suffices to show that $N^{+}_{\bullet}$ has the structure of a strong monoidal functor. Let $A$ and $B$ be augmented simplicial modules. For all $n > 0$ and $i\in [n-2]$, the degeneracy map $s_{i}: (A\star B)_{n-2}\rightarrow (A\star B)_{n-1}$ is given by
$$
s_{i}\vert_{A_{k}\otimes B_{n-k-3}} =
\begin{cases}
s^{A}_{i}\otimes \id_{B_{n-k-3}} & \text{if }i\leq k\\
\id_{A_{k}}\otimes s^{B}_{i-k-1} & \text{if }i > k
\end{cases}
$$
It follows that the submodule $\sum_{i=0}^{n-2}s_{i}((A\star B)_{n-2})$ of $(A\otimes B)_{n-1}$ is equal to
$$
\bigoplus_{\substack{p,q\geq 0\\ p + q = n}}\left(\sum_{i=0}^{p-2}(s^{A}_{i}\otimes \id_{B_{q-1}})(A_{p-2}\otimes B_{q-1}) + \sum_{i=0}^{q-2}(\id_{A_{p-1}}\otimes s^{B}_{i})(A_{p-1}\otimes B_{q-2})\right)
$$
Consequently, we have an isomorphism
$$
N^{+}_{n}(A\star B)\simeq \bigoplus_{\substack{p,q\geq 0\\ p + q = n}}(N^{+}_{p}(A)\otimes N^{+}_{q}(B)) = (N^{+}(A)\otimes N^{+}(B))_{n}
$$
Moreover this isomorphism is a chain map. This follows from the fact that for all $n\geq 0$ and $i\in [n]$, the face map $d_{i}: (A\star B)_{n}\rightarrow (A\star B)_{n-1}$ is given by,
$$
d_{i}\vert_{A_{k}\otimes B_{l}} =
\begin{cases}
d^{A}_{i}\otimes \id_{B_{l}} & \text{if }i\leq k\\
\id_{A_{k}}\otimes d^{B}_{i-k-1} & \text{if }i > k
\end{cases}
$$
So, we get an isomorphism
$$
\mu_{A,B}: N^{+}_{\bullet}(A\star B)\xrightarrow{\sim} N^{+}_{\bullet}(A)\otimes N^{+}_{\bullet}(B)
$$
It is a direct verification that this isomorphism is natural in $A$ and $B$, and coassociative. Finally, we clearly have an isomorphism $\epsilon: N^{+}_{\bullet}(\Delta^{-1};k)\xrightarrow{\sim} k[0]$ and it follows easily that $\mu$ is counital with respect to $\epsilon$.
\end{proof}

\subsection{Frobenius structures and $S^{+}\Mod(k)$-categories}\label{subsection: Frobenius structures and augmented simplicial categories}

We now consider categories enriched in the monoidal category $(S^{+}\Mod(k),\star,F(\Delta^{-1}))$ of the previous subsection, and refer to them as \emph{$S^{+}\Mod(k)$-categories}. Further we denote the category of small $S^{+}\Mod(k)$-categories by
$$
k\Cat_{\Delta_{+}} = S^{+}\Mod(k)\text{-}\Cat
$$
Our main result in this section is the following theorem.

\begin{Thm}\label{theorem: augmented simp. cat. and Frob. temp. obj. equiv.}

There is an adjoint equivalence of categories
\[\begin{tikzcd}
	{k\Cat_{\Delta_{+}}} & {\Fs\Mod(k)}
	\arrow[""{name=0, anchor=center, inner sep=0}, "{\mathcal{T}}", shift left=2, from=1-1, to=1-2]
	\arrow[""{name=1, anchor=center, inner sep=0}, "{\mathcal{K}}", shift left=2, from=1-2, to=1-1]
	\arrow["\sim"{description}, Rightarrow, draw=none, from=0, to=1]
\end{tikzcd}\]
\end{Thm}

The proof of Theorem \ref{theorem: augmented simp. cat. and Frob. temp. obj. equiv.} will be completed in Appendix \ref{section: Frobenius structures and augmented simplicial categories}, where we carry out some straightforward verifications. For now we only describe both functors on the level of objects (see Constructions \ref{construction: temp. obj. assoc. to augmented simp. cat.} and \ref{construction: augmented simp. cat. assoc. to Frob. temp. obj.}). The main conceptual reason why they are inverse equivalences is given in Lemma \ref{lemma: properties of retraction of kernel embedding} and Proposition \ref{proposition: unit and counit of tensor algebra joint kernel adjunction} below.\\

To define the functor $\mathcal{T}$, it will be more convenient to extend the augmented simplex category $\asimp$ to the equivalent category $\Lin$ of finite linearly ordered sets and order morphisms between them.

Given finite linearly ordered sets $I$ and $J$, their disjoint union $I\sqcup J$ carries a unique linear order which restricts to those of $I$ and $J$ and such that $i < j$ for all $i\in I$ and $j\in J$. This clearly defines a monoidal product $-\sqcup -: \Lin\times \Lin\rightarrow \Lin$ with monoidal unit given by the empty poset $\emptyset$. In fact, we have a canonical equivalence of monoidal categories
\begin{equation}\label{equation: linearly ordered sets and augmented simplex category equiv.}
(\Lin,\sqcup,\emptyset)\simeq (\asimp,\star,[-1])
\end{equation}
which identifies a linearly ordered set $J = \{j_{0} < ... < j_{k}\}$ with the ordinal $[k]$.

Further, if $I_{1},I_{2}\subseteq J$ are subsets of a finite linearly ordered set $J$, we'll also write $I_{1}\sqcup I_{2}$ for the union $I_{1}\cup I_{2}$ to indicate that $i < j$ for all $i\in I_{1}$ and $j\in I_{2}$. Note that up to isomorphism this coincides with the above definition.

\begin{Rem}\label{remark: extension of augmented simp. obj.}
Under the monoidal equivalence \eqref{equation: linearly ordered sets and augmented simplex category equiv.}, we may identify augmented simplicial modules with functors $\Lin^{op}\rightarrow \Mod(k)$. Further, we may identify $S^{+}\Mod(k)$-categories with categories enriched in $\Fun(\Lin^{op},\Mod(k))$ (equipped with the Day convolution), and will implicitly do so in what follows.

Let $\mathcal{C}\in k\Cat_{\Delta_{+}}$ and set $S = \Ob(\mathcal{C})$. Via the monoidal equivalence \eqref{equation: linearly ordered sets and augmented simplex category equiv.}, $\mathcal{C}$ comes equipped with a quiver $\mathcal{C}_{J}\in k\Quiv_{S}$ for every finite linearly ordered set $J$, and a quiver map $\mathcal{C}_{J}\rightarrow \mathcal{C}_{I}$ for every morphism $f: I\rightarrow J$ in $\Lin$.

Moreover, the composition law of $\mathcal{C}$ is determined by quiver maps
$$
m_{I,J}: \mathcal{C}_{I}\otimes_{S} \mathcal{C}_{J}\rightarrow \mathcal{C}_{I\sqcup J}
$$
which are natural in $I$ and $J$. Its identities are determined by a quiver map $k_{S}\rightarrow \mathcal{C}_{\emptyset}$ (where $k_{S}$ denotes the monoidal unit of $k\Quiv_{S}$).

Because of the associativity of the composition in $\mathcal{C}$, we also have an induced quiver morphism, for all $p\geq 2$:
$$
m_{I_{1},...,I_{p}}: \mathcal{C}_{I_{1}}\otimes_{S} ...\otimes_{S} \mathcal{C}_{I_{p}}\rightarrow \mathcal{C}_{I_{1}\sqcup ...\sqcup I_{p}}
$$
for all finite linearly ordered sets $I_{1},...,I_{p}$. Further, we write $m_{I_{1},...,I_{p}} = u$ if $p = 0$ and $m_{I_{1},...,I_{p}} = \id_{\mathcal{C}_{I_{1}}}$ if $p = 1$.
\end{Rem}

\begin{Not}\label{notation: complement of a necklace}
Given $n > 0$, a subset $I\subseteq \{1,...,n\}$ and $k\in \{1,...,n\}\setminus I$, we write
$$
I_{< k} = \{i\in I\mid i < k\}\quad \text{and}\quad I_{> k} = \{i\in I\mid i > k\}
$$
Consider a necklace $T = \{0 = t_{0} < t_{1} < ... < t_{k} = p\}$. For all $i\in \{1,...,k\}$, we set
$$
T^{c}_{i} = (([p]\setminus T)_{< t_{i}})_{> t_{i-1}} = \{t_{i-1}+1 < t_{i-1}+2 < ... < t_{i}-1\}
$$
considered as an object of $\Lin$. Note that we have
$$
T^{c}_{1}\sqcup ...\sqcup T^{c}_{k} = [p]\setminus T
$$
\end{Not}

\begin{Con}\label{construction: temp. obj. assoc. to augmented simp. cat.}
Let $\mathcal{C}$ be a $S^{+}\Mod(k)$-category with object set $S = \Ob(\mathcal{C})$. We construct a Frobenius templicial $k$-module $(\mathcal{T}(\mathcal{C}),S)$ as follows.
\begin{itemize}
\item Given an $n\geq 0$, consider the quiver
$$
\mathcal{T}(\mathcal{C})_{n} = \bigoplus_{(T,n)\in \nec}\mathcal{C}^{\otimes}_{T}
$$
where, for every necklace $T = \{0 = t_{0} < t_{1} < ... < t_{k} = n\}$:
\begin{align*}
\mathcal{C}^{\otimes}_{T} &= \mathcal{C}_{T^{c}_{1}}\otimes_{S} \mathcal{C}_{T^{c}_{2}}\otimes_{S} ...\otimes_{S} \mathcal{C}_{T^{c}_{k}}\\
&= \mathcal{C}_{t_{1}-2}\otimes_{S} \mathcal{C}_{t_{2}-t_{1}-2}\otimes_{S} ... \otimes_{S} \mathcal{C}_{n-t_{k-1}-2}\quad \in k\Quiv_{S}
\end{align*}
\item Given a morphism $f: [m]\rightarrow [n]$ in $\fint$, we define a quiver morphism
$$
\mathcal{T}(\mathcal{C})(f): \mathcal{T}(\mathcal{C})_{n}\rightarrow \mathcal{T}(\mathcal{C})_{m}
$$
as follows. For $(T,n)\in \nec$, let us write $U = f^{-1}(T)$ and further denote $U = \{0 = u_{0} < u_{1} < ... < u_{l} = m\}$. Then for every $j\in \{0,...,l\}$, $f(u_{j}) = t_{p_{j}}$ for some $p_{j}\in \{0,...,k\}$. Thus we can restrict $f$ to
$$
f\vert_{U^{c}_{j}}: U^{c}_{j}\rightarrow T^{c}_{p_{j-1}+1}\sqcup ...\sqcup T^{c}_{p_{j}}
$$
in $\Lin$ for all $j\in \{1,...,l\}$. Now define a quiver map $\mathcal{C}^{\otimes}(f)_{T}: \mathcal{C}^{\otimes}_{T}\rightarrow \mathcal{C}^{\otimes}_{U}$ as
$$
\mathcal{C}^{\otimes}(f)_{T} = \mathcal{C}(f\vert_{U^{c}_{1}})m_{T^{c}_{1},...,T^{c}_{p_{1}}}\otimes_{S} ... \otimes_{S} \mathcal{C}(f\vert_{U^{c}_{l}})m_{T^{c}_{p_{l-1}+1},...,T^{c}_{k}}
$$
These combine to define a map $\mathcal{T}(\mathcal{C})(f): \mathcal{T}(\mathcal{C})_{n}\rightarrow \mathcal{T}(\mathcal{C})_{m}$.
\item The comultiplications and Frobenius structure of $\mathcal{T}(\mathcal{C})$ are given by the following canonical projection and coprojection maps, for all $k,l\geq 0$:
\begin{align*}
\mu_{k,l}: \mathcal{T}(\mathcal{C})_{k+l} = \bigoplus_{(T,k+l)\in \nec}\mathcal{C}^{\otimes}_{T}\rightarrow \bigoplus_{\substack{(T,k+l)\in \nec\\ k\in T}}\mathcal{C}^{\otimes}_{T}\simeq \mathcal{T}(\mathcal{C})_{k}\otimes_{S} \mathcal{T}(\mathcal{C})_{l}\\
Z^{k,l}: \mathcal{T}(\mathcal{C})_{k}\otimes_{S} \mathcal{T}(\mathcal{C})_{l}\simeq \bigoplus_{\substack{(T,k+l)\in \nec\\ k\in T}}\mathcal{C}^{\otimes}_{T}\rightarrow \bigoplus_{(T,k+l)\in \nec}\mathcal{C}^{\otimes}_{T} = \mathcal{T}(\mathcal{C})_{k+l}
\end{align*}
\end{itemize}
\end{Con}

\begin{Rem}
Note that the functor $\mathcal{T}$ is a bar type construction, where both domain and codomain are equipped with some simplicial structure. As we will see in \S \ref{subsection: The templicial dg-nerve}, the dg-nerve $N^{dg}: k\Cat_{dg}\rightarrow \SSet$ can be expressed as a composition of functors including $\mathcal{T}$.
In \cite{holstein2022categorical}, Holstein and Lazarev give a construction of the dg-nerve which involves a bar construction turning a dg-category into a pointed curved coalgebra. Though similar in spirit, it is currently unclear exactly how both approaches are related.
\end{Rem}

\begin{Ex}
Let $0 < j < n$ and consider the coface map $\delta_{j}: [n-1]\rightarrow [n]$ in $\fint$. For a necklace $T = \{0 = t_{0} < t_{1} < ... < t_{k} = n\}$, we have
$$
\mathcal{C}^{\otimes}(\delta_{j})_{T} =
\begin{cases}
\id_{\mathcal{C}_{t_{1}-2}}\otimes ...\otimes d_{j-t_{p-1}-1}\otimes ...\otimes \id_{\mathcal{C}_{n-t_{k-1}-2}} & \text{if }j\not\in T\\
\id_{\mathcal{C}_{t_{1}-2}}\otimes ...\otimes m_{j-t_{p-1}-2,t_{p+1}-j-2}\otimes ...\otimes \id_{\mathcal{C}_{n-t_{k-1}-2}} & \text{if }j\in T
\end{cases}
$$
where $p\in \{1,...,k\}$ is the unique integer such that $t_{p-1} < j\leq t_{p}$. 

Similarly, consider the codegeneracy map $\sigma_{i}: [n+1]\rightarrow [n]$ in $\fint$ with $0\leq i\leq n$. For a necklace $T = \{0 = t_{0} < t_{1} < ... < t_{k} = n\}$, we have
$$
\mathcal{C}^{\otimes}(\sigma_{i})_{T} =
\begin{cases}
\id_{\mathcal{C}_{t_{1}-2}}\otimes ...\otimes s_{i-t_{p-1}-1}\otimes ...\otimes \id_{\mathcal{C}_{n-i_{k-1}-2}} & \text{if }i\not\in T\\
\id_{\mathcal{C}_{t_{1}-2}}\otimes ...\otimes \id_{\mathcal{C}_{i-t_{p-1}-2}}\otimes u\otimes \id_{\mathcal{C}_{t_{p+1}-i-2}}\otimes ...\otimes \id_{\mathcal{C}_{n-t_{k-1}-2}} & \text{if }i\in T
\end{cases}
$$
where $p\in \{1,...,k\}$ is the unique integer such that $t_{p-1} < i\leq t_{p}$.
\end{Ex}

\begin{Con}\label{construction: augmented simp. cat. assoc. to Frob. temp. obj.}
Let $(X,S)$ be a templicial $k$-module with a Frobenius structure $Z$. We construct an $S^{+}\Mod(k)$-category $\mathcal{K}(X)$ as follows.
\begin{itemize}
\item Set $\Ob(\mathcal{K}(X)) = S$.
\item For all $a,b\in S$ define an augmented simplicial module $\mathcal{K}(X)(a,b)$ by
$$
\mathcal{K}(X)_{n}(a,b) = \bigcap_{k=1}^{n+1}\ker((\mu_{k,n+2-k})_{a,b})\subseteq X_{n+2}(a,b)
$$
for all $n\geq -1$. So for example $\mathcal{K}(X)_{-1} = X_{1}$ and $\mathcal{K}(X)_{0} = \ker(\mu_{1,1})$. 

Given $f: [m]\rightarrow [n]$ in $\asimp$, the morphism $[0]\star f\star [0]: [m+2]\rightarrow [n+2]$ restricts to a map
$$
\mathcal{K}(X)(f)_{a,b}: \mathcal{K}(X)_{n}(a,b)\rightarrow \mathcal{K}(X)_{m}(a,b)
$$
\item For $a,b,c\in S$, the composition of $\mathcal{K}(X)$ is given by
$$
d_{p+2}Z^{p+2,q+2}: \mathcal{K}(X)_{p}(a,b)\otimes \mathcal{K}(X)_{q}(b,c)\rightarrow \mathcal{K}(X)_{p+q+1}(a,c)
$$
for all $p,q\geq -1$.
\item For all $a\in S$, the identity on $a$ is given by
$$
s_{0}(a)\in X_{1}(a,a) = \mathcal{K}(X)_{-1}(a,a)
$$
\end{itemize}
\end{Con}

\begin{Lem}\label{lemma: properties of retraction of kernel embedding}
Let $(X,S)$ be a templicial object with Frobenius structure $Z$. Consider the quiver maps
$$
\xi_{n} = \sum_{(T,n)\in \nec}(-1)^{\ell(T)+1}Z^{T}\mu_{T}: X_{n}\rightarrow X_{n}
$$
for $n\geq 1$. Then the following statements are true.
\begin{enumerate}[1.]
\item\label{item: properties of retraction of kernel embedding 1} For all $(U,n)\in \nec$ with $\ell(U) > 1$, we have $\mu_{U}\xi_{n} = 0$ and $\xi_{n}Z^{U} = 0$. In particular, $\xi_{n}$ restricts to a map $X_{n}\rightarrow \mathcal{K}(X)_{n-2}$.
\item\label{item: properties of retraction of kernel embedding 2} Let $\xi_{T} = \xi_{t_{1}}\otimes_{S} ...\otimes_{S} \xi_{n-t_{k-1}}$ for any necklace $T = \{0 = t_{0} < t_{1} < ...  < t_{k} = n\}$. Then we have
$$
\sum_{(T,n)\in \nec}Z^{T}\xi_{T}\mu_{T} = \id_{X_{n}}
$$
\item\label{item: properties of retraction of kernel embedding 3}
For all necklaces $(T,n)$ and $(U,n)$, we have
$$
\xi_{U}\mu_{U}Z^{T}\vert_{\mathcal{K}(X)^{\otimes}_{T}} =
\begin{cases}
\id_{\mathcal{K}(X)^{\otimes}_{T}} & \text{if }T = U\\
0 & \text{if }T\neq U
\end{cases}
$$
\end{enumerate}
\end{Lem}
\begin{proof}
\begin{enumerate}[1.]
\item Take $0 < k < n$, then by Corollary \ref{corollary: higher mu-Z compatibility}.3, we have
\begin{align*}
&\mu_{k,n-k}\left(\sum_{(T,n)\in \nec}(-1)^{\ell(T)}Z^{T}\mu_{T}\right) = \sum_{(U,n)\in \nec}\sum_{\substack{(T,n)\in \nec\\ T\cup\{k\} = U}}(-1)^{\ell(T)}\mu_{k,n-k}Z^{U}\mu_{U}\\
&\qquad = \sum_{\substack{(U,n)\in \nec\\ k\in U}}\left((-1)^{\ell(U\setminus\{k\})} + (-1)^{\ell(U)}\right)\mu_{k,n-k}Z^{U}\mu_{U} = 0
\end{align*}
For more general $(U,n)\in \nec$ with $\ell(U) > 1$, it follows from the coassociativity of $\mu$ that $\mu_{U}\xi_{n} = 0$. The proof that $\xi_{n}Z^{U} = 0$ is completely analogous using Corollary \ref{corollary: higher mu-Z compatibility}.4 because $Z$ is associative.
\item By the coassociativity of $\mu$ and associativity of $Z$, we have
\begin{align*}
\sum_{(T,n)\in \nec}Z^{T}\xi_{T}\mu_{T} &= \sum_{\substack{(U,n)\in \nec\\ T\subseteq U}}(-1)^{\ell(U)+\ell(T)}Z^{U}\mu_{U}\\
&= \sum_{(U,n)\in \nec}(-1)^{\ell(U)}\left(\sum_{\substack{(T,n)\in \nec\\ T\subseteq U}}(-1)^{\ell(T)}\right)Z^{U}\mu_{U}
\end{align*}
Now note that the middle sum is zero whenever $\ell(U) > 1$. Indeed, choose $k\in U\setminus \{0 < n\}$, then $T\mapsto T\setminus \{k\}$ defines a bijection
$$
\{(T,n)\in \nec\mid T\subseteq U, k\in T\}\xrightarrow {\sim} \{(T,n)\in \nec\mid T\subseteq U, k\not\in T\}
$$
and $\ell(T\setminus\{k\}) = \ell(T) - 1$ if $k\in T$. Hence the only remaining term is for $U = \{0 < n\}$, which gives $\id_{X_{n}}$.
\item By Proposition \ref{proposition: higher mu-Z compatibility}, we have
$$
\xi_{U}\mu_{U}Z^{T}\vert_{\mathcal{K}(X)^{\otimes}_{T}} = (\xi_{u_{1}}Z^{T_{1}}\otimes_{S} ... \otimes_{S} \xi_{n-u_{l-1}}Z^{T_{l}})(\mu_{U_{1}}\otimes_{S} ...\otimes_{S} \mu_{U_{k}})\vert_{\mathcal{K}(X)^{\otimes}_{T}}
$$
where $U = \{0 = u_{0} < u_{1} < ... < u_{l} = n\}$ and $(T_{1},...,T_{l})$ and $(U_{1},...,U_{k})$ are the splittings of $T$ over $U$ and $U$ over $T$ respectively. Now it follows from $1$ that the right hand side of this equation vanishes unless the length of every $U_{i}$ and $T_{j}$
 is $1$. But in that case $T = U$ and $\xi_{T}\mu_{T}Z^{T}\vert_{\mathcal{K}(X)^{\otimes}_{T}} = \xi_{T}\vert_{\mathcal{K}(X)^{\otimes}_{T}} = \id_{\mathcal{K}(X)^{\otimes}_{T}}$.
\end{enumerate}
\end{proof}

\begin{Prop}\label{proposition: unit and counit of tensor algebra joint kernel adjunction}
The following statements are true.
\begin{enumerate}[1.]
\item Let $(X,S)$ be a templicial object with Frobenius structure $Z$. For any $n\geq 1$, the following composite induced by the inclusions $\mathcal{K}(X)_{p-2}\hookrightarrow X_{p}$,
$$
\epsilon_{X_{n}}: \mathcal{T}\mathcal{K}(X)_{n} = \bigoplus_{(T,n)\in \nec}\mathcal{K}(X)^{\otimes}_{T}\rightarrow \bigoplus_{(T,n)\in \nec}X_{T}\xrightarrow{(Z^{T})_{T}} X_{n}
$$
is an isomorphism. 
\item Let $\mathcal{C}$ be an $S^{+}\Mod(k)$-category with object set $S = \Ob(\mathcal{C})$. For any $n\geq -1$, the coprojection $\mathcal{C}_{n} = \mathcal{C}^{\otimes}_{\{0 < n+2\}}\hookrightarrow \mathcal{T}(\mathcal{C})_{n+2}$ factors through an isomorphism
$$
\eta_{\mathcal{C}_{n}}: \mathcal{C}_{n}\xrightarrow{\simeq} \mathcal{K}\mathcal{T}(\mathcal{C})_{n}
$$
\end{enumerate}
\end{Prop}
\begin{proof}
\begin{enumerate}[1.]
\item It follows from Lemma \ref{lemma: properties of retraction of kernel embedding} that $X_{n}$ is isomorphic to the direct sum $\bigoplus_{(T,n)}\mathcal{K}(X)^{\otimes}_{T}$ with projections $\xi_{T}\mu_{T}$ and coprojections $Z^{T}\vert_{\mathcal{K}(X)^{\otimes}_{T}}$.
\item This immediately follows from the definition of the comultiplications of $\mathcal{T}(\mathcal{C})$.
\end{enumerate}
\end{proof}

\subsection{The templicial dg-nerve}\label{subsection: The templicial dg-nerve}

We are now ready to define our linear enrichment of the dg-nerve. It is constructed using the main results of the previous two subsections. After discussing it in some detail, we give a characterization of the (Frobenius) templicial maps into the templicial dg-nerve (see Proposition \ref{proposition: temp. maps into dg-nerve} and Corollary \ref{corollary: Frob. temp. maps into dg-nerve}). As a consequence, we show in Corollary \ref{corollary: underlying simp. set of templicial dg-nerve} that the templicial dg-nerve indeed recovers the classical dg-nerve after composing with the forgetful functor $\tilde{U}$.

\begin{Def}\label{definition: templicial dg-nerve}
We define the \emph{templicial differential graded (dg) nerve} as the composite
$$
N^{dg}_{k}: k\Cat_{dg}\xrightarrow{\Gamma^{+}} k\Cat_{\Delta_{+}}\xrightarrow{\mathcal{T}} \Fs\Mod(k)\rightarrow \ts\Mod(k)
$$
where $\Gamma^{+}$ is the right-adjoint induced by Proposition \ref{proposition: monoidal augmented Dold-Kan correspondence}, $\mathcal{T}$ is the equivalence from Proposition \ref{proposition: tensor algebra functor} and the third arrow represents the forgetful functor.
\end{Def}

\begin{Prop}\label{proposition: Frob. temp. obj. and dg-cat. equivalence}
The functor $N^{dg}_{k}$ factors through an equivalence of categories
$$
k\Cat_{dg,\geq 0}\simeq \Fs\Mod(k)
$$
\end{Prop}
\begin{proof}
In view of Proposition \ref{proposition: monoidal augmented Dold-Kan correspondence}, we have an induced adjunction
\[\begin{tikzcd}
	{k\Cat_{\Delta_{+}}} & k\Cat_{dg}
	\arrow[""{name=0, anchor=center, inner sep=0}, "{N^{+}_{\bullet}}", shift left=2, from=1-1, to=1-2]
	\arrow[""{name=1, anchor=center, inner sep=0}, "{\Gamma^{+}}", shift left=2, from=1-2, to=1-1]
	\arrow["\dashv"{anchor=center, rotate=-90}, draw=none, from=0, to=1]
\end{tikzcd}\]
which after restriction to $k\Cat_{dg,\geq 0}$ becomes an equivalence of categories. The result then follows from Theorem \ref{theorem: augmented simp. cat. and Frob. temp. obj. equiv.}.
\end{proof}

Before continuing, let us first recall the classical dg-nerve $N^{dg}: k\Cat_{dg} \rightarrow \SSet$. A few different (but isomorphic) versions of $N^{dg}$ exist in the literature with varying sign conventions \cite{block2009Riemann}, \cite{lurie2016higher}, \cite[Definition 2.2.8]{faonte2017simplicial}, \cite[Tag 00PL]{kerodon}. We will make use of Faonte's ``small dg-nerve'' \cite[Definition 2.2.8]{faonte2017simplicial}, as it is the most convenient for our purpose. It is defined as follows.

Given a small dg-category $\mathcal{C}$, the \emph{differential graded (dg) nerve} $N^{dg}(\mathcal{C})$ is the simplicial set where for every $n\geq 0$, an $n$-simplex is a pair
$$
\left((A_{i})_{i=0}^{n}, (f_{I})_{\substack{I\subseteq [n]\\ \vert I\vert\geq 2}}\right)
$$
where $A_{0},...,A_{n}\in \Ob(\mathcal{C})$ and for each subset $I = \{i_{0} < ... < i_{m}\}\subseteq [n]$ with $m\geq 1$, $f_{I}\in \mathcal{C}_{m - 1}(A_{i_{0}},A_{i_{m}})$ such that
$$
\partial(f_{I}) = \sum_{j=1}^{m-1}\left((-1)^{j-1}f_{I\setminus\{i_{j}\}} + (-1)^{m(j-1)+1}f_{\{i_{j} < ... < i_{m}\}}\circ f_{\{i_{0} < ... < i_{j}\}}\right)
$$
or, when employing the notation from \S\ref{subsection: Notations and conventions}:
$$
\partial(f_{I}) = \sum_{j=1}^{m-1}(-1)^{j-1}\left(f_{I\setminus\{i_{j}\}} - m(f_{\{i_{0} < ... < i_{j}\}}\otimes f_{\{i_{j} < ... < i_{m}\}})\right)
$$
For any $h: [m]\rightarrow [n]$ in $\fint$, the map $N^{dg}(\mathcal{C})_{n}\rightarrow N^{dg}(\mathcal{C})_{m}$ is given by
$$
\left((A_{i})_{i=0}^{n}, (f_{I})_{\substack{I\subseteq [n]\\ \vert I\vert\geq 2}}\right)\mapsto \left((A_{h(i)})_{i=0}^{m}, (h^{*}f_{J})_{\substack{J\subseteq [m]\\ \vert J\vert\geq 2}}\right)
$$
where
$$
h^{*}f_{J} =
\begin{cases}
f_{h(J)} & \text{if }h\text{ is injective on }J\\
\id_{A_{i}} & \text{if }J = \{j_{0} < j_{1}\}\text{ with }h(j_{0}) = i = h(j_{1})\\
0 & \text{otherwise}
\end{cases}
$$

\begin{Ex}\label{example: classical dg-nerve in low dimensions}
Given a small dg-category $\mathcal{C}$, let us decribe the dg-nerve $N^{dg}(\mathcal{C})$ in low dimensions.
\begin{itemize}
\item The vertices of $N^{dg}(\mathcal{C})$ are given by the objects $A\in \Ob(\mathcal{C})$.
\item The edges of $N^{dg}(\mathcal{C})$ are given by the $0$-cycles of the chain complex $\mathcal{C}_{\bullet}(A_{0},A_{1})$ for some $A_{0},A_{1}\in \Ob(\mathcal{C})$, i.e. $f_{01}\in \mathcal{C}_{0}(A_{0},A_{1})$ such that $\partial(f_{01}) = 0$.
\item A $2$-simplex of $N^{dg}(\mathcal{C})$ is given by a (not necessarily commutative) diagram of $0$-cycles:
\[\begin{tikzcd}
	& A_{1} \\
	A_{0} && A_{2}
	\arrow["f_{01}", from=2-1, to=1-2]
	\arrow["f_{12}", from=1-2, to=2-3]
	\arrow[""{name=0, anchor=center, inner sep=0}, "f_{02}"', from=2-1, to=2-3]
	\arrow["f_{012}"', shorten <=3pt, shorten >=3pt, Rightarrow, from=1-2, to=0]
\end{tikzcd}\]
with $f_{012}\in \mathcal{C}_{1}(A_{0},A_{2})$ such that $\partial(f_{012}) = f_{02} - f_{12}\circ f_{01}$. So $f_{012}$ is a homotopy in $\mathcal{C}_{\bullet}(A_{0},A_{2})$ from $f_{12}\circ f_{01}$ to $f_{02}$.
\end{itemize}
\end{Ex}

Let us now make the templicial dg-nerve a little more explicit as well.

\begin{Rem}\label{remark: explicit templicial dg-nerve}
Let $\mathcal{C}$ be small dg-category and consider the templicial object $N^{dg}_{k}(\mathcal{C})$. The vertex set of $N^{dg}_{k}(\mathcal{C})$ is simply $S = \Ob(\mathcal{C})$.

Take $n\geq 0$. From Construction \ref{construction: temp. obj. assoc. to augmented simp. cat.} we have (also see Notation \ref{notation: complement of a necklace}):
$$
N^{dg}_{k}(\mathcal{C})_{n} = \bigoplus_{(T,n)\in \nec}\Gamma^{+}(\mathcal{C})^{\otimes}_{T} = \bigoplus_{(T,n)\in \nec}\Gamma^{+}(\mathcal{C})_{T^{c}_{1}}\otimes_{S} ...\otimes_{S} \Gamma^{+}(\mathcal{C})_{T^{c}_{k}}\in k\Quiv_{S}
$$
For all $A,B\in \Ob(\mathcal{C})$, $\Gamma^{+}(\mathcal{C})_{T^{c}_{i}}(A,B) = \Gamma^{+}(\mathcal{C}_{\bullet}(A,B))_{T^{c}_{i}}$ is the $k$-module
$$
\left\lbrace (a_{I})_{I}\in \bigoplus_{I\subseteq T^{c}_{i}}\mathcal{C}_{\vert I\vert}(A,B)\,\middle\vert\, \partial(a_{I}) = \sum_{j=1}^{p}(-1)^{j-1}a_{I\setminus \{i_{j}\}}\right\rbrace
$$
where we've written $I = \{i_{1} < ... < i_{p}\}\subseteq T^{c}_{i}$.

The counit of $N^{dg}_{k}(\mathcal{C})$ is just the identity $N^{dg}_{k}(\mathcal{C})_{0} = k_{S}$, the monoidal unit of $k\Quiv_{S}$. The comultiplication maps $\mu_{p,q}$ and Frobenius structure maps $Z^{p,q}$ are defined by the canonical projections and coprojections respectively:
\begin{align*}
\mu_{p,q}: N^{dg}_{k}(\mathcal{C})_{p+q} = \bigoplus_{(T,p+q)\in \nec}\Gamma^{+}(\mathcal{C})^{\otimes}_{T}\twoheadrightarrow \bigoplus_{\substack{(T,p+q)\in \nec\\ p\in T}}\Gamma^{+}(\mathcal{C})^{\otimes}_{T}\simeq N^{dg}_{k}(C)_{p}\otimes_{S} N^{dg}_{k}(\mathcal{C})_{q}\\
Z^{p,q}: N^{dg}_{k}(C)_{p}\otimes_{S} N^{dg}_{k}(\mathcal{C})_{q}\simeq \bigoplus_{\substack{(T,p+q)\in \nec\\ p\in T}}\Gamma^{+}(\mathcal{C})^{\otimes}_{T}\hookrightarrow \bigoplus_{(T,p)\in \nec}\Gamma^{+}(\mathcal{C})^{\otimes}_{T} = N^{dg}_{k}(\mathcal{C})_{p+q}
\end{align*}

Finally, the inner face maps and degeneracy maps of $N^{dg}_{k}(\mathcal{C})$ are completely determined by projection onto the component $\Gamma^{+}(\mathcal{C})^{\otimes}_{\{0 < p\}} = \Gamma^{+}(\mathcal{C})_{\{0 < p\}^{c}}$ of $N^{dg}_{k}(\mathcal{C})_{p}$. More precisely:
\begin{itemize}
\item For all $0 < j < n$, the composite of $d_{j}: N^{dg}_{k}(\mathcal{C})_{n}\rightarrow N^{dg}_{k}(\mathcal{C})_{n-1}$ with the canonical projection $N^{dg}_{k}(\mathcal{C})_{n-1}\twoheadrightarrow \Gamma^{+}(\mathcal{C})_{\{0 < n-1\}^{c}}$
is equal to the composite
$$
N^{dg}_{k}(\mathcal{C})_{n}\twoheadrightarrow \Gamma^{+}(\mathcal{C})_{\{0 < n\}^{c}}\oplus (\Gamma^{+}(\mathcal{C})_{\{0 < j\}^{c}}\otimes_{S} \Gamma^{+}(\mathcal{C})_{\{j < n\}^{c}})\xrightarrow{(d^{+}_{j},m^{+})} \Gamma^{+}(\mathcal{C})_{\{0 < n-1\}^{c}}
$$
If $m$ is the composition law of $\mathcal{C}$, then $d^{+}_{j}$ and $m^{+}$ are defined by:
\begin{align*}
d^{+}_{j}\left((a_{I})_{I\subseteq \{0 < n\}^{c}}\right) &= (a_{\delta_{j}(J)})_{J\subseteq \{0 < n-1\}^{c}}\\
m^{+}\left(((a_{I})_{I\subseteq \{0 < j\}^{c}}\otimes (b_{J})_{J\subseteq \{j < n\}^{c}}\right) &= \left(m(a_{\delta_{j}(K)_{< j}}\otimes b_{\delta_{j}(K)_{> j}})\right)_{K\subseteq \{0 < n-1\}^{c}}
\end{align*}
where we used Notation \ref{notation: complement of a necklace}.

\item For all $0\leq i\leq n$, the composite of $s_{i}: N^{dg}_{k}(\mathcal{C})_{n}\rightarrow N^{dg}_{k}(\mathcal{C})_{n+1}$ with the canonical projection $N^{dg}_{k}(\mathcal{C})_{n+1}\twoheadrightarrow \Gamma^{+}(\mathcal{C})_{\{0 < n+1\}^{c}}$ is equal to
$$
\begin{cases}
N^{dg}_{k}(\mathcal{C})_{n}\twoheadrightarrow \Gamma^{+}(\mathcal{C})_{\{0 < n\}^{c}}\xrightarrow{s^{+}_{i}} \Gamma^{+}(\mathcal{C})_{\{0 < n+1\}^{c}} & \text{if }0 < i < n\\
N^{dg}_{k}(\mathcal{C})_{0} = I_{\Ob(\mathcal{C})}\xrightarrow{u} Z_{0}(\mathcal{C}) = \Gamma^{+}(\mathcal{C})_{\emptyset} & \text{if }n = i = 0\\
0 & \text{otherwise}
\end{cases}
$$
where $u: k_{S}\rightarrow \mathcal{C}_{0}$ represents the identities in $\mathcal{C}$, $Z_{0}$ denotes the functor taking $0$-cycles and
$$
s^{+}_{i}\left((a_{I})_{I\subseteq \{0 < n\}^{c}}\right) = (b_{J})_{J\subseteq \{0 < n+1\}^{c}}
$$
with $b_{J} = a_{\sigma_{i}(J)}$ if $\{i,i+1\}\not\subseteq J$ and $b_{J} = 0$ if $\{i,i+1\}\subseteq J$.
\end{itemize}
\end{Rem}

\begin{Ex}
Given small a dg-category $\mathcal{C}$, let us describe the templicial object $N^{dg}_{k}(\mathcal{C})$ in low dimensions. Note the analogy with Example \ref{example: classical dg-nerve in low dimensions}.
\begin{itemize}
\item The vertices of $N^{dg}_{k}(\mathcal{C})$ are given by the objects $A\in \Ob(\mathcal{C})$.
\item Take objects $A,B\in \Ob(\mathcal{C})$. Then
$$
N^{hc}_{k}(\mathcal{C})_{1}(A,B) = \Gamma^{+}(\mathcal{C}_{\bullet}(A,B))_{-1} = Z_{0}(\mathcal{C}_{\bullet}(A,B))
$$
is the submodule of $\mathcal{C}_{0}(A,B)$ of $0$-cycles.
\item In two dimensions, we have
\begin{align*}
N^{dg}_{k}(\mathcal{C})_{2}(A,B) &= \Gamma^{+}(\mathcal{C}_{\bullet}(A,B))_{0}\oplus \bigoplus_{C\in \Ob(\mathcal{C})}\Gamma^{+}(\mathcal{C}_{\bullet}(A,C))_{-1}\otimes \Gamma^{+}(\mathcal{C}_{\bullet}(C,B))_{-1}\\
&\simeq \mathcal{C}_{1}(A,B)\oplus \bigoplus_{C\in \Ob(\mathcal{C})}Z_{0}(\mathcal{C}_{\bullet}(A,C))\otimes Z_{0}(\mathcal{C}_{\bullet}(C,B)) 
\end{align*}
The comultiplication map $\mu_{1,1}: N^{dg}_{k}(\mathcal{C})_{2}\rightarrow N^{dg}_{k}(\mathcal{C})_{1}\otimes_{\Ob(\mathcal{C})} N^{dg}_{k}(\mathcal{C})_{1}$ is given by projection onto the second term in the expression above. On the other hand, the face map $d_{1}: N^{dg}_{k}(\mathcal{C})_{2}\rightarrow N^{dg}_{k}(\mathcal{C})_{1}$ is defined as follows. Given a pair $(h,\alpha)$ with $h\in \mathcal{C}_{1}(A,B)$ and $\alpha$ a tensor belonging to the second term in the expression above, we have
$$
d_{1}(h,\alpha) = \partial(h) + m(\alpha)
$$
where $\partial: \mathcal{C}_{1}(A,B)\rightarrow Z_{0}(\mathcal{C}_{\bullet}(A,B))$ is the differential. Setting $f = d_{1}(h,\alpha)$, we thus find that $h$ describes a homotopy in $\mathcal{C}_{\bullet}(A,B)$ between $f$ and $m(\alpha)$.
\end{itemize}
\end{Ex}

The description of the simplices of the dg-nerve in Example \ref{example: classical dg-nerve in low dimensions} can be generalized to the following remark from \cite{kerodon}.

\begin{Rem}[\cite{kerodon}, Tag 00PV]\label{remark: simp. map into dg-nerve}
Let $\mathcal{C}$ be a small dg-category over $k$ and let $K$ be a simplicial set. A map of simplicial sets $f: K\rightarrow N^{dg}(\mathcal{C})$ is equivalent to the following data:
\begin{itemize}
\item A map of sets $f_{0}: S\rightarrow \Ob(\mathcal{C})$.
\item For all $a,b\in K_{0}$ and $n > 0$, a map
$$
f_{n}: K_{n}(a,b)\rightarrow U\left(\mathcal{C}_{n-1}(f_{0}(a),f_{0}(b))\right)
$$
\end{itemize}
Moreover, this data must satisfy the following conditions:
\begin{enumerate}[(a)]
\item For all $a,b\in K_{0}$, $0\leq i\leq n$ and $\sigma\in K_{n}(a,b)$,
\begin{equation*}
f_{n+1}(s^{K}_{i}(\sigma)) =
\begin{cases}
\id_{f(a)} & \text{if }n = 0, a = b\\
0 & \text{otherwise}
\end{cases}
\end{equation*}
\item For all $a,b\in K_{0}$, $n > 0$ and $\sigma\in K_{n}(a,b)$,
\begin{equation*}
\partial(f_{n}(\sigma)) = \sum_{j=1}^{n-1}(-1)^{j-1}\left(f_{n-1}(d^{K}_{j}(\sigma)) - m(f_{j}(d^{K}_{j+1}...d^{K}_{n}(\sigma))\otimes f_{n-j}(d^{K}_{0}...d^{K}_{0}(\sigma)))\right)
\end{equation*}
\end{enumerate}
\end{Rem}

We can show the analogous statement for the templicial dg-nerve. However, as the proof is rather tedious and not very insightful, we postpone it to Appendix \ref{subsection: Templicial maps into the templicial dg-nerve}.

\begin{Prop}\label{proposition: temp. maps into dg-nerve}
Let $\mathcal{C}$ be a small dg-category over $k$ and $(X,S)$ a templicial $k$-module. A templicial map $(\alpha,f): X\rightarrow N^{dg}_{k}(\mathcal{C})$ is equivalent to the following data:
\begin{itemize}
\item A map of sets $f: S\rightarrow \Ob(\mathcal{C})$.
\item For all $n > 0$, a quiver map
$$
\beta_{n}: X_{n}\rightarrow f^{*}\mathcal{C}_{n-1}
$$
\end{itemize}
satisfying the following properties:
\begin{enumerate}[(a)]
\item For all $0\leq i\leq n$,
\begin{equation}\label{equation: temp. map into dg-nerve 1}
\beta_{n+1}s^{X}_{i} =
\begin{cases}
u & \text{if }n = 0\\
0 & \text{if }n > 0
\end{cases}
\end{equation}
where $u$ denotes the identities of the dg-category $f^{*}\mathcal{C}$.
\item For all $n > 0$,
\begin{equation}\label{equation: temp. map into dg-nerve 2}
\partial\beta_{n} = \sum_{j=1}^{n-1}(-1)^{j-1}\left(\beta_{n-1}d^{X}_{j} - m(\beta_{j}\otimes_{S} \beta_{n-j})\mu^{X}_{j,n-j}\right)
\end{equation}
where $m$ and $\partial$ are respectively the composition law and the differential of the dg-category $f^{*}\mathcal{C}$ (induced by the lax structure of $f^{*}$).
\end{enumerate}
Moreover, for all $n > 0$, $\beta_{n}$ is adjoint to the composite
$$
f_{!}X_{n}\xrightarrow{\alpha_{n}} N^{dg}_{k}(\mathcal{C})_{n}\rightarrow \Gamma^{+}(\mathcal{C})_{\{0 < n\}^{c}}\xrightarrow{\pi_{\{0 < n\}^{c}}} \mathcal{C}_{n-1}
$$
\end{Prop}

\begin{Cor}\label{corollary: Frob. temp. maps into dg-nerve}
Let $\mathcal{C}$ be a small dg-category over $k$ and $(X,S)$ a Frobenius templicial $k$-module. A Frobenius templicial map $(\alpha,f): X\rightarrow N^{dg}_{k}(\mathcal{C})$ is equivalent to the following data:
\begin{itemize}
\item A map of sets $f: S\rightarrow \Ob(\mathcal{C})$.
\item For all $n > 0$, a quiver map
$$
\beta_{n}: X_{n}\rightarrow f^{*}\mathcal{C}_{n-1}
$$
\end{itemize}
satisfying properties \eqref{equation: temp. map into dg-nerve 1} and \eqref{equation: temp. map into dg-nerve 2}, as well as
\begin{equation}\label{equation: Frob. temp. map into dg-nerve}
\beta_{p+q}Z^{p,q}_{X} = 0\quad \text{and}\quad \beta_{p+q-1}d_{p}Z^{p,q}_{X} = m(\beta_{p}\otimes \beta_{q})
\end{equation}
for all $p,q > 0$.
\end{Cor}

\begin{Rem}
Note that under property \eqref{equation: Frob. temp. map into dg-nerve}, \eqref{equation: temp. map into dg-nerve 2} is equivalent to
$$
\partial\beta_{n} = \sum_{j=1}^{n-1}(-1)^{j-1}\beta_{n-1}d^{X}_{j}(\id_{X_{n}} - Z^{j,n-j}_{X}\mu^{X}_{j,n-j})
$$
\end{Rem}

\begin{Cor}\label{corollary: underlying simp. set of templicial dg-nerve}
There is a natural isomorphism $\tilde{U}\circ N^{dg}_{k}\simeq N^{dg}$.
\end{Cor}
\begin{proof}
Let $\mathcal{C}$ be a small dg-category over $k$ and $n\geq 0$. By Proposition \ref{proposition: temp. maps into dg-nerve}, a templicial map $\tilde{F}(\Delta^{n})\rightarrow N^{dg}_{k}(\mathcal{C})$ is equivalent to a map of sets $f: [n]\rightarrow \Ob(\mathcal{C})$ with a collection of quiver morphisms $\beta_{m}: \tilde{F}(\Delta^{n})_{m}\rightarrow f^{*}\mathcal{C}_{m-1}$ for $m > 0$ satisfying properties \eqref{equation: temp. map into dg-nerve 1} and \eqref{equation: temp. map into dg-nerve 2}. The map $f$ is equivalent to a choice of objects $A_{0},...,A_{n}\in \Ob(\mathcal{C})$. Further, for $i,j\in [n]$ we have
$$
\tilde{F}(\Delta^{n})_{m}(i,j) = F(\{h\in \simp([m],[n])\mid h(0) = i, h(m) = j\})
$$
and thus we may represent $\beta_{m}$ by a collection of elements $\beta_{i_{0},...,i_{m}}\in \mathcal{C}_{m-1}(A_{i_{0}},A_{i_{m}})$ for $0\leq i_{0}\leq ... \leq i_{m}\leq n$. Then by property \eqref{equation: temp. map into dg-nerve 1}, $\beta_{i,i} = \id_{i}$ and $\beta_{i_{0},...,i_{m}} = 0$ whenever $m\geq 2$ and $i_{p} = i_{p+1}$ for some $p\in [m-1]$. Hence, $\beta_{m}$ is completely determined by the elements $\beta_{\{i_{0} < ... < i_{m}\}}$ with $i_{0} < ... < i_{m}$. Moreover, property \eqref{equation: temp. map into dg-nerve 2} translates to
$$
\partial(\beta_{\{i_{0} < ... < i_{m}\}}) = \sum_{j=1}^{m-1}(-1)^{j-1}\left(\beta_{I\setminus\{i_{j}\}} - m(\beta_{\{i_{0} < ... < i_{j}\}}\otimes \beta_{\{i_{j} < ... < i_{m}\}})\right)
$$
Hence, the pair $((a_{i})_{i},(\beta_{I})_{I})$ is precisely an  $n$-simplex of $N^{dg}(\mathcal{C})_{n}$. We have thus obtained a bijection between $\tilde{U}(N^{dg}_{k}(\mathcal{C}))_{n}$ and $N^{dg}(\mathcal{C})_{n}$. It now follows easily from the definitions that this bijection is natural in $n$ and $\mathcal{C}$.
\end{proof}

\begin{Rem}
Let $\mathcal{C}$ be a dg-category and $\mathcal{C}^{\triangle}$ its associated simplicial category through the Dold-Kan correspondence. Then there is a canonical trivial fibration of simplicial sets $N^{hc}(\mathcal{C}^{\triangle})\rightarrow N^{dg}(\mathcal{C})$ \cite[Tag 00SV]{kerodon}. In a subsequent paper \cite{mertens2023nerves}, we will give a different description of $N^{dg}_{k}$ by means of necklaces, which allows to lift this map to a comparison map between the templicial homotopy coherent and templicial dg-nerve. 
\end{Rem}

\section{Quasi-categories versus Frobenius structures}\label{section: Quasi-categories versus Frobenius structures}

In this section we will compare quasi-categories in $\mathcal{V}$ with Frobenius templicial objects in $\mathcal{V}$. This comparison will make use of the naF (or nonassociative Frobenius) structures from \S\ref{subsection: Frobenius structures}.

In order to endow a quasi-category in $\mathcal{V}$ with a naF-structure, we impose an additional projectivity condition (Definition \ref{definition: proj. temp. obj.}). In \S\ref{subsection: Deg-projective quasi-categories have a naF-structure}, we show that a deg-projective quasi-category indeed has a naF-structure (Proposition \ref{proposition: cofibrant fibrant temp. obj. has naF-structure}).
In particular, this applies to ordinary quasi categories (in $\mathcal{V} = \Set$), whence these can always be equipped with a naF-structure.

To show a converse, in \S \ref{subsection: Templicial modules with a naF-structure are quasi-categories} we restrict to the setting of templicial modules, that is, we take $\mathcal{V} = \Mod(k)$.
Our main result states that every naF-templicial module is a quasi-category in modules (Theorem \ref{theorem: linear naF-structure implies fibrant}). 
We prove this by successively filling particular types of simplicial subsets of the standard simplex. More precisely, we first go from necklaces (which can be filled almost by definition of a naF-structure) to \emph{wings}, which are defined as the union of the two outer faces. Next, we show that filling wings is equivalent to filling horns.
Our proofs make essential use of the linear setup, producing fillings as linear combinations, and Example \ref{wingsex} shows for instance that filling wings does not imply filling horns for ordinary simplicial sets.

As a consequence of Theorem \ref{theorem: linear naF-structure implies fibrant}, the templicial dg-nerve $N^{dg}_{k}$ from Section \ref{section: Enriching the differential graded nerve} lands in quasi-categories. In Proposition \ref{proposition: comparison of linear nerve and dg-nerve}, we further show that $N^{dg}_{k}$ coincides with the templicial nerve from \cite{lowen2023enriched} when restricted to linear categories (i.e. dg-categories concentrated in degree zero).

\subsection{Deg-projective quasi-categories have a naF-structure}\label{subsection: Deg-projective quasi-categories have a naF-structure}

For this subsection, we assume that the forgetful functor $U: \mathcal{V}\rightarrow \Set$ preserves and reflects regular epimorphisms. This holds for our main case of interest $\mathcal{V} = \Mod(k)$.

Let us call a morphism in $\mathcal{V}$ \emph{projective} if it has the left lifting property with respect to all regular epimorphisms in $\mathcal{V}$. In particular, an object $P\in \mathcal{V}$ is regular projective in the usual sense if and only if the morphism $0\rightarrow P$ is projective. Note that since $U$ preserves regular epimorphisms, the monoidal unit $I$ of $\mathcal{V}$ is necessarily regular projective, and more generally so is every $F(S)$ with $S$ a set.

It is easily verified that a morphism in $\mathcal{V}$ is projective if and only if it is a retract of a coprojection $A\rightarrow A\amalg B$ with $B$ a regular projective object. Moreover since $U$ preserves and reflects regular epimorphisms, we obtain a weak factorisation system (projective morphisms, regular epimorphisms) on $\mathcal{V}$. For the following examples, in which all epimorphisms are regular, this  takes the following familiar form:
\begin{itemize}
\item On $\mathcal{V} = \Set$: (injective maps, surjective maps).
\item On $\mathcal{V} = \Mod(k)$: (coprojections $A\rightarrow A\oplus B$ with $B$ a projective $k$-module, surjective linear maps).
\end{itemize}
Note that a quiver morphism $f: Q\rightarrow P$ in $\mathcal{V}\Quiv_{S}$ is a regular epimorphism if and only if for all $a,b\in S$, $f_{a,b}: Q(a,b)\rightarrow P(a,b)$ is a regular epimorphism in $\mathcal{V}$. It is follows that a quiver morphism $f$ is projective if and only if $f_{a,b}$ is projective in $\mathcal{V}$ for all $a,b\in S$, and there is a corresponding weak factorisation system on $\mathcal{V}\Quiv_{S}$.

\begin{Def}\label{definition: proj. temp. obj.}
We call a templicial object \emph{deg-projective} if for all $n\geq 1$, the canonical quiver map
$$
X^{deg}_{n}\rightarrow X_{n}
$$
is projective, where $X^{deg}_{n}$ is the \emph{quiver of degenerate $n$-simplices} of $X$ (also see \cite[Definition 2.16]{lowen2023enriched}). That is,
$$
X^{deg}_{n} = \colim_{\substack{\sigma: [n]\twoheadrightarrow [k]\\ \text{surjective in }\simp\\ 0\leq k < n}}X_{k}
$$
\end{Def}

\begin{Ex}\label{example: simplicial set is projective}
Any simplicial set $K$ is deg-projective when considered as a templicial set. Moreover, the templicial object $\tilde{F}(K)$ in $\mathcal{V}$ is always deg-projective.
\end{Ex}

\begin{Ex}
Consider the ring $\mathbb{Z}/2\mathbb{Z}$ as a one-object $\mathbb{Z}$-linear category. Then its templicial nerve $N_{\mathbb{Z}}(\mathbb{Z}/2\mathbb{Z})$ (see \cite[Definition 2.11]{lowen2023enriched}) is not deg-projective.
\end{Ex}


Given a templicial object $(X,S)$ and integers $0 < j < n$, consider the following limit of quivers
$$
M_{j,n}X = \lim_{\substack{f: T\hookrightarrow \{0 < n\}\\ f\neq \id_{\{0 < n\}}, f\neq \delta_{j}}}\!\!\!X_{T}\in \mathcal{V}\Quiv_{S}
$$
taken over the full subcategory of $(\nec_{/\{0 < n\}})^{op}$ spanned by all the necklace maps $T\hookrightarrow \{0 < n\}$ that are injective on vertices, except $\delta_{j}: \{0 < n-1\}\hookrightarrow \{0 < n\}$ and the identity on $\{0 < n\}$. As $U: \mathcal{V}\rightarrow \Set$ preserves limits, we see that an element of $U(M_{j,n}X(a,b))$ with $a,b\in S$ may be identified with collections $(x_{k})^{n-1}_{k=1}$ and $(y_{i})^{n-1}_{l=1, i\neq j}$ satisfying the conditions of Proposition \ref{proposition: properties of quasi-cats.}.\ref{item: elementwise charac. of quasi-cat.}. Consequently, the condition that $X$ is a quasi-category in $\mathcal{V}$ is equivalent to the condition that the canonical quiver morphism $X_{n}\rightarrow M_{j,n}X$ is a regular epimorphism (as $U$ preserves and reflects regular epimorphisms).

\begin{Prop}\label{proposition: cofibrant fibrant temp. obj. has naF-structure}
Any deg-projective quasi-category in $\mathcal{V}$ has a naF-structure. In particular, every ordinary quasi-category has a naF-structure.
\end{Prop}
\begin{proof}
We define quiver morphisms $Z^{p,q}: X_{p}\otimes X_{q}\rightarrow X_{p+q}$ by induction on $n = p + q$, for all $p,q\geq 0$. Define $Z^{p,0}$ and $Z^{0,q}$ to be the left and right unit isomorphisms. Now take $n > 0$ and let $p,q > 0$ be such that $p + q = n$. Then consider the following commutative diagram:
\[\begin{tikzcd}
	{(X^{deg}_{p}\otimes X_{q})\amalg_{(X^{deg}_{p}\otimes X^{deg}_{q})}(X_{p}\otimes X^{deg}_{q})} & {X_{n}} \\
	{X_{p}\otimes X_{q}} & {M_{p,n}X}
	\arrow[from=1-1, to=2-1]
	\arrow[from=1-1, to=1-2]
	\arrow[from=1-2, to=2-2]
	\arrow[from=2-1, to=2-2]
	\arrow[dashed, from=2-1, to=1-2]
\end{tikzcd}\]
As $X$ is a deg-projective templicial object, it follows that the left vertical morphism is projective. The top horizontal morphism is induced by the already defined morphisms $Z^{k,l}$ with $k + l < n$, which is well-defined by the fact that the $Z^{k,l}$ are natural with respect to the degeneracy maps of $X$. The bottom horizontal morphism is determined by the morphisms
$$
X_{p}\otimes X_{q}\xrightarrow{\xi_{k}} X_{k}\otimes X_{n-k}\quad \text{and}\quad X_{p}\otimes X_{q}\xrightarrow{\zeta_{i}} X_{n-1}
$$
for all $0 < k,i < n$ with $i\neq p$, where
$$
\xi_{k} =
\begin{cases}
(Z^{p,k-p}\otimes \id_{X_{l}})(\id_{X_{p}}\otimes \mu_{k-p,l}), & p\leq k\\
(\id_{X_{k}}\otimes Z^{p-k,q})(\mu_{k,p-k}\otimes \id_{X_{q}}), & p\geq k
\end{cases}
,\quad
\zeta_{i} =
\begin{cases}
Z^{p-1,q}(d_{i}\otimes \id_{X_{q}}), & i < p\\
Z^{p,q-1}(\id_{X_{p}}\otimes d_{i-p}), & i > p
\end{cases}
$$
Hence as the right vertical morphism is a regular epimorphism, there exists a lift $Z^{p,q}: X_{p}\otimes X_{q}\rightarrow X_{n}$ which by construction is natural with respect to the degeneracy and inner face morphisms of $X$, and satisfies the Frobenius equations \eqref{equation: mu-Z compatibility}.

In particular, if $X$ is an ordinary quasi-category, it is deg-projective as a templicial set by Example \ref{example: simplicial set is projective} and thus it has a naF-structure.
\end{proof}

The converse to Proposition \ref{proposition: cofibrant fibrant temp. obj. has naF-structure} is false in general, as Example \ref{example: naF-structure does not imply quasi-cat.} shows. However, in the next subsection we will see that the converse does hold in the linear case (see Theorem \ref{theorem: linear naF-structure implies fibrant}).

\begin{Ex}\label{example: naF-structure does not imply quasi-cat.}
Let $X$ be the simplicial set defined as the colimit of
\begin{center}
\begin{tikzpicture}
\draw
(-2,0) node(0){$\Delta^{3}$}
(-1,1) node(1){$\Lambda^{3}_{3}$}
(0,0) node(2){$\Delta^{3}$}
(1,1) node(3){$\Lambda^{3}_{0}$}
(2,0) node(4){$\Delta^{3}$};

\draw[->] (1) -- (0);
\draw[->] (1) -- (2);
\draw[->] (3) -- (2);
\draw[->] (3) -- (4);
\end{tikzpicture}
\end{center}
It is the standard $3$-simplex $\Delta^{3}$, whose simplices we will represent by their vertices $[i_{0},...,i_{m}]$, with two non-degenerate $3$-simplices $x$ and $y$ glued on. We have
\begin{align*}
\forall i\in\{0,1,2\}: d_{i}(x) = [0,...,\cancel{i},...,3]\quad \text{but }d_{3}(x)\neq [0,1,2]\\
\forall j\in\{1,2,3\}: d_{j}(y) = [0,...,\cancel{j},...,3]\quad \text{but }d_{0}(y)\neq [1,2,3]
\end{align*}

In $X$, not all horns can be filled. Indeed, since
\begin{gather*}
d_{0}d_{3}(x) = d_{2}([1,2,3]) = [1,2] = d_{0}([0,1,2]) = d_{2}d_{0}(y),\\
d_{2}d_{3}(x) = [0,1] = d_{2}([0,1,3])\quad \text{and}\quad d_{1}d_{0}(y) = [1,3] = d_{0}([0,1,3])
\end{gather*}
the faces $d_{3}(x)$, $d_{0}(y)$ and $[0,1,3]$ form a horn $\Lambda^{3}_{1}$ in $X$. But there is no $3$-simplex in $X$ with these faces.

However, $X$ does have a naF-structure. It suffices to define $Z^{p,q}(a,b)$ on non-degenerate simplices $a$ and $b$. For those contained in $\Delta^{3}$, define
$$
Z^{p,q}([i_{0},...,i_{p}],[i_{p},...,i_{p+q}]) = [i_{0},...,i_{p+q}]
$$
note that this includes all edges of $X$. Further, set
$$
Z^{2,1}(d_{3}(x),[2,3]) = x,\quad \text{and}\quad Z^{1,2}([0,1],d_{0}(y)) = y
$$
It is easy to check that this satisfies the Frobenius equations \eqref{equation: mu-Z compatibility}.
\end{Ex}

Proposition \ref{proposition: cofibrant fibrant temp. obj. has naF-structure} does not hold without assuming deg-projectivity.

\begin{Ex}
Let $\mathcal{V} = \Mod(\mathbb{Z}) = \Ab$ and consider the $\Ab$-enriched quiver $Q$ with vertex set $S = \{a,b\}$ and
$$
Q(x,y) =
\begin{cases}
\mathbb{Z} & \text{if }x = a, y = b\\
\mathbb{Z}/2\mathbb{Z} & \text{if }x = y\\
0 & \text{otherwise}
\end{cases}
$$ 
The unit of $Q$ is given by the quotient map $q: \mathbb{Z}\twoheadrightarrow \mathbb{Z}/2\mathbb{Z}$. We can extend $Q$ to a templicial object $X$ by setting $X_{n}(a,b) = \mathbb{Z}^{\times n}$, $X_{n}(b,a) = 0$ and $X_{n}(x,x) = \mathbb{Z}/2\mathbb{Z}$ for $x\in \{a,b\}$. Then $X$ is easily seen to be a quasi-category in $\Ab$, but the canonical map $X^{deg}_{1}(a,a)\rightarrow X_{1}(a,a)$ is given by $q$ which is not projective. Note that $X$ does not have a naF-structure, because this would require the existence of a map
$$
Z^{1,1}: (X_{1}\otimes X_{1})(a,b)\simeq \mathbb{Z}/2\mathbb{Z}\oplus \mathbb{Z}/2\mathbb{Z}\rightarrow X_{2}(a,b)\simeq \mathbb{Z}\oplus \mathbb{Z}
$$
which is a section of $\mu_{1,1} = q\oplus q: \mathbb{Z}\oplus \mathbb{Z}\twoheadrightarrow \mathbb{Z}/2\mathbb{Z}\oplus \mathbb{Z}/2\mathbb{Z}$.
\end{Ex}

\subsection{Templicial modules with a naF-structure are quasi-categories}\label{subsection: Templicial modules with a naF-structure are quasi-categories}

Our main result in this subsection is Theorem \ref{theorem: linear naF-structure implies fibrant} which states that every templicial $k$-module with a Frobenius structure is a quasi-category in modules. As a consequence, we find that the templicial dg-nerve of any dg-category is a quasi-category in modules (Corollary \ref{corollary: templicial dg-nerve is fibrant}).

We start by introducing the \emph{wings} $W^{n}$ of a simplex $\Delta^{n}$ for $n\geq 2$, which are defined as the union of its two outer faces. Given a necklace $(T,n)$ and $0 < j < n$, the unique inert necklace map $T\hookrightarrow \{0 < n\}$ can thus be identified with a composite of inclusions of bipointed simplicial sets:
$$
T\subseteq W^{n}\subseteq \Lambda^{n}_{j}\subseteq \Delta^{n}
$$
By design, a naF-structure on a templicial object $X$ allows to fill up any necklace $T$ in $X$ to a simplex via the morphism $Z^{T}: X_{T}\rightarrow X_{n}$. To prove Theorem \ref{theorem: linear naF-structure implies fibrant}, we show successively that we can fill each of the above simplicial subsets to an entire simplex. First from necklaces to wings (Proposition \ref{proposition: nonassociative Frobenius implies linear wedges}) and then from wings to horns (Proposition \ref{proposition: linear wedges equiv. fibrant}).

\begin{Def}
For $n\geq 2$, we write $W^{n}$ for the simplicial subset of $\Delta^{n}$ defined by
$$
W^{n}([m]) = \{f: [m]\rightarrow [n]\mid f(m) < n\text{ or }f(0) > 1\}
$$
for all $m\geq 0$. We call $W^{n}$ the \emph{wings} of $\Delta^{n}$. It consists of the $0$th and $n$th face of $\Delta^{n}$. We say a functor $Y: \nec^{op}\rightarrow \mathcal{V}$ \emph{lifts wings} if for all $n\geq 2$, any lifting problem in $\mathcal{V}^{\nec^{op}}$:
\[\begin{tikzcd}
	{\tilde{F}(W^{n})_{\bullet}(0,n)} & {Y} \\
	{\tilde{F}(\Delta^{n})_{\bullet}(0,n)}
	\arrow[from=1-1, to=1-2]
	\arrow[from=1-1, to=2-1]
	\arrow[dashed, from=2-1, to=1-2]
\end{tikzcd}\]
where the vertical morphism is induced by the inclusion $W^{n}\subseteq \Delta^{n}$, has a solution.
\end{Def}

\begin{Prop}\label{proposition: wedges as unions of faces}
For all $n\geq 2$, we have
$$
W^{n}_{\bullet}(0,n) = \bigcup_{k=1}^{n-1}(\Delta^{k}\vee \Delta^{n-k})_{\bullet}(0,n)
$$
as a subfunctor of $\Delta^{n}_{\bullet}(0,n)$.
\end{Prop}
\begin{proof}
Following Example \ref{example: hom-object of necklace cat. assoc. to simp. set}, it suffices to note that for any necklace $(T,p)$, a map $f: T\rightarrow \Delta^{n}$ in $\SSet_{*,*}$ factors through $W^{n}\subseteq \Delta^{n}$ if and only if there exists a $k\in T$ such that $0 < f(k) < n$.
\end{proof}

\begin{Lem}\label{lemma: wedge inclusions are inner anodyne}
For all $n\geq 2$, the inclusion $W^{n}_{\bullet}(0,n)\hookrightarrow \Delta^{n}_{\bullet}(0,n)$ belongs to the weak saturated closure of the set
$$
\left\lbrace(\Lambda^{n}_{j})_{\bullet}(0,n)\hookrightarrow \Delta^{n}_{\bullet}(0,n)\,\middle\vert\, 0 < j < n\right\rbrace
$$
\end{Lem}
\begin{proof}
Denote the weak saturated closure of the above set by $\mathcal{A}$. For $0 < k < n$, let us denote by $A^{n}_{k}$ the simplicial subset of $\Delta^{n}$ given by the union of the faces $0,...,k-1$ and $n$. We consider it as an object of $\SSet_{*,*}$ with distinguished points given by $0$ and $n$. Then $A^{n}_{k}$ contains all vertices of $\Delta^{n}$ and
$$
(A^{n}_{k})_{\bullet}(0,n) = W^{n}_{\bullet}(0,n)\cup \bigcup_{j=1}^{k-1}\delta_{j}(\Delta^{n-1})_{\bullet}(0,n)
$$
by Proposition \ref{proposition: wedges as unions of faces}. We will show by double induction on $n\geq 2$ and $0 < k < n$ that the inclusion
\begin{equation}\label{diagram: wedge inclusions are anodyne}
(A^{n}_{k})_{\bullet}(0,n) \hookrightarrow \Delta^{n}_{\bullet}(0,n)
\end{equation}
belongs to $\mathcal{A}$. The result then follows by choosing $k = 1$.

If $k = n-1$, then \eqref{diagram: wedge inclusions are anodyne} coincides with the horn inclusion $(\Lambda^{n}_{n-1})_{\bullet}(0,n)\hookrightarrow \Delta^{n}_{\bullet}(0,n)$ by \cite[	Proposition 5.1]{lowen2023enriched}. Note that this covers the entire case $n = 2$.

Assume further that $k <  n - 1$ and let $(T,p)$ be a necklace. For a map $h: T\rightarrow \Delta^{n}$ in $\SSet_{*,*}$ and $0 < i < n$, $h$ factors through $(\Delta^{i}\vee \Delta^{n-i})_{\bullet}(0,n)$ if and only if $i\in h(T)$, and $h$ factors through $\delta_{i}(\Delta^{n-1})$ if and only if $h([p])\subseteq [n]\setminus \{i\}$. Now take a map $g: T\rightarrow \Delta^{n-1}$ in $\SSet_{*,*}$. It follows that $\delta_{k}g: T\rightarrow \Delta^{n}$ factors through $A^{n}_{k}\cap \delta_{k}(\Delta^{n-1})$ if and only if $g$ factors through $A^{n-1}_{k}$. Hence, we have a pushout diagram in $\Set^{\nec^{op}}$:
\[\begin{tikzcd}
	{(A^{n-1}_{k})_{\bullet}(0,n-1)} & {(A^{n}_{k})_{\bullet}(0,n)} \\
	{\Delta^{n-1}_{\bullet}(0,n-1)} & {(A^{n}_{k+1})_{\bullet}(0,n)}
	\arrow["{\delta_{k}}", from=1-1, to=1-2]
	\arrow[hook, from=1-1, to=2-1]
	\arrow["{\delta_{k}}"', from=2-1, to=2-2]
	\arrow[hook, from=1-2, to=2-2]
\end{tikzcd}\]
By the induction hypothesis, the left vertical map belongs to $\mathcal{A}$ and thus so does the right vertical map. Further, the inclusion $(A^{n}_{k+1})_{\bullet}(0,n)\hookrightarrow \Delta^{n}_{\bullet}(0,n)$ also belongs to $\mathcal{A}$. This completes the proof.
\end{proof}

\begin{Prop}\label{proposition: fibrant implies wedges}
Let $(X,S)$ be a quasi-category in $\mathcal{V}$. Then for every $a,b\in S$, the functor $X_{\bullet}(a,b): \nec^{op}\rightarrow \mathcal{V}$ lifts wings.
\end{Prop}
\begin{proof}
This is an immediate consequence of Lemma \ref{lemma: wedge inclusions are inner anodyne}.
\end{proof}

\begin{Prop}\label{proposition: nonassociative Frobenius implies linear wedges}
Let $(X,S)$ be a templicial $k$-module with a naF-structure. Then for every $a,b\in S$, the functor $X_{\bullet}(a,b): \nec^{op}\rightarrow \Mod(k)$ lifts wings.
\end{Prop}
\begin{proof}
Let $Z$ denote the naF-structure of $X$. Take $n\geq 2$ and $a,b\in S$. By Proposition \ref{proposition: wedges as unions of faces}, a morphism $\tilde{F}(W^{n})_{\bullet}(0,n)\rightarrow X_{\bullet}(a,b)$ in $\Mod(k)^{\nec^{op}}$ corresponds to a collection $(x_{k})_{k=1}^{n-1}$ of elements with $x_{k}\in U((X_{k}\otimes_{S} X_{n-k})(a,b))$ for all $0 < k < n$ such that for all $0 < k < l < n$ we have
\begin{equation}\label{equation: nonassociative Frobenius implies linear wedges}
(\id_{X_{k}}\otimes \mu_{l-k,n-l})(x_{k}) = (\mu_{k,l-k}\otimes \id_{X_{n-l}})(x_{l})
\end{equation}
To extend the above morphism to $\tilde{F}(\Delta^{n})_{\bullet}(0,n)$, we must find some $z\in X_{n}(a,b)$ such that $\mu_{k,n-k}(z) = x_{k}$ for all $0 < k < n$.

Given a necklace $(T,n)$ with $\ell(T)\geq 2$, we can choose $k\in T\setminus \{0,n\}$. Consider the splitting $(T_{1},T_{2})$ of $T$ over $\{0 < k < n\}$. Then set
$$
x_{T} = (\mu_{T_{1}}\otimes \mu_{T_{2}})(x_{k})\in U(X_{T}(a,b))
$$
Note that by \eqref{equation: nonassociative Frobenius implies linear wedges}, this expression does not depend on the choice of $k$. Then it follows by Proposition \ref{proposition: higher mu-Z compatibility} and Corollary \ref{corollary: higher mu-Z compatibility}.1 that
\begin{align*}
\mu_{k,n-k}Z^{T}(x_{T}) &= (Z^{T_{1}}\otimes Z^{T_{2}})(\id\otimes \mu_{T'}\otimes \id)(x_{T})\\
&= (Z^{T_{1}}\otimes Z^{T_{2}})(x_{T\cup\{k\}}) = \mu_{k,n-k}Z^{T\cup\{k\}}(x_{T\cup\{k\}})
\end{align*}
where $T'$ is some necklace with $\ell(T') = 2$ and the second equality follows from \eqref{equation: nonassociative Frobenius implies linear wedges}. Now consider
$$
z = \sum_{\substack{(T,n)\in \nec\\ \ell(T)\geq 2}}(-1)^{\ell(T)}Z^{T}(x_{T})
$$
Then it follows analogously to Lemma \ref{lemma: properties of retraction of kernel embedding}.\ref{item: properties of retraction of kernel embedding 1} that $\mu_{k,n-k}(z - Z^{k,n-k}(x_{k})) = 0$ and thus $\mu_{k,n-k}(z) = x_{k}$.
\end{proof}

This next result is shown using a similar argument as the classical proof that every simplicial group is a Kan complex (see for instance \cite{moore1958semi}).

\begin{Prop}\label{proposition: linear wedges equiv. fibrant}
Let $(X,S)$ be a templicial $k$-module. Then the following statements are equivalent.
\begin{enumerate}[(1)]
\item $X$ is a quasi-category in modules.
\item For all $a,b\in S$, the functor $X_{\bullet}(a,b): \nec^{op}\rightarrow \Mod(k)$ lifts wings.
\end{enumerate}
\end{Prop}
\begin{proof}
If $X$ is a quasi-category in modules, then $(2)$ holds by Proposition \ref{proposition: fibrant implies wedges}. Conversely, take $0 < j < n$, $a,b\in S$ and let $(x_{k})_{k=1}^{n-1}$ and $(y_{i})_{i=1,i\neq j}^{n-1}$ be collections of elements satisfying the conditions of Proposition \ref{proposition: properties of quasi-cats.}.3. Consider the following condition on elements $z\in U(X_{n}(a,b))$:
\begin{equation}\label{condition: compatibility with wedge}
\mu_{k,l}(z) = x_{k}\quad (\text{for all } 0 < k < n)
\end{equation}
Let us start by noting that if $z\in X_{n}$ satisfies \eqref{condition: compatibility with wedge}, then we have for all $0 < k < n$:
\begin{align*}
&\mu_{k,n-k}(s_{i}(y_{i} - d_{i}(z))) = 0 &(\text{for all }0 < i < j)\\
&\mu_{k,n-k}(s_{i-1}(y_{i} - d_{i}(z))) = 0 &(\text{for all }j < i < n)
\end{align*}
Indeed, for the first equation, there are three cases:
\begin{align*}
&\mu_{k,n-k}(s_{i}(y_{i} - d_{i}(z)))\\
&= \begin{cases}
(s_{i}\otimes \id_{X_{n-k}})(\mu_{k-1,n-k}(y_{i}) - (d_{i}\otimes \id_{X_{n-k}})(x_{k})) & \text{if }i < k\\
(\id_{X_{k}}\otimes s_{0})(\mu_{k,n-k-1}(y_{k}) - (d_{k}\otimes \id_{X_{n-k-1}})(x_{k+1})) & \text{if }i = k\\
(\id_{X_{k}}\otimes s_{i-k})(\mu_{k,n-k-1}(y_{i}) - (\id_{X_{k}}\otimes d_{i-k})(x_{k})) & \text{if }i > k
\end{cases}\\
&= 0
\end{align*}
The second equation follows similarly.

Now assuming $(2)$, there exists an element $z^{0}\in U(X_{n}(a,b))$ satisfying condition \eqref{condition: compatibility with wedge}. Then define, inductively on $l\in \{1,...,j-1\}$:
$$
z^{l} = z^{l-1} + s_{l}(y_{l} - d_{l}(z^{l-1}))\in U(X_{n}(a,b))
$$
By the previous remarks, each $z^{l}$ satisfies (\ref{condition: compatibility with wedge}). We then prove by induction on $l$ that for all $0 < i \leq l$:
$$
d_{i}(z^{l}) = y_{i}
$$
Indeed, for $l = 0$ this is trivial and if $l > 0$ we have:
$$
d_{i}(z^{l}) =
\begin{cases}
y_{i} - s_{l-1}(d_{i}(y_{l}) - d_{l-1}(y_{i})) & \text{if }i < l\\
d_{l}(z^{l-1}) + y_{l} - d_{l}(z^{l-1}) & \text{if }i = l
\end{cases}\quad = y_{i}
$$

Finally, set $z^{n} = z^{j-1}$ and define inductively on $l\in \{j+1,...,n-1\}$:
$$
z^{l} = z^{l+1} + s_{l-1}(y_{l} - d_{l}(z^{l+1}))
$$
Then again $z^{l}$ satisfies (\ref{condition: compatibility with wedge}) for all $j < l \leq n$. We prove by induction on $l$ that for all $i\in\{1,...,j-1\}\cup\{l,...,n-1\}$:
$$
d_{i}(z^{l}) = y_{i}
$$
Again this is trivial for $l = n$ and if $l < n$ we have
$$
d_{i}(z^{l}) =
\begin{cases}
y_{i} - s_{l-2}(d_{i}(y_{l}) - d_{l-1}(y_{i})) & \text{if }i < l-1\\
d_{l}(z^{l+1}) + y_{l} - d_{l}(z^{l+1}) & \text{if }i = l\\
y_{i} - s_{l-1}(d_{i-1}(y_{l}) - d_{l}(y_{i})) & \text{if }i > l
\end{cases}
\quad = y_{i}
$$
Note that the case $i = l-1$ does not occur. Thus it suffices to set $z = z^{j+1}$.
\end{proof}

The previous proposition does not hold for ordinary simplicial sets, as the following example shows.

\begin{Ex}\label{wingsex}
Consider the simplicial set $X = \Delta^{3}\amalg_{\partial\Delta^{2}}\Delta^{2}$, gluing an extra $2$nd face to the standard $3$-simplex. Formally, it is the pushout of the inclusion $\partial\Delta^{2}\subseteq \Delta^{2}$ along the map $\partial\Delta^{2}\rightarrow \Delta^{3}$
sending vertices $0\mapsto 0$, $1\mapsto 1$ and $2\mapsto 3$. Denote the simplices of $\Delta^{3}$ by ordered sequences $[i_{0},...,i_{m}]$ and denote the extra face by $x\in X_{2}$. We then have $d_{0}(x) = [1,3]$, $d_{1}(x) = [0,3]$ and $d_{2}(x) = [0,1]$, but $x\neq [0,1,3]$.

Then $X$ is certainly not a quasi-category as there exists no $3$-simplex $z$ with $d_{0}(z) = [1,2,3]$, $d_{2}(z) = x$ and $d_{3}(z) = [0,1,2]$.

However, all wings in $X$ can be filled. Indeed, a map $\alpha: W^{n}\rightarrow X$ is uniquely determined by simplices $y,z\in X_{n-1}$ such that $d_{0}(y) = d_{n-1}(z)$. If either $y$ or $z$ is degenerate, $\alpha$ extends trivially to $\Delta^{n}$. Assuming they are both non-degenerate, we have either $n = 2$ or $n = 3$. As $W^{2} = \Lambda^{2}_{1}$ and the quasi-category $\Delta^{3}$ contains all edges of $X$, the case $n = 2$ is covered. If $n = 3$, we must have $y = [0,1,2]$ and $z = [1,2,3]$, which can be filled by $[0,1,2,3]$ itself.
\end{Ex}

\begin{Thm}\label{theorem: linear naF-structure implies fibrant}
Let $X$ be a templicial $k$-module with a naF-structure. Then $X$ is a quasi-category in modules.
\end{Thm}
\begin{proof}
Combine Propositions \ref{proposition: nonassociative Frobenius implies linear wedges} and \ref{proposition: linear wedges equiv. fibrant}.
\end{proof}

\begin{Cor}\label{corollary: free functor preserves quasi-cat.}\label{corfree}
Let $\mathcal{C}$ be an ordinary quasi-category. Then $\tilde{F}(\mathcal{C})$ is a quasi-category in modules.
\end{Cor}
\begin{proof}
By Proposition \ref{proposition: cofibrant fibrant temp. obj. has naF-structure}, $\mathcal{C}$ has a naF-structure. It is easy to see that $\tilde{F}(\mathcal{C})$ has an induced naF-structure as well and thus the result follows from Theorem \ref{theorem: linear naF-structure implies fibrant}.
\end{proof}

\begin{Cor}\label{corollary: temp. obj. assoc. to augmented simp. cat. is fibrant}
Let $\mathcal{C}$ be a small $S^{+}\Mod(k)$-category. Then the underlying templicial $k$-module of $\mathcal{T}(\mathcal{C})$ is a quasi-category in modules.
\end{Cor}
\begin{proof}
This immediately follows from Theorem \ref{theorem: linear naF-structure implies fibrant}.
\end{proof}

\begin{Cor}\label{corollary: templicial dg-nerve is fibrant}
Let $\mathcal{C}$ be a small dg-category over $k$. Then its templicial dg-nerve $N^{dg}_{k}(\mathcal{C})$ is a quasi-category in modules.
\end{Cor}
\begin{proof}
Apply Corollary \ref{corollary: temp. obj. assoc. to augmented simp. cat. is fibrant} to the $S^{+}\Mod(k)$-category $\Gamma^{+}(\mathcal{C})$.
\end{proof}

To end the section, let us compare the templicial dg-nerve to the templicial nerve $N_{k}: k\Cat\rightarrow \ts\Mod(k)$ of \cite{lowen2023enriched}. Let $\iota: k\Cat\hookrightarrow k\Cat_{dg}$ denote the inclusion which considers linear categories as dg-categories concentrated in degree $0$, and let $H_{0}: k\Cat_{dg}\rightarrow k\Cat$ denote the functor taking $0$th homology. The templicial nerve $N_{k}$ has a left-adjoint $h_{k}$. If $X$ is a quasi-category in modules, the underlying ordinary category of $h_{k}X$ is isomorphic to $h\tilde{U}(X)$, the homotopy category of $\tilde{U}(X)$.

\begin{Prop}\label{proposition: comparison of linear nerve and dg-nerve}
We have natural isomorphisms
$$
N^{dg}_{k}\circ \iota \simeq N_{k}\quad \text{and}\quad h_{k}\circ N^{dg}_{k}\simeq H_{0}
$$
\end{Prop}
\begin{proof}
Let us denote the functor from right to left in the equivalence of Proposition \ref{proposition: Frob. temp. obj. and dg-cat. equivalence} by $DG = N^{+}_{\bullet}\mathcal{K}: \Fs\Mod(k)\rightarrow k\Cat_{dg,\geq 0}$. Clearly, $\iota$ factors through $k\Cat_{dg,\geq 0}$ and the templicial nerve functor $N_{k}$ factors through $\Fs\Mod(k)$ by Example \ref{example: nerve has Frob structure}. Therefore, it suffices to show that we have natural isomorphisms
$$
\iota\simeq DG\circ N_{k}\quad \text{and}\quad h_{k}\simeq H_{0}\circ DG
$$

Let $\mathcal{C}$ be a small $k$-linear category. Since the comultiplication maps of $N_{k}(\mathcal{C})$ are invertible, we have that the $S^{+}\Mod(k)$-category $\mathcal{K}(N_{k}(\mathcal{C}))$ is concentrated in degree $-1$ and thus $DG(N_{k}(\mathcal{C}))$ is concentrated in degree $0$. It follows that $DG\circ N_{k}$ is naturally isomorphic to $\iota$.

Let $(X,S)$ be a Frobenius templicial $k$-module. Then it is a quasi-category in modules by Theorem \ref{theorem: linear naF-structure implies fibrant}. Boiling down the definitions, we see that the set of objects of $DG(X)$ is $S$ as well and that for every $a\in S$, the degenerate 1-simplex $s_{0}(a)$ represents the identity in both $h_{k}X$ and $H_{0}(DG(X))$. Take $a,b,c\in S$. Then the differential $\partial: DG_{1}(X)(a,c)\rightarrow DG_{0}(X)(a,c)$ is just the restriction $d_{1}\vert_{\ker(\mu_{1,1})}: \ker(\mu_{1,1})(a,c)\rightarrow X_{1}(a,c)$. Hence, for any three $f\in X_{1}(a,b)$, $g\in X_{1}(b,c)$ and $h\in X_{1}(a,c)$, the composition $gf$ is homologous to $h$ in $DG(X)$ if and only if there exists a $w\in \ker(\mu_{1,1})(x,z)$ such that $d_{1}(w) = h - gf$. This is equivalent to the existence of a templicial map $\alpha: \tilde{F}(\Delta^{2})\rightarrow X$ with $\alpha_{0,1} = 0$, $\alpha_{1,2} = s_{0}(x)$ and $\alpha_{0,2} = h - gf$ (using the notation of Proposition \ref{proposition: properties of temp. objs.}.\ref{item: simplex of underlying simp. set}). In other words, $[h - gf] = [0]$ in $h\tilde{U}(X)$ and thus $[g]\circ [f] = [h]$ in $h_{k}X$. Specializing to the case $f = s_{0}(x)$, we find that $[g] = [h]$ in $H_{0}(DG(X))$ if and only if $[g] = [f]$ in $h_{k}X$. This shows that $[f]\mapsto [f]$ defines an isomorphism of $k$-linear categories
$$
h_{k}X\simeq H_{0}(DG(X))
$$
It follows easily that this isomorphism is natural in $X$.
\end{proof}

\appendix

\section{Frobenius structures and $S^{+}\Mod(k)$-categories}\label{section: Frobenius structures and augmented simplicial categories}

The goal of this section is to prove Theorem \ref{theorem: augmented simp. cat. and Frob. temp. obj. equiv.}. We do this by explicitly constructing the functors $\mathcal{T}$ and $\mathcal{K}$ and then showing they are inverse to each other. Let us start by constructing the functor $\mathcal{T}:  k\Cat_{\Delta_{+}}\rightarrow \Fs\Mod(k)$.

\begin{Prop}\label{proposition: Frob. temp. obj. assoc. to augmented simp. cat.}
Let $\mathcal{C}$ be a small $S^{+}\Mod(k)$-category with object set $S$. Then $\mathcal{T}(\mathcal{C})$ of Construction \ref{construction: temp. obj. assoc. to augmented simp. cat.} is a well-defined Frobenius templicial $k$-module.
\end{Prop}
\begin{proof}
We first show that $\mathcal{T}(\mathcal{C}): \fint^{op}\rightarrow k\Quiv_{\Ob(\mathcal{C})}$ is a well-defined functor. Take morphisms $f: [m]\rightarrow [n]$ and $g: [n]\rightarrow [p]$ in $\fint$ and $(T,p)\in \nec$. Setting $U = g^{-1}(T)$ and $V = f^{-1}(U)$, we must show that
$$
\mathcal{C}^{\otimes}(f)_{U}\circ \mathcal{C}^{\otimes}(g)_{T} = \mathcal{C}^{\otimes}(gf)_{T}
$$
By the functoriality of the monoidal product $-\otimes_{S} -$, we may assume that $\ell(V) = 1$. Now write $U = \{0 = u_{0} < u_{1} < ... < u_{l} = n\}$ and let $k = \ell(T)$. It follows from the naturality and the associativity of $m$ that
\begin{align*}
&\mathcal{C}^{\otimes}(f)_{U}\circ \mathcal{C}^{\otimes}(g)_{T}\\
&\quad = \mathcal{C}(f\vert_{V^{c}_{1}})m_{U^{c}_{1},...,U^{c}_{l}}\left(\mathcal{C}(g\vert_{U^{c}_{1}})m_{T^{c}_{1},...,T^{c}_{p_{1}}}\otimes_{S} ... \otimes_{S} \mathcal{C}(g\vert_{U^{c}_{l}})m_{T^{c}_{p_{l-1}+1},...,T^{c}_{k}}\right)\\
&\quad = \mathcal{C}((g\vert_{U^{c}_{1}}\sqcup ... \sqcup g\vert_{U^{c}_{l}})f\vert_{V^{c}_{1}})m_{T^{c}_{1},...,T^{c}_{k}}\\
&\quad = \mathcal{C}(gf\vert_{V^{c}_{1}})m_{T^{c}_{1},...,T^{c}_{k}} = \mathcal{C}^{\otimes}(gf)_{T}
\end{align*}
Further, for a necklace $(T,n)$ with $k = \ell(T)$ we have
$$
\mathcal{C}^{\otimes}(\id_{[n]})_{T} = \mathcal{C}(\id_{T^{c}_{1}})m_{T^{c}_{1}}\otimes_{S} ...\otimes_{S} \mathcal{C}(\id_{T^{c}_{k}})m_{T^{c}_{k}} = \id_{\mathcal{C}_{T^{c}_{1}}}\otimes_{S} ...\otimes_{S} \id_{\mathcal{C}_{T^{c}_{k}}}
$$

Next we verify that the comultiplication and multiplication maps defined in Construction \ref{construction: temp. obj. assoc. to augmented simp. cat.} make $\mathcal{T}(\mathcal{C})$ into a Frobenius monoidal functor. First note that by definition, $\mathcal{T}(\mathcal{C})_{0}\simeq k_{S}$ is the monoidal unit of $k\Quiv_{S}$ and $\mathcal{C}^{\otimes}_{T}\otimes_{S} \mathcal{C}^{\otimes}_{U}\simeq \mathcal{C}^{\otimes}_{T\vee U}$ for all necklaces $T$ and $U$. Now take $f: [k]\rightarrow [p]$ and $g: [l]\rightarrow [q]$ in $\fint$, and $(T,p),(U,q)\in \nec$. Then we have $f^{-1}(T)\vee g^{-1}(U) = (f + g)^{-1}(T\vee U)$ and it follows that under the above isomorphism:
$$
\mathcal{C}^{\otimes}(f + g)_{T\vee U} = \mathcal{C}^{\otimes}(f)_{T}\otimes_{S} \mathcal{C}^{\otimes}(g)_{U}
$$
From this it is easy to see that $\mu_{k,l}$ and $Z^{k,l}$ are natural in $k,l\geq 0$.

We complete the proof by showing that the maps $\mu_{k,l}$ and $Z^{k,l}$ satisfy the Frobenius equation \eqref{equation: mu-Z compatibility}. Take $k,l,p,q\geq 0$ such that $k + l = p + q$ and assume that $k\geq p$. Then for all $(T,p+q)\in \nec$ with $p\in T$ we have
$$
(Z^{p,k-p}\otimes_{S} \id_{\mathcal{T}(\mathcal{C})_{l}})(\id_{\mathcal{T}(\mathcal{C})_{p}}\otimes_{S} \mu_{k-p,l})\iota_{T} =
\begin{cases}
\iota_{T} & \text{if }k\in T\\
0 & \text{if }k\not\in T
\end{cases}
= \mu_{k,l}Z^{p,q}\iota_{T}
$$
A similar proof shows the case $k\leq p$.
\end{proof}

\begin{Prop}\label{proposition: tensor algebra functor}
The assignment $\mathcal{C}\mapsto (\mathcal{T}(\mathcal{C}),\Ob(\mathcal{C}))$ of Proposition \ref{proposition: Frob. temp. obj. assoc. to augmented simp. cat.} extends to a functor
$$
\mathcal{T}: k\Cat_{\Delta_{+}}\rightarrow \Fs\Mod(k)
$$
\end{Prop}
\begin{proof}
Let $H: \mathcal{C}\rightarrow \mathcal{D}$ be an $S^{+}\Mod(k)$-enriched functor between small $S^{+}\Mod(k)$-categories and let $f: S\rightarrow T$ denote its object map. For every finite linearly ordered set $J$, we have a quiver map $\mathcal{C}_{J}\rightarrow f^{*}(\mathcal{D}_{J})$ in $k\Quiv_{S}$. Denote its adjoint in $k\Quiv_{T}$ by $H_{J}: f_{!}(\mathcal{C}_{J})\rightarrow \mathcal{D}_{J}$. Then for any necklace $(U,n)$ with $k = \ell(U)$, define:
$$
H^{\otimes}_{U}: f_{!}(\mathcal{C}^{\otimes}_{U})\rightarrow f_{!}(\mathcal{C}_{U^{c}_{1}})\otimes_{T} ... \otimes_{T} f_{!}(\mathcal{C}_{U^{c}_{k}})\xrightarrow{H_{U^{c}_{1}}\otimes_{T} ...\otimes_{T} H_{U^{c}_{k}}} \mathcal{D}^{\otimes}_{U}
$$
Then for all $n\geq 0$, consider the following quiver maps in $k\Quiv_{T}$:
$$
\mathcal{T}(H)_{n}: f_{!}(\mathcal{T}(\mathcal{C})_{n})\simeq \bigoplus_{(U,n)\in \nec}f_{!}(\mathcal{C}^{\otimes}_{U})\xrightarrow{\bigoplus_{U}H^{\otimes}_{U}} \bigoplus_{(U,n)\in \nec}\mathcal{D}^{\otimes}_{U} = \mathcal{T}(\mathcal{D})_{n}
$$

Now it follows from the functoriality of $H$ that
\begin{itemize}
\item $H_{J}f_{!}(\mathcal{C}(h)) = \mathcal{D}(h)H_{I}$ for all $h: I\rightarrow J$ in $\Lin$, and
\item for all $I,J\in \Lin$, the following diagram commutes
\[\begin{tikzcd}
	{f_{!}(\mathcal{C}_{I\sqcup J})} & {\mathcal{D}_{I\sqcup J}} \\
	{f_{!}(\mathcal{C}_{I}\otimes_{S} \mathcal{C}_{J})} & {f_{!}(\mathcal{C}_{I})\otimes_{T} f_{!}(\mathcal{C}_{J})} & {\mathcal{D}_{I}\otimes_{T}\mathcal{D}_{J}}
	\arrow[from=2-1, to=2-2]
	\arrow["{m^{\mathcal{C}}_{I,J}}", from=2-1, to=1-1]
	\arrow["{H_{I\sqcup J}}", from=1-1, to=1-2]
	\arrow["{m^{\mathcal{D}}_{I,J}}"', from=2-3, to=1-2]
	\arrow["{H_{I}\otimes_{T} H_{J}}"', from=2-2, to=2-3]
\end{tikzcd}\]
\end{itemize}
From this it easily follows that the quiver maps $(\mathcal{T}(H)_{n})_{n\geq 0}$ define a natural transformation $\mathcal{T}(H): f_{!}\mathcal{T}(C)\rightarrow \mathcal{T}(\mathcal{D})$ between functors $\fint^{op}\rightarrow k\Quiv_{T}$. Further, it is clear that for all necklaces $U$ and $V$, $H^{\otimes}_{U\vee V}$ is equal to the composite
$$
f_{!}(\mathcal{C}^{\otimes}_{U\vee V})\rightarrow f_{!}(\mathcal{C}^{\otimes}_{U})\otimes_{T} f_{!}(\mathcal{C}^{\otimes}_{V})\xrightarrow{H^{\otimes}_{U}\otimes H^{\otimes}_{V}} \mathcal{D}^{\otimes}_{U}\otimes_{T}\mathcal{D}^{\otimes}_{V}\simeq \mathcal{D}^{\otimes}_{U\vee V}
$$
From this it is easy to see that the natural transformation $\mathcal{T}(H)$ is monoidal and satisfies \eqref{diagram: Frob. temp. map}. Thus $(\mathcal{T}(H),f): (\mathcal{T}(C),S)\rightarrow (\mathcal{T}(\mathcal{D}),T)$ a Frobenius templicial map.

It now immediately follows from the definitions that this defines a functor
$$
\mathcal{T}: k\Cat_{\Delta_{+}}\rightarrow \Fs\Mod(k)
$$
\end{proof}

Let us now construct the inverse $\mathcal{K}: \Fs\Mod(k)\rightarrow k\Cat_{\Delta_{+}}$ of $\mathcal{T}$.

\begin{Lem}\label{lemma: narrow morphism induces map on joint kernels}
Let $(X,S)$ be a templicial $k$-module with comultiplication $\mu$ and $m,n\geq 1$. Let $f: [m]\rightarrow [n]$ be an order morphism such that $f^{-1}(\{0\}) = \{0\}$ and $f^{-1}(\{n\}) = \{m\}$. Then the quiver map $X(f): X_{n}\rightarrow X_{m}$ restricts to
$$
\mathcal{K}(X)_{n-2}\rightarrow \mathcal{K}(X)_{m-2}
$$
\end{Lem}
\begin{proof}
Take $a,b\in S$ and $x\in X_{n}(a,b)$ such that $\mu_{k,n-k}(x) = 0$ for all $0 < k < n$. Then for all $0 < p < m$, there exist unique morphisms $f_{1}: [p]\rightarrow [k]$ and $f_{2}: [m-p]\rightarrow [n-k]$ in $\fint$ such that $f_{1} + f_{2} = f$ with $k = f(p)$ (see \S\ref{subsection: Notations and conventions}). By the hypothesis on $f$, we have that $0 < k < n$ as well. Now
$$
\mu_{p,m-p}(X(f)(x)) = (X(f_{1})\otimes_{S} X(f_{2}))\mu_{k,n-k}(x) = 0
$$
and the result follows.
\end{proof}

\begin{Lem}\label{lemma: join and sum interaction}
Let $f: [k]\rightarrow [p]$ and $g: [l]\rightarrow [q]$ be morphisms in $\asimp$. Then
$$
\delta_{p+2}([0]\star f\star g\star [0]) = ([0]\star f\star [0] + [0]\star g\star [0])\delta_{k+2}
$$
\end{Lem}
\begin{proof}
Clearly the morphisms on both sides of the equation preserve the endpoints. Evaluating either side in $0 < i < k + l + 3$, we obtain
$$
\begin{cases}
f(i - 1) + 1 & \text{if }i\leq k+1\\
g(i - k - 2) + p + 3 & \text{if }i\geq k+2
\end{cases}
$$
\end{proof}

\begin{Prop}\label{proposition: augmented simp. cat. assoc. to Frob. temp. obj.}
Let $(X,S)$ be a templicial $k$-module with Frobenius structure $Z$. Then $\mathcal{K}(X)$ of Construction \ref{construction: augmented simp. cat. assoc. to Frob. temp. obj.} is a well-defined $S^{+}\Mod(k)$-category.
\end{Prop}
\begin{proof}
Let $f: [m]\rightarrow [n]$ be a morphism in $\asimp$. Then $[0]\star f\star [0]$ satisfies the hypothesis of Lemma \ref{lemma: narrow morphism induces map on joint kernels} and thus $X([0]\star f\star [0]): X_{n+2}(a,b)\rightarrow X_{m+2}(a,b)$ indeed restricts to a map $\mathcal{K}(X)_{n}(a,b)\rightarrow \mathcal{K}(X)_{m}(a,b)$ for all $a,b\in S$. It is clear that this is functorial so that we obtain a well-defined augmented simplicial $k$-module $\mathcal{K}(X)(a,b)$. Let $p,q\geq -1$ and set $n = p + q + 3$ and consider the quiver map
$$
d_{p+2}Z^{p+2,q+2}: X_{p+2}\otimes_{S} X_{q+2}\rightarrow X_{p+q+3}
$$
For all $0 < k < n$, we have
\begin{align*}
&\mu_{k,n-k}d_{p+2}Z^{p+2,q+2} =
\begin{cases}
(d_{p+2}\otimes_{S} \id_{X_{n-k}})\mu_{k+1,n-k}Z^{p+2,q+2} & \text{if }p+2\leq k\\
(\id_{X_{k}}\otimes_{S} d_{p+2-k})\mu_{k,n-k+1}Z^{p+2,q+2} & \text{if }p+2 > k
\end{cases}\\
&=
\begin{cases}
(d_{p+2}Z^{p+2,k-p-1}\otimes_{S} \id_{X_{n-k}})(\id_{X_{p+2}}\otimes_{S} \mu_{k-p-1,n-k})& \text{if }p+2\leq k\\
((\id_{X_{k}}\otimes_{S} d_{p+2-k}Z^{p+2-k,q+2})(\mu_{k,p+2-k}\otimes_{S} \id_{X_{q+2}})& \text{if }p+2 > k
\end{cases} 
\end{align*}
which implies that $d_{p+2}Z^{p+2,q+2}$ restricts to a map
$$
m_{p,q}: \mathcal{K}(X)_{p}(a,b)\otimes \mathcal{K}(X)_{q}(b,c)\rightarrow \mathcal{K}(X)_{p+q+1}(a,b,c)
$$
for all $a,b,c\in S$. Given morphisms $f: [k]\rightarrow [p]$ and $g: [l]\rightarrow [q]$ in $\asimp$, Lemma \ref{lemma: join and sum interaction} implies that
$$
X([0]\star f\star g\star [0])d_{p+2}Z^{p+2,q+2} = d_{k+2}Z^{k+2,l+2}\left(X([0]\star f\star [0])\otimes_{S} X([0]\star g\star [0])\right)
$$
It follows that the maps $(m_{p,q})_{p,q\geq -1}$ define a morphism of augmented simplicial modules
$$
m: \mathcal{K}(X)(a,b)\star \mathcal{K}(X)(b,c)\rightarrow \mathcal{K}(X)(a,c)
$$

It remains to show that $m$ is associative and unital with respect to $s_{0}(a)$ for $a\in S$. For this it suffices to note that for all $p,q,r\geq -1$:
\begin{align*}
&d_{p+q+3}Z^{p+q+3,r+2}(d_{p+2}Z^{p+2,q+2}\otimes_{S} \id_{X_{r+2}})\\
&\quad = d_{p+q+3}d_{p+2}Z^{p+q+4,r+2}(Z^{p+2,q+2}\otimes_{S} \id_{X_{r+2}})\\
&\quad = d_{p+2}d_{p+q+4}Z^{p+2,q+r+4}(\id_{X_{p+2}}\otimes_{S} Z^{q+2,r+2})\\
&\quad = d_{p+2}Z^{p+2,q+r+3}(\id_{X_{p+2}}\otimes_{S} d_{q+2}Z^{q+2,r+2})
\end{align*}
\begin{align*}
d_{p+2}Z^{p+2,1}(\id_{X_{p+2}}\otimes_{S} s_{0}\epsilon^{-1}) = d_{p+2}s_{p+2}Z^{p+2,0}(\id_{X_{p+2}}\otimes_{S} \epsilon^{-1}) = \id_{X_{p+2}}\\
d_{1}Z^{1,p+2}(s_{0}\epsilon^{-1}\otimes_{S} \id_{X_{p+2}}) = d_{1}s_{0}Z^{0,p+2}(\epsilon^{-1}\otimes_{S} \id_{X_{p+2}}) = \id_{X_{p+2}}
\end{align*}
where $\epsilon: X_{0}\xrightarrow{\sim} k_{S}$ is the counit of $X$ and $k_{S}$ is the monoidal unit of $k\Quiv_{S}$.
\end{proof}

\begin{Prop}\label{proposition: joint kernel functor}
The assignment $(X,Z)\mapsto \mathcal{K}(X)$ of Proposition  \ref{proposition: augmented simp. cat. assoc. to Frob. temp. obj.} extends to a functor
$$
\mathcal{K}: \Fs\Mod(k)\rightarrow k\Cat_{\Delta_{+}}
$$
\end{Prop}
\begin{proof}
Take a Frobenius templicial map $\alpha: X\rightarrow Y$ with vertex map $f: S\rightarrow T$. Consider the underlying natural transformation $\alpha: f_{!}X\rightarrow Y$ and its adjoint $\alpha': X\rightarrow f^{*}Y$. Then $\alpha'$ clearly restricts to a natural transformation
$$
\mathcal{K}(\alpha): \mathcal{K}(X)\rightarrow f^{*}\mathcal{K}(Y)
$$
between functors $\asimp^{op}\rightarrow k\Quiv_{S}$, which can equivalently be considered as a morphism of $S^{+}\Mod(k)$-enriched quivers. It then follows from the compatibilty of $\alpha$ with the Frobenius structures that $\mathcal{K}(\alpha)$ is an $S^{+}\Mod(k)$-enriched functor.

It immediately follows from the definitions that this defines a functor
$$
\mathcal{K}: \Fs\Mod(k)\rightarrow k\Cat_{\Delta_{+}}
$$
\end{proof}

We finish the section by showing Theorem \ref{theorem: augmented simp. cat. and Frob. temp. obj. equiv.}.

\begin{Prop}\label{proposition: tensor algebra essentially surjective}
Let $X$ be a Frobenius templicial $k$-module. The quiver isomorphisms $(\epsilon_{X_{n}})_{n\geq 0}$ of Proposition \ref{proposition: unit and counit of tensor algebra joint kernel adjunction} define an isomorphism
$$
\epsilon_{X}: \mathcal{T}\mathcal{K}(X)\xrightarrow{\sim} X
$$
of Frobenius templicial $k$-modules which is natural in $X$.
\end{Prop}
\begin{proof}
Let $T = \{0 = t_{0} < t_{1} < t_{2} < ... < t_{k} = n\}$ be a necklace. Writing $m$ for the composition law of $\mathcal{K}(X)$ (see Construction \ref{construction: augmented simp. cat. assoc. to Frob. temp. obj.}), note that
$$
m_{T^{c}_{1},...,T^{c}_{k}} = X(\delta_{T})Z^{T}\vert_{\mathcal{K}(X)^{\otimes}_{T}}: \mathcal{K}(X)_{T^{c}_{1}}\otimes_{S} ...\otimes_{S} \mathcal{K}(X)_{T^{c}_{k}}\rightarrow \mathcal{K}(X)_{[n]\setminus T}
$$
where we denoted $\delta_{T} = \delta_{t_{k-1}}\cdots \delta_{t_{2}}\delta_{t_{1}}$. Now take a morphism $f: [m]\rightarrow [n]$ in $\fint$ with a necklace $(T,n)$ and set $U = f^{-1}(T) = \{0 = u_{0} < u_{1} < ... < u_{l} = m\}$ and let $(T_{1},...,T_{l})$ be the splitting of $T$ over $f(U)$. Further, $f = f_{1} + ... + f_{l}$ for some unique $f_{i}: [u_{i} - u_{i-1}]\rightarrow [f(u_{i}) - f(u_{i-1})]$ in $\fint$. In fact, we have the following equality of morphisms in $\asimp$:
$$
f_{i} = \delta_{T_{i}}([0]\star f\vert_{U^{c}_{i}}\star [0])
$$
where we identified $U^{c}_{i}$ with $[u_{i}-u_{i-1}-2]$. Hence, we find that
\begin{align*}
&Z^{U}\vert_{\mathcal{K}(X)^{\otimes}_{U}}\circ \mathcal{K}(X)^{\otimes}(f)_{T}\\
&\quad = Z^{U}\left(\mathcal{K}(X)(f\vert_{U^{c}_{1}})X(\delta_{T_{1}})Z^{T_{1}}\otimes_{S} ...\otimes_{S} \mathcal{K}(X)(f\vert_{U^{c}_{l}})X(\delta_{T_{l}})Z^{T_{l}}\right)\vert_{\mathcal{K}(X)^{\otimes}_{T}}\\
&\quad = Z^{U}(X(f_{1})Z^{T_{1}}\otimes_{S} ...\otimes_{S} X(f_{l})Z^{T_{l}})\vert_{\mathcal{K}(X)^{\otimes}_{T}}\\
&\quad = X(f)Z^{f(U)}(Z^{T_{1}}\otimes_{S} ...\otimes_{S} Z^{T_{l}})\vert_{\mathcal{K}(X)^{\otimes}_{T}} = X(f)\circ Z^{T}\vert_{\mathcal{K}(X)^{\otimes}_{T}}
\end{align*}
where we used the associativity of $Z$ in the last equality. We have thus shown that $\epsilon_{X}: \mathcal{T}\mathcal{K}(X)\rightarrow X$ is a natural transformation between functors $\fint^{op}\rightarrow k\Quiv_{S}$.

Next, it follows from Proposition \ref{proposition: higher mu-Z compatibility} that for all necklaces $(T,n)$ and $0 < k < n$:
$$
\mu_{k,n-k}Z^{T}\vert_{\mathcal{K}(X)^{\otimes}_{T}} =
\begin{cases}
Z^{T_{1}}\vert_{\mathcal{K}(X)^{\otimes}_{T_{1}}}\otimes_{S} Z^{T_{2}}\vert_{\mathcal{K}(X)^{\otimes}_{T_{2}}} & \text{if }k\in T\\
0 & \text{if }k\not\in T
\end{cases}
$$
where $(T_{1},T_{2})$ is the splitting of $T$ over $\{0 < k < n\}$. From this is it easy to see that $\epsilon_{X}$ is respects the comultiplications of $\mathcal{T}\mathcal{K}(X)$ and $X$ as well. Then it is clear from the definitions that $\epsilon_{X}$ is in fact a Frobenius templicial map.

Finally, the naturality of $\epsilon_{X}$ in $X$ quickly follows from the definitions of $\mathcal{T}$ and $\mathcal{K}$, and diagram \eqref{diagram: Frob. temp. map}.
\end{proof}

\begin{Prop}\label{proposition: tensor algebra kernel adjunction unit}
Let $\mathcal{C}$ be a small $S^{+}\Mod(k)$-category. Then the quiver morphisms $(\eta_{\mathcal{C}_{n}})_{n\geq -1}$ of Proposition \ref{proposition: unit and counit of tensor algebra joint kernel adjunction} define an isomorphism in $k\Cat_{\Delta_{+}}$:
$$
\eta: \mathcal{C}\rightarrow \mathcal{K}\mathcal{T}(\mathcal{C})
$$
that is natural in $\mathcal{C}$.
\end{Prop}
\begin{proof}
If $f: [m]\rightarrow [n]$ is a morphism in $\asimp$, then $g = [0]\star f\star [0]: [m+2]\rightarrow [n+2]$ belongs to $\fint$ and $g^{-1}(\{0 < n+2\}) = \{0 < m+2\}$. Thus $\mathcal{C}^{\otimes}(f)_{\{0 < n+2\}} = \mathcal{C}(f)$. It follows that the quiver maps $(\eta_{\mathcal{C}_{n}})_{n\geq -1}$ define a map $\eta_{\mathcal{C}}: \mathcal{C}\rightarrow \mathcal{K}\mathcal{T}(\mathcal{C})$ in $S^{+}\Mod(k)\text{-}\Quiv_{S}$ where $S = \Ob(\mathcal{C})$.

To show that $\eta_{\mathcal{C}}$ is also an $S^{+}\Mod(k)$-functor, let $m$ denote the composition law of $\mathcal{C}$. Also, let the quiver map $u: k_{S}\rightarrow \mathcal{C}_{-1}$ represent the identities of $\mathcal{C}$ with $k_{S}$ the monoidal unit of $k\Quiv_{S}$. For $p,q\geq -1$, consider the coface map $\delta_{p+2}: [p+q+3]\rightarrow [p+q+4]$. Then it suffices to note that $m_{p,q}$ is precisely
$$
\mathcal{C}^{\otimes}(\delta_{p+2})_{\{0 < p+2 < p+q+4\}}: \mathcal{C}_{p}\otimes_{S} \mathcal{C}_{q}\rightarrow \mathcal{C}_{p+q+1}
$$
and that the degeneracy map $s_{0}: \mathcal{T}(\mathcal{C})_{0}\rightarrow \mathcal{T}(\mathcal{C})_{1}$ coincides with $u$.

Finally, the naturality of $\eta_{\mathcal{C}}$ follows immediately from the definitions.
\end{proof}

\begin{proof}[Proof of Theorem \ref{theorem: augmented simp. cat. and Frob. temp. obj. equiv.}]
In view of Propositions \ref{proposition: tensor algebra essentially surjective} and \ref{proposition: tensor algebra kernel adjunction unit}, it remains to verify the triangle identities for the unit $\eta$ and the counit $\epsilon$.

Let $(X,Z)$ be a Frobenius templicial module. Then $\mathcal{K}(\epsilon_{X})\circ \eta_{\mathcal{K}(X)} = \id_{\mathcal{K}(X)_{n}}$ follows from the fact that for all $n\geq -1$, $Z^{\{0 < n+2\}}\vert_{\mathcal{K}(X)_{n}}$ is the identity on $\mathcal{K}(X)_{n}$.

Let $\mathcal{C}$ be an $S^{+}\Mod(k)$-category with object set $S = \Ob(\mathcal{C})$. Then to prove $\epsilon_{\mathcal{T}(\mathcal{C})}\circ \mathcal{T}(\eta_{\mathcal{C}}) = \id_{\mathcal{T}(\mathcal{C})}$, it suffices to note that the composite
$$
\mathcal{T}(\mathcal{C})_{n} = \bigoplus_{(T,n)\in \nec}\!\!\!\mathcal{C}_{T^{c}_{1}}\otimes_{S} ...\otimes_{S} \mathcal{C}_{T^{c}_{k}}\hookrightarrow \bigoplus_{(T,n)\in \nec}\!\!\!\mathcal{T}(\mathcal{C})_{t_{1}}\otimes_{S} ...\otimes_{S} \mathcal{T}(\mathcal{C})_{n-t_{k-1}}\xrightarrow{(Z^{T})_{T}} \mathcal{T}(\mathcal{C})_{n}
$$
is the identity for all $n\geq 0$, where $Z$ is the Frobenius structure of $\mathcal{T}(\mathcal{C})$. The latter follows quickly from the definition of $Z$.
\end{proof}

\section{Templicial maps into the templicial dg-nerve}\label{subsection: Templicial maps into the templicial dg-nerve}

The goal of this section is to prove Proposition \ref{proposition: temp. maps into dg-nerve}. Let us first fix a templicial object $(X,S)$ and a dg-category $\mathcal{C}$ over $k$ with object set $S$. We denote the comultiplication of $X$ by $\mu$ and the composition law and identities of $\mathcal{C}$ by $m$ and $u$ respectively. Then define (making use of Notation \ref{notation: complement of a necklace})
\begin{itemize}
\item the set $\mathcal{S}_{1}$ of all collections of morphisms in $k\Quiv_{S}$:
$$
\left(\beta_{n}: X_{n}\rightarrow \mathcal{C}_{n-1}\right)_{n > 0}
$$
\item the set $\mathcal{S}_{2}$ of all collections of morphisms in $k\Quiv_{S}$:
$$
\left(\alpha_{I}: X_{n}\rightarrow \mathcal{C}_{\vert I\vert}\right)_{I\subseteq \{1,...,n-1\}, n > 0}
$$
such that for all $n > 0$, $I\subseteq \{1,...,n-1\}$ and $j\in \{1,...,n-1\}\setminus I$, we have:
\begin{equation}\label{equation: temp. maps into dg-nerve}
\alpha_{I} = \alpha_{\delta^{-1}_{j}(I)}d^{X}_{j} - m(\alpha_{I_{< j}}\otimes_{S} \alpha_{I_{> j}})\mu^{X}_{j,n-j}
\end{equation}
\end{itemize} 

\begin{Lem}\label{lemma: temp. maps into dg-nerve 1}
The following map is a bijection:
$$
\mathcal{S}_{2}\rightarrow \mathcal{S}_{1}: (\alpha_{I})_{I\subseteq \{1,...,n-1\}, n > 0}\mapsto (\alpha_{\{1,...,n-1\}})_{n > 0}
$$
\end{Lem}
\begin{proof}
It follows from equation \eqref{equation: temp. maps into dg-nerve} that a collection $(\alpha_{I})_{I}$ in $\mathcal{S}_{2}$ is completely determined by the morphisms $\alpha_{\{1,...,n-1\}}$ for $n > 0$. Thus the above map is injective. We now show the surjectivity.

For the surjectivity, take $(\beta_{n})_{n > 0}\in \mathcal{S}_{1}$. Set $\alpha_{\{1,...,n-1\}} = \beta_{n}$ for every $n > 0$. Given $I\subsetneq \{1,...,n-1\}$, choose some $j\in \{1,...,n-1\}\setminus I$ and define $\alpha_{I}: X_{n}\rightarrow \mathcal{C}_{\vert I\vert}$ by equation \eqref{equation: temp. maps into dg-nerve}, inductively on $n$. Using the associativity of $m$ and coassociativity of $\mu$, it follows by induction that this doesn't depend on the choice of $j$. Hence, $(\alpha_{I})_{I}$ belongs to the set $\mathcal{S}_{2}$.
\end{proof}

\begin{Lem}\label{lemma: temp. maps into dg-nerve 2}
For any $(\alpha_{I})_{I}\in \mathcal{S}_{2}$ and $(\beta_{n})_{n > 0} = (\alpha_{\{1,...,n-1\}})_{n > 0}$, the following statements are equivalent:
\begin{enumerate}
\item for all $n > 0$ and $I = \{i_{1} < ... < i_{m}\}\subseteq\{1,...,n-1\}$,
$$
\partial\alpha_{I} = \sum_{j=1}^{m}(-1)^{j-1}\alpha_{I\setminus\{i_{j}\}}
$$
\item for all $n > 0$,
$$
\partial\beta_{n} = \sum_{j=1}^{n-1}(-1)^{j-1}\left(\beta_{n-1}d^{X}_{j} - m(\beta_{j}\otimes_{S} \beta_{n-j})\mu^{X}_{j,n-j}\right)
$$
\end{enumerate}
\end{Lem}
\begin{proof}
It is immediate from the definitions that $(1)$ implies $(2)$. Conversely, assume that $(2)$ holds. Let $I = \{i_{1} < ... < i_{m}\}\subseteq \{1,...,n-1\}$ with $n > 0$. We show by induction on $n$ that $\partial\alpha_{I} = \sum_{j=1}^{m}(-1)^{j-1}\alpha_{I\setminus\{i_{j}\}}$. If $I = \{1,...,n-1\}$, this follows directly from $(2)$. Note that this also covers the case $n = 1$. Otherwise, choose $k\in \{1,...,n-1\}\setminus I$ and let $p = \vert I_{< k}\vert$. Then:
\begin{align*}
\partial\alpha_{I} &= \partial\alpha_{\delta^{-1}_{k}(I)}d^{X}_{k} - \partial m(\alpha_{I_{< k}}\otimes_{S} \alpha_{I_{> k}})\mu^{X}_{k,n-k}\\
&= \sum_{j=1}^{p}(-1)^{j-1}\alpha_{\delta^{-1}_{k}(I)\setminus\{i_{j}\}}d^{X}_{k} + \sum_{j=p+1}^{m}(-1)^{j-1}\alpha_{\delta^{-1}_{k}(I)\setminus\{i_{j}-1\}}d^{X}_{k}\\
&\hphantom{=} -  m(\partial\alpha_{I_{< k}}\otimes_{S} \alpha_{I_{> k}})\mu^{X}_{k,n-k} - (-1)^{p}m(\alpha_{I_{< k}}\otimes_{S} \partial\alpha_{I_{> k}})\mu^{X}_{k,n-k}\\
&= \sum_{j=1}^{p}(-1)^{j-1}\left(\alpha_{\delta^{-1}_{k}(I\setminus\{i_{j}\})}d^{X}_{j} - m(\alpha_{I_{< k}\setminus \{i_{j}\}}\otimes_{S} \alpha_{I_{> k}})\mu^{X}_{k,n-k}\right)\\
&\hphantom{=} + \sum_{j=p+1}^{m}(-1)^{j-1}\left(\alpha_{\delta^{-1}_{k}(I\setminus\{i_{j}\})}d^{X}_{j} - m(\alpha_{I_{< k}}\otimes_{S} \alpha_{I_{> k}\setminus\{i_{j}-k\}})\mu^{X}_{k,n-k}\right)\\
&= \sum_{j=1}^{m}(-1)^{j-1}\alpha_{I\setminus \{i_{j}\}}
\end{align*}
\end{proof}

\begin{Lem}\label{lemma: temp. maps into dg-nerve 3}
For any $(\alpha_{I})_{I}\in \mathcal{S}_{2}$ and $(\beta_{n})_{n > 0} = (\alpha_{\{1,...,n-1\}})_{n > 0}$, the following statements are equivalent:
\begin{enumerate}
\item for all $0\leq i\leq n$ and $I\subseteq \{1,...,n\}$,
$$
\alpha_{I}s^{X}_{i} =
\begin{cases}
u & \text{if }n = 0\\
\alpha_{\sigma_{i}(I)} & \text{if }0 < i < n, \{i,i+1\}\not\subseteq I\\
0 & \text{otherwise}
\end{cases}
$$
\item for all $0\leq i\leq n$,
$$
\beta_{n+1}s_{i}^{X} =
\begin{cases}
u & \text{if }n = 0\\
0 & \text{if }n > 0
\end{cases}
$$
\end{enumerate}
\end{Lem}
\begin{proof}
Note that $(1)$ trivially implies $(2)$. Assume that $(2)$ holds and take $0\leq i\leq n$ and $I\subseteq \{1,...,n\}$. When $I = \{1,...,n\}$, $(1)$ follows directly from $(2)$. In particular, this covers the case $n = 0$. Otherwise, choose $j\in \{1,...,n\}\setminus I$. Without loss of generality, we may assume that $j\leq i$. We proceed by induction on $n > 0$ using equation \eqref{equation: temp. maps into dg-nerve}. Consider the following four cases:
\begin{itemize}
\item If $j < i < n$, then the statements $\{i,i+1\}\subseteq I$, $\{i-1,i\}\subseteq \delta^{-1}_{j}(I)$ and $\{i,i+1\}\subseteq I_{> j}$ are all equivalent and thus
\begin{align*}
&\alpha_{I}s^{X}_{i} = \alpha_{\delta^{-1}_{j}(I)}s^{X}_{i-1}d^{X}_{j} - m(\alpha_{I_{< j}}\otimes_{S} \alpha_{I_{> j}}s^{X}_{i-j})\mu^{X}_{j,n-j}\\
&\quad =
\begin{cases}
\alpha_{\sigma_{i-1}(\delta^{-1}_{j}(I))}d^{X}_{j} - m(\alpha_{I_{< j}}\otimes_{S} \alpha_{\sigma_{i-j}(I_{> j})})\mu^{X}_{j,n-j} = \alpha_{\sigma_{i}(I)} & \text{if }\{i,i+1\}\not\subseteq I\\
0 & \text{if }\{i,i+1\}\subseteq I
\end{cases}
\end{align*}
\item If $j = i < n$, then $\{i,i+1\}\not\subseteq I$ and thus
$$
\alpha_{I}s^{X}_{i} = \alpha_{\delta^{-1}_{i}(I)} - m(\alpha_{I_{< i}}\otimes_{S} \alpha_{I_{> i}}s^{X}_{0})\mu^{X}_{i,n-i} = \alpha_{\delta^{-1}_{i}(I)} = \alpha_{\sigma_{i}(I)}
$$
\item If $j < i = n$,
$$
\alpha_{I}s^{X}_{n} = \alpha_{\delta^{-1}_{j}(I)}s^{X}_{n-1}d^{X}_{j} - m(\alpha_{I_{< j}}\otimes_{S} \alpha_{I_{> j}}s^{X}_{n-j})\mu^{X}_{j,n-j} = 0
$$
\item If $j = i = n$,
$$
\alpha_{I}s^{X}_{n} = \alpha_{\delta^{-1}_{n}(I)} - m(\alpha_{I_{< n}}\otimes_{S} \alpha_{\emptyset}s^{X}_{0})\mu^{X}_{n,0} = \alpha_{\delta^{-1}_{n}(I)} - \alpha_{I_{< n}} = 0
$$
\end{itemize}
Hence we have shown $(1)$.
\end{proof}

\begin{Lem}\label{lemma: Frob. temp. maps into dg-nerve}
Suppose $X$ has a Frobenius structure $Z_{X}$. For any $(\alpha_{I})_{I}\in \mathcal{S}_{2}$ and $(\beta_{n})_{n > 0} = (\alpha_{\{1,...,n\}})_{n > 0}$, the following statements are equivalent:
\begin{enumerate}
\item for all $p,q > 0$ and $I\subseteq \{1,...,p+q\}$,
$$
\alpha_{I}Z^{p,q}_{X} = 0
$$
\item for all $p,q > 0$,
$$
\beta_{p+q}Z^{p,q}_{X} = 0\quad \text{and}\quad \beta_{p+q-1}d^{X}_{p}Z^{p,q}_{X} = m(\beta_{p}\otimes_{S} \beta_{q})
$$
\end{enumerate}
\end{Lem}
\begin{proof}
As before $(1)$ trivially implies $(2)$. Assume that $(2)$ holds and take $p,q > 0$ and $I\subseteq \{1,...,p+q\}$. When $I = \{1,...,p+q\}$ or $I = \{1,...,p+q\}\setminus \{p\}$, $(1)$ follows directly from $(2)$ (using \eqref{equation: temp. maps into dg-nerve}). Thus suppose that there is some $j\in \{1,...,p+q\}\setminus I$ with $j\neq p$. Without loss of generality, we may assume that $p < j$. Then
\begin{align*}
\alpha_{I}Z^{p,q} &= \alpha_{\delta^{-1}_{j}(I)}d^{X}_{j}Z^{p,q} - m(\alpha_{I_{< j}}\otimes_{S} \alpha_{I_{> j}})\mu^{X}_{j,p+q-j}Z^{p,q}_{X}\\
&= \alpha_{\delta^{-1}_{j}(I)}Z^{p,q-1}_{X}(\id_{X_{p}}\otimes_{S} d^{X}_{j-p}) - m(\alpha_{I_{< j}}Z^{p,j-p}_{X}\otimes_{S} \alpha_{I_{> j}})(\id_{X_{p}}\otimes_{S} \mu^{X}_{j-p,p+q-j})
\end{align*}
which is $0$ by induction.
\end{proof}

\begin{proof}[Proof of Proposition \ref{proposition: temp. maps into dg-nerve}]
Note that $f^{*}\mathcal{C}$ is a dg-category with set of objects $S$, so we may replace $\mathcal{C}$ by $f^{*}\mathcal{C}$ and assume that $f$ is the identity on $S$.

We can translate the data of a templicial map $(\alpha,\id_{S}): X\rightarrow N^{dg}_{k}(\mathcal{C})$ as follows. Let $(\alpha_{n}: X_{n}\rightarrow N^{dg}_{k}(\mathcal{C})_{n})_{n > 0}$ be a collection of morphisms in $k\Quiv_{S}$. For any $n > 0$ and $I\subseteq \{0 < n\}^{c}$, consider the following composite:
$$k
\alpha_{I}: X_{n}\xrightarrow{\alpha_{n}} N^{dg}_{k}(\mathcal{C})_{n}\rightarrow \Gamma^{+}(\mathcal{C})_{\{0 < n\}^{c}}\xrightarrow{\pi_{I}} \mathcal{C}_{\vert I\vert}
$$
It follows from the templicial structure of $N^{dg}_{k}(\mathcal{C})$ (Remark \ref{remark: explicit templicial dg-nerve}) that the assignment $(\alpha_{n})_{n > 0}\mapsto (\alpha_{I})_{I\subseteq \{0 < n\}^{c}, n > 0}$ induces a bijection:
\begin{gather*}
\left\lbrace (\alpha_{n}: X_{n}\rightarrow N^{dg}_{k}(\mathcal{C})_{n})_{n > 0}\, \middle\vert\, \forall 0 < j < n: \mu^{N^{dg}_{k}(\mathcal{C})}_{j,n-j}\alpha_{n} = (\alpha_{j}\otimes_{S} \alpha_{n-j})\mu^{X}_{j,n-j}\right\rbrace\\
\simeq\\
\left\lbrace (\alpha_{I}: X_{n}\rightarrow \mathcal{C}_{\vert I\vert})_{\substack{n > 0\\ I\subseteq \{0 < n\}^{c}}}\, \middle\vert\, \forall n, \forall I: \partial\alpha_{I} = \sum_{j=1}^{m}(-1)^{j-1}\alpha_{I\setminus\{i_{j}\}}\right\rbrace
\end{gather*}
where we have denoted $I = \{i_{1} < ... < i_{m}\}\subseteq \{1,...,n-1\}$.

Further, it follows that the morphisms $(\alpha_{n})_{n > 0}$ are compatible with the degeneracy maps if and only if for all $0\leq i\leq n$ and $I\subseteq \{1,...,n\}$,
$$
\alpha_{I}s^{X}_{i} =
\begin{cases}
u & \text{if }n = 0\\
\alpha_{\sigma_{i}(I)} & \text{if }0 < i < n, \{i,i+1\}\not\subseteq I\\
0 & \text{otherwise}
\end{cases}
$$
and $(\alpha_{n})_{n > 0}$ are compatible with the face maps if and only if for all $0 < j < n$ and $I\subseteq \{1,...,n-2\}$:
$$
\alpha_{I}d^{X}_{j} = \alpha_{\delta_{j}(I)} + m(\alpha_{\delta_{j}(I)_{< j}}\otimes_{S} \alpha_{\delta_{j}(I)_{> j-1}})\mu^{X}_{j,n-j}
$$
Hence the result follows from Lemmas \ref{lemma: temp. maps into dg-nerve 2} and \ref{lemma: temp. maps into dg-nerve 3}.
\end{proof}

\begin{proof}[Proof of Corollary \ref{corollary: Frob. temp. maps into dg-nerve}]
Let $\alpha: X\rightarrow N^{dg}_{k}(\mathcal{C})$ be a templicial map and consider the collection $(\alpha_{I}: X_{n}\rightarrow \mathcal{C}_{\vert I\vert})_{n > 0, I\subseteq \{1,...,n\}}$ associated to $\alpha$ as in the proof of Proposition \ref{proposition: temp. maps into dg-nerve}. It follows from the Frobenius structure of $N^{dg}_{k}(\mathcal{C})$ (Remark \ref{remark: explicit templicial dg-nerve}) that $\alpha$ is a Frobenius templicial map if and only if $\alpha_{I}Z^{p,q} = 0$ for all $p,q > 0$ and $I\subseteq \{1,...,p+q\}$. Hence, the result follows from Lemma \ref{lemma: Frob. temp. maps into dg-nerve}.

\end{proof}

\printbibliography

\end{document}